\documentclass[leqno]{siamltex}

\usepackage{amsfonts}
\usepackage{amsmath}

\usepackage{todonotes}

\title{Local well-posedness of the multi-layer shallow water model with free surface\thanks{This 
        work was supported by the French Naval Hydrographic and Oceanographic Service.}}

\author{R.~Monjarret\thanks{Institut de Math\'ematiques de Toulouse, Universit\'e Toulouse III - Paul Sabatier, 118 route de Narbonne, 31062 Toulouse Cedex 9, France ({\tt ronan.monjarret@math.univ-toulouse.fr}).}}

\begin{document}

\maketitle

\begin{abstract}
In this paper, we address the question of the hyperbolicity and the local well-posedness of the multi-layer shallow water model, with free surface, in two dimensions. We first provide a general criterion that proves the symmetrizability of this model, which implies hyperbolicity and local well-posedness in $\mathcal{H}^s(\mathbb{R}^2)$, with $s>2$. Then, we analyze rigorously the eigenstructure associated to this model  and prove a more general criterion of hyperbolicity and local well-posedness, under a particular asymptotic regime and  a weak stratification assumptions of the densities and the velocities. Finally, we consider a new {\em conservative} multi-layer shallow water model, we prove the symmetrizability, the hyperbolicity and the local well-posedness and rely it to the basic multi-layer shallow water model.
\end{abstract}



\pagestyle{myheadings}
\thispagestyle{plain}
\markboth{R. Monjarret}{Local well-posedness of the multi-layer shallow water model with free surface}

\section{Introduction}

We consider $n$ immiscible, homogeneous, inviscid and incompressible superposed fluids, with no surface tension and under the influence of gravity and the {\em Coriolis} forces; the pressure is assumed to be hydrostatic: Constant at the interface liquid/air ({\em i.e.} the free surface) and continuous at the interfaces liquid/liquid ({\em i.e.} the internal surfaces). Moreover, the shallow water assumption is considered in each fluid layer: There exist vertical and horizontal characteristic lengths, for each fluid, and the vertical one is assumed much smaller than the horizontal one.

\noindent
For more details on the formal derivation of these equations, see \cite{desaint1871theorie}, \cite{pedlosky1982geophysical}, \cite{gill1982atmosphere} (the single-layer model), \cite{long1956long} (the two-layer model with rigid lid), \cite{schijf1953theoretical}, \cite{ovsyannikov1979two}, \cite{liska1997analysis} (the two-layer model with free surface). In the {\em curl}-free case, these models have been obtained rigorously with an asymptotic model of the three-dimensional {\em Euler} equations, under the shallow water assumption, in \cite{alvarez2007nash} for the single-layer model with free surface and in \cite{duchene2010asymptotic} for the two-layer one. This has been obtained in \cite{Castrowellposed2014} for the single-layer case and without assumption on the vorticity.

\noindent
Unlike the two-layer model --- see \cite{monjarret2014local} --- the analysis of the hyperbolicity of the multi-layer model, with $n \ge 3$, cannot be performed explicitly. Very few results have been proved concerning the general multi-layer model. They are in particular cases: \cite{stewart2012multilayer} and \cite{castro2010hyperbolicity} in the three-layer case; \cite{audusse2005multilayer} in the very particular case $\rho_1=\ldots=\rho_n$; \cite{audusse2011approximation} and \cite{audusse2011multilayer}, where the interfaces between layers have no physical meaning. In the general case, it was proved only the local well-posedness of the model, in one dimension, under conditions of weak-stratification in density and velocity (see \cite{duchene2013note}). Though, there is no  explicit estimate of this stratification, nor asymptotic one: we know there exists conditions such that the multi-layer model with free surface is locally well-posed but we do not know the characterization of these conditions.

\noindent
The first aim of this paper is to obtain criteria of symmetrizability and hyperbolicity of the multi-layer shallow water model, in order to insure the local well-posedness of the associated {\em Cauchy} problem. The second aim is to characterize the eigenstructure of the space-differential operator, associated with the model, for the treatment of a well-posed open boundary problem with {\em characteristic} variables --- see full proof in \cite{browning1982initialization}, for the single-layer case. The third aim is to prove the local well-posedness of the new conservative model, and characterize its eigenstructure.

\noindent
The main result of this paper is, under weak density stratification and weak velocity stratification, we obtained an asymptotic expansion of all the eigenvalues associated with the multi-layer model. The interpretation of these expressions is really interesting:
\begin{itemize}
\item the eigenvalues, related to the free surface, are asymptotically the same as the single-layer model.
\item the eigenvalues, corresponding to an internal surface $i \in [\![1,n-1]\!]$, are asymptotically as the internal eigenvalues of a two-layer model, where the upper layer would be all the layers, directly above the interface $i$, where the corresponding interface has a density-gap smaller than the density-gap of the interface $i$, and the lower layer would be all the layers, directly below the interface $i$, where the corresponding interface has a density-gap smaller than the density-gap of the interface $i$.
\end{itemize}

\noindent
{\em Outline:} In section $1$, the model is introduced. In section $2$, useful definitions are reminded and a sufficient condition of hyperbolicity and local well-posedness in $\mathcal{H}^s(\mathbb{R}^2)$, is given. In section $3$, the hyperbolicity of the model is studied in particular cases. In sections $4$ and $5$, the asymptotic expansion of all the eigenvalues and all the eigenvecotrs is performed, in order to deduce a new criterion of local well-posedness in $\mathcal{H}^s(\mathbb{R}^2)$, which will be interpreted as the condition obtained in a two-layer model, and compared to the one proved in section $2$. Finally, in the last section, after discussing the conservative quantities of the model and reminding the definition of the horizontal vorticity, a new model is introduced: Benefits of this model are explained, local well-posedness, in $\mathcal{H}^s(\mathbb{R}^2)$, is proved and links, with the multi-layer shallow water model, are fully justified.

\subsection{Governing equations}

\noindent
Let us introduce $\rho_i$ the constant density of the $i^{\mathrm{th}}$ fluid layer, $i \in [\![1,n]\!]$,  $h_i(t,X)$ its height and $\boldsymbol{u_i}(t,X):={}^{\top} (u_i(t,X),v_i(t,X))$ its depth-averaged horizontal velocity, where $t$ denotes the time and $X:=(x,y)$ the horizontal cartesian coordinates, as drawn in figure \ref{1figuremultilayer}.

\noindent
The governing equations of the multi-layer shallow water model with free surface, in two dimensions, is given by a system of $3n$ partial differential equations of $1^{\mathrm{st}}$ order. For all $i \in [\![1,n]\!]$, there is the mass-conservation of the layer $i$:
\begin{equation}
\frac{\partial h_i}{\partial t} + {\bf \nabla} {\bf \cdot} (h_i \boldsymbol{u_i}) = 0,
\label{1massconservation}
\end{equation}
whereas equations on the momentum of the layer $i$:
\begin{equation}
\frac{\partial \boldsymbol{u_i}}{\partial t}+(\boldsymbol{u_i} {\bf \cdot} {\bf \nabla}) \boldsymbol{u_i} + {\bf \nabla} P_i - f \boldsymbol{u_i}^{\bot} = 0,
\label{1momentumconservation}
\end{equation}
where $P_i:=g\big( b+\sum_{k=1}^n \alpha_{i,k} h_k \big)$, $\boldsymbol{u_i}^{\bot}:={}^{\top} (v_i,-u_i)$, $g$ is the gravitational acceleration, $b$ is the bottom topography, $f$ is the {\em Coriolis} parameter and

\begin{equation}
 \alpha_{i,k}=\left\{
 \begin{array}{ll}
 \frac{\rho_k}{\rho_i}, & k < i, \\[4pt]
 1, & k \ge i.
 \end{array}
 \right.
 \label{1alphaik}
\end{equation}

\noindent
A useful notation is introduced, for all $i \in [\![1,n-1]\!]$, $\gamma_i$ is defined by:
\begin{equation}
\gamma_i := \frac{\rho_i}{\rho_{i+1}},
\end{equation}

\noindent
and is the density ratio between the layer $i$ and the layer $i+1$. Then, if $k \le i-1$, the ratio $\frac{\rho_k}{\rho_i}$ is equal to
\begin{equation}
\frac{\rho_k}{\rho_i}=\prod_{j=k}^{i-1} \gamma_j,
\label{1profgamma}
\end{equation}
the vector $\boldsymbol{\gamma}$ will denote ${}^{\top} (\gamma_1,\ldots,\gamma_{n-1}) \in \mathbb{R}_+^{n-1}$ and the vector ${\bf h}$ will be equal to ${}^{\top} (h_1,\ldots,h_n) \in \mathbb{R}_+^n$.

\noindent
{\em Remark:} At this point, no assumption is made over the range of $\rho_i$ and an interesting consequence of the expansion, made in this paper, will be to verify the {\em Rayleigh-Taylor} stability ({\em i.e.} $ \rho_n > \rho_{n-1} > \ldots > \rho_1 > 0$).

\begin{figure}[ht]
\centering
\includegraphics[width=8cm,height=60mm]{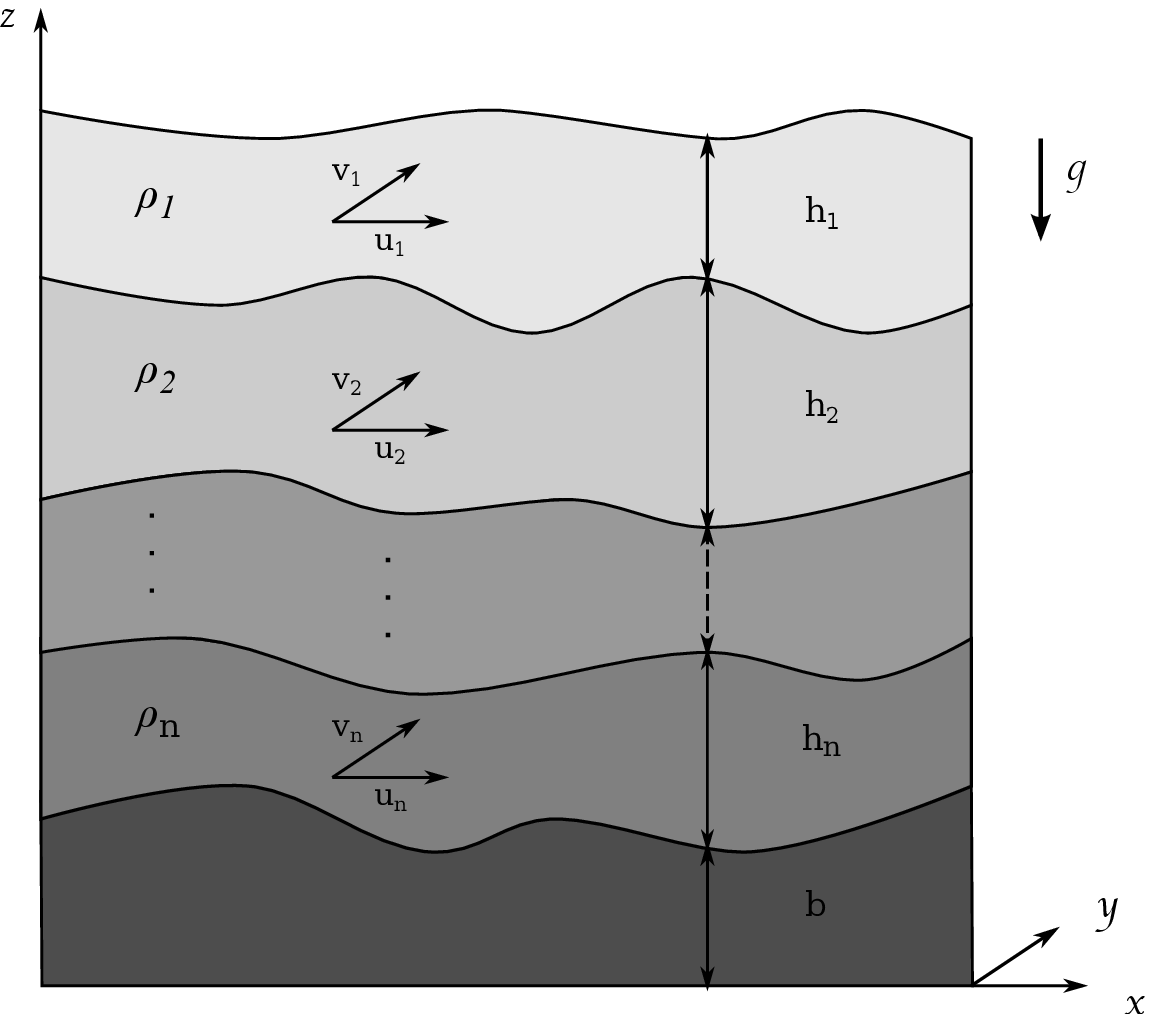}
\caption{Configuration of the multi-layer shallow water model with free surface}
\label{1figuremultilayer}
\end{figure}

\noindent
Moreover, with $i \in [\![1,n]\!]$ and $\mathsf{M}$ a $n \times n$ matrix, we denote by $\mathrm{C}_i\left(\mathsf{M} \right)$ and $\mathrm{L}_i\left(\mathsf{M} \right)$ respectively the $i^{\mathrm{th}}$ column and the $i^{\mathrm{th}}$ line of $\mathsf{M}$. We will denote the total depth by 
\begin{equation}
H:=\sum_{i=1}^n h_i,
\label{1H}
\end{equation}
and the average velocity in each direction by
\begin{equation}
\left\{
\begin{array}{l}
\bar{u}:=\frac{1}{H}\sum_{i=1}^n h_i u_i,\\[4pt]
\bar{v}:=\frac{1}{H}\sum_{i=1}^n h_i v_i.
\end{array}
\right.
\label{1meanvelocity}
\end{equation}


\noindent
In order to get rid of the constant $g$, in the following analysis, we proceed the following rescaling:
\begin{equation}
\forall i \in [\![1,n-1]\!],\ 
\left\{
\begin{array}{l}
\hat{h}_i \leftarrow g h_i,\\[4pt]
\hat{u}_i \leftarrow u_i,\\[4pt]
\hat{v}_i \leftarrow v_i,
\end{array}
\right.
\label{1rescal}
\end{equation}
and in order to simplify the notations, $\hat{ }\ $ will be removed. Then, with the vector
\begin{equation}
{\bf u}:={}^{\top} (h_1,\ldots,h_n,u_1,\ldots,u_n,v_1,\ldots,v_n),
\label{1defu}
\end{equation}
the $1^{\mathrm{st}}$ order quasi-linear partial differential equations system (\ref{1massconservation}-\ref{1momentumconservation}) can be written as
\begin{equation}
\frac{\partial {\bf u}}{\partial t} + \mathsf{A}_x({\bf u},\boldsymbol{\gamma}) \frac{\partial {\bf u}}{\partial x} + \mathsf{A}_y ({\bf u},\boldsymbol{\gamma}) \frac{\partial {\bf u}}{\partial y} + {\bf b}({\bf u})=0,
\label{1systemmultilayer}
\end{equation}
where the $3n \times 3n$ block matrices $\mathsf{A}_x({\bf u},\boldsymbol{\gamma})$, $\mathsf{A}_y({\bf u},\boldsymbol{\gamma})$ and the vector ${\bf b}({\bf u}) \in \mathbb{R}^{3n}$ are defined by
\begin{equation}
\mathsf{A}_x({\bf u},\boldsymbol{\gamma}) :=  \left[\begin{array}{c|c|c}
 \begin{array}{c} \mathsf{V}_x  \end{array}
 & \begin{array}{c} \mathsf{H} \end{array}
 & \begin{array}{c} \mathsf{0} \end{array}\\
 \hline
 \begin{array}{c} \mathsf{\Gamma} \end{array}
 & \begin{array}{c} \mathsf{V}_x \end{array}
 & \begin{array}{c} \mathsf{0} \end{array}\\
 \hline
 \begin{array}{c} \mathsf{0} \end{array}
 & \begin{array}{c} \mathsf{0} \end{array}
 & \begin{array}{c} \mathsf{V}_x \end{array}\\
\end{array}\right],
\mathsf{A}_y({\bf u},\boldsymbol{\gamma}) :=  \left[\begin{array}{c|c|c}
 \begin{array}{c} \mathsf{V}_y  \end{array}
 & \begin{array}{c} \mathsf{0} \end{array}
 & \begin{array}{c} \mathsf{H} \end{array}\\
 \hline
 \begin{array}{c} \mathsf{0} \end{array}
 & \begin{array}{c} \mathsf{V}_y \end{array}
 & \begin{array}{c} \mathsf{0} \end{array}\\
 \hline
 \begin{array}{c} \mathsf{\Gamma} \end{array}
 & \begin{array}{c} \mathsf{0} \end{array}
 & \begin{array}{c} \mathsf{V}_y \end{array}\\
\end{array}\right],
\label{1Axy}
\end{equation}
\begin{equation}
{\bf b}({\bf u}):= {}^{\top} \left(0,\ldots,0,-f v_1 +  \frac{\partial b}{\partial x},\ldots,-f v_n + \frac{\partial b}{\partial x}, f u_1 + \frac{\partial b}{\partial y},f u_n +  \frac{\partial b}{\partial y}\right),
\label{1bottom}
\end{equation}
with the $n \times n$ block matrices
\begin{equation}
\left\{
\begin{array}{cl}
\mathsf{V}_x&:=\mathrm{diag}[u_i]_{i \in [\![1,n]\!]},\\[4pt]
\mathsf{V}_y&:=\mathrm{diag}[v_i]_{i \in [\![1,n]\!]},\\[4pt]
\mathsf{H}&:=\mathrm{diag}[h_i]_{i \in [\![1,n]\!]},\\[4pt]
\Gamma&:=[\alpha_{i,k}]_{(i,k)\in[\![1,n]\!]^2},
\end{array}
\right.
\label{1blockAx}
\end{equation}
where $\mathrm{diag}[x_i]_{i \in [\![1,n]\!]}$ is the $n \times n$ diagonal matrix with $(x_1,\ldots,x_n)$ on the diagonal.

\noindent
As it will be reminded in the next subsections, the hyperbolicity of the model is an interesting property to prove the local well-posedness. The study of the hyperbolicity of the model \eqref{1systemmultilayer} is well-known in the case $n=1$: there are $3$ waves, in each direction, which are well-defined if the height remains strictly positive. In the case $n=2$: if $\rho_1 > \rho_2$, the model is never hyperbolic, and if $\rho_1 = \rho_2$ the model is so if and only if $u_2=u_1$, as proved in \cite{audusse2005multilayer}. Moreover, in \cite{monjarret2014local}, an exact characterization of the domain of hyperbolicity of the model \eqref{1systemmultilayer} was proved: unlike the one-dimensional model, the hyperbolicity of the two-dimensional model is verified if the shear velocity $| u_2 - u_1|^2 + | v_2 -v_1 |^2$ is bounded by a positive parameter, $F_{crit}^{-\ 2}$, depending only on $\gamma_1=\frac{\rho_1}{\rho_2}$, $h_1$ and $h_2$. In the {\em Boussinesq} approximation ({\em i.e.} $1-\gamma_1 \ll 1$), the asymptotic expansion of this parameter is:
\begin{equation}
F_{crit}^{-\ 2} = (h_1+h_2)(1-\gamma_1) + o(1-\gamma_1).
\end{equation}

\noindent
This condition, also explained in \cite{castro2011numerical}, \cite{duchene2013rigid}, \cite{stewart2012multilayer} and \cite{vreugdenhil1979two}, can be interpreted as a physical instability condition --- also known as {\em Kelvin-Helmoltz} stability --- and if it is not verified, the equations have exponential growing solutions. In the general case $n \ge 3$, apart from the formal study in \cite{frings2012adaptive} and the particular case in \cite{chumakova2009stability}, there is no result of hyperbolicity. In order to treat this lack of hyperbolicity, several numerical methods have been proposed in \cite{abgrall2009two}, \cite{bouchut2008entropy} and \cite{bouchut2010robust}.

\noindent
{\em Remark:} The multi-layer shallow water model, with free surface, describes fluids such as the ocean: the evolution of the density can be assumed piecewise-constant (which is verified), the horizontal characteristic length is much greater than the vertical one and the pressure can be expected only dependent of the height of fluid. This model is used by the {\em French Naval Hydrographic and Oceanographic Service}, with $40$ layers, to provide the underwater weather forecast in the bay of Biscay, for example.

\subsection{Rotational invariance} As the multi-layer shallow water model with free surface is based on physical partial differential equations, it verifies the so-called rotational invariance: the $3n \times 3n$ matrix
\begin{equation}
\mathsf{A}({\bf u},\boldsymbol{\gamma},\theta):=\cos(\theta)\mathsf{A}_x({\bf u},\boldsymbol{\gamma})+\sin(\theta) \mathsf{A}_y({\bf u},\boldsymbol{\gamma}) 
\label{1Atheta}
\end{equation}
depends only on the matrix $\mathsf{A}_x({\bf u},\boldsymbol{\gamma})$ and the parameter $\theta$. Indeed, there is the following relation:
\begin{equation}
 \forall ({\bf u},\boldsymbol{\gamma},\theta) \in \mathbb{R}^{3n} \times \mathbb{R}_+^{*\ n-1} \times [0,2\pi],\ \mathsf{A}({\bf u},\boldsymbol{\gamma},\theta)=\mathsf{P}(\theta)^{\mathsf{-1}} \mathsf{A}_x\left(\mathsf{P}(\theta){\bf u},\boldsymbol{\gamma} \right) \mathsf{P}(\theta),
\label{1rotinv}
\end{equation}
where the $3n \times 3n$ matrix $\mathsf{P}(\theta)$ is defined by
\begin{equation}
\mathsf{P}(\theta) :=  \left[\begin{array}{c|c|c}
 \begin{array}{c} \mathsf{I_n}  \end{array}
 & \begin{array}{c} \mathsf{0} \end{array}
 & \begin{array}{c} \mathsf{0} \end{array}\\
 \hline
 \begin{array}{c} \mathsf{0} \end{array}
 & \begin{array}{c} \cos(\theta) \mathsf{I_n} \end{array}
 & \begin{array}{c} \sin(\theta)\mathsf{I_n} \end{array}\\
 \hline
 \begin{array}{c} \mathsf{0} \end{array}
 & \begin{array}{c} -\sin(\theta)\mathsf{I_n} \end{array}
 & \begin{array}{c} \cos(\theta) \mathsf{I_n} \end{array}\\
\end{array}\right].
\label{1ptheta}
\end{equation}

\noindent
Notice that $\mathsf{P}(\theta)^{\mathsf{-1}}={}^{\top}\mathsf{P}(\theta)$. The equality (\ref{1rotinv}) will permit to simplify the analysis of $\mathsf{A}({\bf u},\boldsymbol{\gamma},\theta)$ to the analysis of $\mathsf{A}_x({\bf u},\boldsymbol{\gamma})$.

\section{Well-posedness of the model: a $\boldsymbol{1^{\mathrm{st}}}$ criterion}\label{1sectionhyp}
In this section, we remind useful criteria of local well-posedness in $\mathcal{L}^2(\mathbb{R}^2)^{3n}$, also called hyperbolicity, and in $\mathcal{H}^s(\mathbb{R}^2)^{3n}$. Connections between each one will be given and a $1^{\mathrm{st}}$ criterion of local well-posedness of the system \eqref{1systemmultilayer} will be deduced.

\subsection{Hyperbolicity}
First, we give the definition of hyperbolicity, then a useful criterion of this property and an important property of hyperbolic problem. We will consider the euclidean space $(\mathcal{L}^2(\mathbb{R}^2)^{3n},\|\cdot\|_{\mathcal{L}^2})$.

\begin{definition}[{\rm Hyperbolicity}]
Let ${\bf u} \in \mathcal{L}^2(\mathbb{R}^2)^{3n}$ and $\boldsymbol{\gamma} \in \mathbb{R}_+^{*\ n-1}$. The system {\em (\ref{1systemmultilayer})} is hyperbolic if and only if
\begin{equation}
\exists\ c>0,\ \forall \theta \in [0,2\pi],\  \sup_{\tau \in \mathbb{R}} \| \exp \left( - i \tau \mathsf{A}({\bf u},\boldsymbol{\gamma},\theta) \right) \|_{\mathcal{L}^2} \le c.
\label{1expAutheta}
\end{equation}
\label{1defcrithypuseful}
\end{definition}

\noindent
A useful criterion of hyperbolicity is in the next proposition:
\begin{proposition}
Let ${\bf u} \in \mathcal{L}^2(\mathbb{R}^2)^{3n}$ and $\boldsymbol{\gamma} \in \mathbb{R}_+^{*\ n-1}$. The system {\em (\ref{1systemmultilayer})} is hyperbolic if and only
\begin{equation}
\forall (X,\theta) \in \mathbb{R}^2 \times [0,2\pi],\ \sigma \left( \mathsf{A}({\bf u}(X),\boldsymbol{\gamma},\theta) \right) \subset \mathbb{R}.
\label{1spectrumAutheta}
\end{equation}
\label{1propcrithypuseful}
\end{proposition}

\begin{proposition}
Let ${\bf u}: \mathbb{R}^2 \rightarrow \mathbb{R}^{3n}$ a constant function. If the system {\em (\ref{1systemmultilayer})} is hyperbolic, then the {\em Cauchy} problem, associated with the linear system
\begin{equation}
\frac{\partial {\bf v}}{\partial t} + \mathsf{A}_x({\bf u}) \frac{\partial {\bf v}}{\partial x} + \mathsf{A}_y ({\bf u}) \frac{\partial {\bf v}}{\partial y}=0,
\label{1linearprobl}
\end{equation}
and the initial data ${\bf v^0} \in \mathcal{L}^2(\mathbb{R}^2)^{3n}$, is well-posed in $\mathcal{L}^2(\mathbb{R}^2)^{3n}$ and the unique solution ${\bf v}$ is such that
\begin{equation}
\left\{
\begin{array}{l}
\forall\ T >0,\ \exists\ c_T > 0,\ \sup_{t \in [0,T]} \|{\bf v}(t)\|_{\mathcal{L}^2} \le c_T \|{\bf v^0}\|_{\mathcal{L}^2},\\[4pt]
{\bf v} \in \mathcal{C}(\mathbb{R}_+;\mathcal{L}^2(\mathbb{R}^2))^{3n}.
\end{array}
\right.
\label{1regularityuhyp}
\end{equation}
\label{1defhyp}
\end{proposition}

\noindent
{\em Remark:} An interesting property of hyperbolic problems is the conservation of this property under $\mathcal{C}^1$ change of variables. More details about the main properties of hyperbolicity in \cite{Serre1996systemes}.

\subsection{Symmetrizability}
In order to prove the local well-posedness of the model (\ref{1systemmultilayer}), in $\mathcal{H}^s(\mathbb{R}^2)^{3n}$, we give below a useful criterion.

\begin{definition}[{\rm Symmetrizability}]
Let ${\bf u} \in \mathcal{H}^s(\mathbb{R}^2)^{3n}$. If there exists a $\mathcal{C}^{\infty}$ mapping $\mathsf{S}: \mathcal{H}^s(\mathbb{R}^2)^{3n} \times [0,2\pi] \rightarrow \mathcal{M}_{3n}(\mathbb{R})$ such that for all $\theta \in [0,2\pi]$,
\begin{enumerate}
\item $\mathsf{S}({\bf u},\boldsymbol{\gamma},\theta)$ is symmetric,
\item $\mathsf{S}({\bf u},\boldsymbol{\gamma},\theta)$ is positive-definite,
\item $\mathsf{S}({\bf u},\boldsymbol{\gamma},\theta) \mathsf{A}({\bf u},\theta)$ is symmetric.
\end{enumerate}
Then, the model {\em (\ref{1systemmultilayer})} is said symmetrizable and the mapping $\mathsf{S}$ is called a symbolic-symmetrizer.
\label{1defsymm}
\end{definition}

\begin{proposition}
Let $s>2$ and ${\bf u^0} \in \mathcal{H}^s(\mathbb{R}^2)^{3n}$. If the model {\em (\ref{1systemmultilayer})} is symmetrizable, then the {\em Cauchy} problem, associated with {\em (\ref{1systemmultilayer})} and initial data ${\bf u^0}$, is locally well-posed in $\mathcal{H}^s(\mathbb{R}^2)^{3n}$. Furthermore, there exists $T>0$ such that the unique solution ${\bf u}$ verifies
\begin{equation}
\left\{
\begin{array}{l}
{\bf u} \in \mathcal{C}^1([0,T] \times \mathbb{R}^2)^{3n},\\[4pt]
{\bf u} \in \mathcal{C}([0,T];\mathcal{H}^s(\mathbb{R}^2))^{3n} \cap \mathcal{C}^1([0,T];\mathcal{H}^{s-1}(\mathbb{R}^2))^{3n}.
\end{array}
\right.
\label{1strongsol}
\end{equation}
\label{1propsymm}
\end{proposition}

\noindent
{\em Remark:} The proof of the last proposition is in \cite{benzoni2007multi}, for instance.

\noindent
In this paper, the model (\ref{1massconservation}--\ref{1momentumconservation}) is expressed with the variables $(h_i,\boldsymbol{u_i})$ with $i \in [\![1,n]\!]$. However, we could have worked with the unknowns $h_i$ and $\boldsymbol{q_i}:=h_i \boldsymbol{u_i}$, as it is well-known this quantities are conservative in the one-dimensional case. However, in the particular case of the multi-layer shallow water model with free surface, it is not true: The multi-layer model, in one space-dimension, is conservative with $(h_i,\boldsymbol{u_i})$ variables and not conservative with $(h_i,\boldsymbol{q_i})$ variables.

\noindent
As it was noticed in \cite{Serre1996systemes}, if the model is conservative and has a total energy, there exists a natural symmetrizer: the hessian of this total energy. In one dimension, the total energy of the model \eqref{1systemmultilayer} is defined, modulo a constant, by
\begin{equation}
e_1({\bf u},\boldsymbol{\gamma}):=\frac{1}{2}\sum_{i=1}^n \alpha_{n,i} h_i \left(u_i^2+h_i\right)+\sum_{i=1}^{n-1}\sum_{j=i+1}^n \alpha_{n,i} h_i h_j.
\label{1energy1}
\end{equation}
As the model (\ref{1massconservation}--\ref{1momentumconservation}), in one space-dimension and variables $(h_i,\boldsymbol{u_i})$, is conservative, it is straightforward the hessian of $e_1$ is a symmetrizer of the one-dimensional model. However, it is not anymore a symmetrizer with the non-conservative variables $(h_i,\boldsymbol{q_i})$. This is another reason the analysis, in this paper, is performed with variables $(h_i,\boldsymbol{u_i})$. Moreover, as the two-dimensional model is not conservative, the symmetrizer $\mathsf{S}$, defined in definition \ref{1defsymm}, is not the hessian of the total energy of the two-dimensional model
\begin{equation}
e_2({\bf u},\boldsymbol{\gamma}):=\frac{1}{2}\sum_{i=1}^n \alpha_{n,i} h_i \left(u_i^2+v_i^2+h_i\right)+\sum_{i=1}^{n-1}\sum_{j=i+1}^n \alpha_{n,i} h_i h_j.
\label{1energy2}
\end{equation}
This is why the symmetrizer is called symbolic: it will depend on $\theta$. If it does not depend on (such as the irrotational model in two dimensions), the symmetrizer is called {\em Friedrichs}-symmetrizer.

\noindent
{\em Remarks:} 1) In all this paper, the parameter $s \in \mathbb{R}$ is assumed such that
\begin{equation}
s > 1+ \frac{d}{2},
\label{1params}
\end{equation}
where $d:=2$ is the space-dimension. 2) The criterion (\ref{1spectrumAutheta}) is a necessary and sufficient condition of hyperbolicity, whereas the symmetrizability is only a sufficient condition of local well-posedness in $\mathcal{H}^s(\mathbb{R}^2)^{3n}$.

\subsection{Connections between hyperbolicity and symmetrizability}
In this subsection, we do not formulate all the connections between these two types of local well-posedness but only the useful ones for this paper.

\begin{proposition}
If the system {\em (\ref{1systemmultilayer})} is symmetrizable, then it is hyperbolic.
\label{1symimpliqhyp}
\end{proposition}
{\em Remark:} This property is obvious in the linear case, with the change of variables ${\bf \tilde{u}}:=\mathsf{S}({\bf u^0},\boldsymbol{\gamma},\theta){\bf u}$. See \cite{benzoni2007multi} and \cite{Serre1996systemes} for more details.

\begin{proposition}
Let $({\bf u^0},\boldsymbol{\gamma}) \in \mathcal{H}^s(\mathbb{R}^2)^{3n} \times \mathbb{R}_+^{*\ n-1}$ such that the model is hyperbolic and for all $\theta \in [0,2\pi]$, the matrix $\mathsf{A}({\bf u^0},\boldsymbol{\gamma},\theta)$ is diagonalizable. Then, the system {\em (\ref{1systemmultilayer})} is symmetrizable and the unique solution verifies the conditions {\em (\ref{1strongsol})}.
\label{1diagimpliqsym}
\end{proposition}
\begin{proof}
Let $\mu \in \sigma(\mathsf{A}({\bf u^0},\boldsymbol{\gamma},\theta))$, we denote $\mathsf{P}^{\mu}({\bf u^0},\boldsymbol{\gamma},\theta)$ the projection onto the $\mu$-eigenspace of $\mathsf{A}({\bf u^0},\boldsymbol{\gamma},\theta)$. One can construct a symbolic-symmetrizer:
\begin{equation}
\mathsf{S_1}({\bf u^0},\boldsymbol{\gamma},\theta):= \sum_{\mu \in \sigma(\mathsf{A}({\bf u^0},\boldsymbol{\gamma},\theta))} {}^{\top} \mathsf{P}^{\mu}({\bf u^0},\boldsymbol{\gamma},\theta) \mathsf{P}^{\mu}({\bf u^0},\boldsymbol{\gamma},\theta).
\label{1symprojdef}
\end{equation}

\noindent
Then, $\mathsf{S_1}({\bf u^0},\boldsymbol{\gamma},\theta)$ verifies conditions of the proposition \ref{1propsymm} because $\mathsf{A}({\bf u^0},\boldsymbol{\gamma},\theta)$ is diagonalizable --- which implies $\mathsf{S_1}({\bf u^0},\boldsymbol{\gamma},\theta)$ induces a scalar product on $\mathbb{R}^{3n}$ --- and the spectrum of $\mathsf{A}({\bf u^0},\boldsymbol{\gamma},\theta)$ is a subset of $\mathbb{R}$. Then, proposition \ref{1propsymm} implies the well-posedness of the system (\ref{1systemmultilayer}), in $\mathcal{H}^s(\mathbb{R}^2)^{3n}$, and there exists $T>0$ such that conditions (\ref{1strongsol}) are verified.\end{proof}

\noindent
To conclude, the analysis of the eigenstructure of $\mathsf{A}({\bf u},\boldsymbol{\gamma},\theta)$ is a crucial point, in order to provide its diagonalizability. Moreover, it provides also the characterization of the {\em Riemann} invariants (see \cite{smoller1983shock} and \cite{toro2009riemann}), which is an important benefit for numerical resolution.

\noindent
{\em Remark:} The proposition \eqref{1diagimpliqsym} was proved in \cite{taylor1996partial}, in the particular case of a strictly hyperbolic model ({\em i.e.} all the eigenvalues are real and distinct).

\subsection{A $\boldsymbol{1^{\mathrm{st}}}$ criterion of local well-posedness}
According to the proposition \ref{1symimpliqhyp}, the symmetrizability implies the hyperbolicity. Then, we give a rough criterion of symmetrizability to insure the well-posedness in $\mathcal{H}^s(\mathbb{R}^2)^{3n}$ and $\mathcal{L}^2(\mathbb{R}^2)^{3n}$.

\begin{theorem}
Let $s>2$ and $({\bf u^0},\boldsymbol{\gamma}) \in  \mathcal{H}^s(\mathbb{R}^2)^{3n} \times ]0,1[^{n-1}$. There exists a sequence $(\delta_i({\bf h},\boldsymbol{\gamma}))_{i \in [\![1,n]\!]} \subset \mathbb{R}_+^*$ such that
\begin{equation}
\forall\ i \in [\![1,n]\!],\ 
\left\{
\begin{array}{l}
\inf_{X \in \mathbb{R}^2} h_i^0(X) > 0,\\[4pt]
\inf_{X \in \mathbb{R}^2} \delta_i({\bf h^0}(X),\boldsymbol{\gamma})-|u_i^0(X) - \bar{u}^0(X)|^2 - |v_i^0(X)-\bar{v}^0(X)|^2>0,
\end{array}
\right.
\label{1condwellposed}
\end{equation}
then, the {\em Cauchy} problem, associated with the system {\em \eqref{1systemmultilayer}} and the initial data ${\bf u^0}$, is hyperbolic, locally well-posed in $\mathcal{H}^s(\mathbb{R}^2)^{3n}$ and the unique solution verifies conditions {\em (\ref{1strongsol})}.
\label{1condwellposedMLSW}
\end{theorem}
\begin{proof}
First, we prove the next lemma.
\begin{lemma}
Let $\boldsymbol{\gamma} \in \mathbb{R}_+^{*\ n-1}$, $\mathcal{S}$ an open subset of $\mathcal{H}^s(\mathbb{R}^2)^{3n}$ and $\mathsf{S}_x({\bf u},\boldsymbol{\gamma})$ be a symmetric matrix such that $\mathsf{S}_x({\bf u},\boldsymbol{\gamma}) \mathsf{A}_x({\bf u},\boldsymbol{\gamma})$ is symmetric. If there exists ${\bf u^0} \in \mathcal{S}$ such that
\begin{equation}
\forall \theta \in [0,2\pi],\ \mathsf{S}_x(P(\theta){\bf u^0},\boldsymbol{\gamma})>0,
\label{1defposSp}
\end{equation}
then the {\em Cauchy} problem, associated with system {\em (\ref{1systemmultilayer})} and initial data ${\bf u^0}$, is hyperbolic, locally well-posed in $\mathcal{H}^s(\mathbb{R}^2)^{3n}$ and the unique solution verifies conditions {\em (\ref{1strongsol})}.
\label{1wellposedrotinv}
\end{lemma}
\begin{proof}
We consider ${\bf u^0} \in \mathcal{S}$ such that
\begin{equation}
\forall (X,\theta) \in \mathbb{R}^2 \times [0,2\pi],\ \mathsf{S}_x(\mathsf{P}(\theta){\bf u^0}(X),\boldsymbol{\gamma}) >0.
\end{equation}
We define the mapping
\begin{equation}
\mathsf{S}: ({\bf u},\boldsymbol{\gamma},\theta) \mapsto \mathsf{P}(\theta)^{\mathsf{-1}} \mathsf{S}_x(\mathsf{P}(\theta) {\bf u},\boldsymbol{\gamma}) \mathsf{P}(\theta).
\label{1Srotin}
\end{equation}
Then, using the rotational invariance (\ref{1rotinv}), the mapping $\mathsf{S}$ verifies assumptions of the definition \ref{1defsymm}, with ${\bf u^0} \in \mathcal{H}^s(\mathbb{R}^2)^{3n}$, and $\mathsf{S}$ is a symbolic-symmetrizer of the system (\ref{1systemmultilayer}).
\end{proof}

\noindent
As it was noticed before, the one-dimensional multi-layer model, with variables $(h_i,u_i)$ is conservative: a natural symmetrizer of this model is the hessian of the total energy $e_1$. The next matrix defines a symbolic-symmetrizer of the two-dimensional model --- using the mapping \eqref{1Srotin} --- and it has been constructed from the {\em Friedrichs}-symmetrizer of the one-dimensional model:
\begin{equation}
\mathsf{S}_x(\mathbf{u},\boldsymbol{\gamma},u_0) =  \left[\begin{array}{c|c|c}
 \begin{array}{c} \mathsf{\Delta} \mathsf{\Gamma}  \end{array}
 & \begin{array}{c} \mathsf{\Delta}  (\mathsf{V}_x-u_0\mathsf{I_n}) \end{array}
 & \begin{array}{c} \mathsf{0} \end{array}\\
 \hline
 \begin{array}{c} \mathsf{\Delta} (\mathsf{V}_x-u_0 \mathsf{I_n}) \end{array}
 & \begin{array}{c} \mathsf{\Delta} \mathsf{H} \end{array}
 & \begin{array}{c} \mathsf{0} \end{array}\\
 \hline
 \begin{array}{c} \mathsf{0} \end{array}
 & \begin{array}{c} \mathsf{0} \end{array}
 & \begin{array}{c} \mathsf{\Delta} \mathsf{H} \end{array}\\
\end{array}\right],
\label{1Sx}
\end{equation}
where $\mathsf{\Delta}:=\mathrm{diag}(\alpha_{n,1},\ldots,\alpha_{n,n-1},1)$ and $u_0 \in \mathbb{R}$ is a parameter, which will be chosen in order to simplify the calculus.

\noindent
{\em Remark:} If $u_0=0$, the matrix
\begin{equation}
\left[\begin{array}{c|c}
 \begin{array}{c} \mathsf{\Delta} \mathsf{\Gamma}  \end{array}
 & \begin{array}{c} \mathsf{\Delta}  (\mathsf{V}_x-u_0\mathsf{I_n}) \end{array}\\
 \hline
 \begin{array}{c} \mathsf{\Delta} (\mathsf{V}_x-u_0 \mathsf{I_n}) \end{array}
 & \begin{array}{c} \mathsf{\Delta} \mathsf{H} \end{array}\\
\end{array}\right],
\label{1Sx1D}
\end{equation}
is exactly the hessian of the total energy $e_1$.

\noindent
Moreover, we introduce the $3n \times 3n$ symmetric matrix $\mathsf{S}_x^0({\bf h},\boldsymbol{\gamma})$ defined by
\begin{equation}
\mathsf{S}_x^0({\bf h},\boldsymbol{\gamma}) :=  \left[\begin{array}{c|c|c}
 \begin{array}{c} \mathsf{\Delta} \mathsf{\Gamma}  \end{array}
 & \begin{array}{c} \mathsf{0} \end{array}
 & \begin{array}{c} \mathsf{0} \end{array}\\
 \hline
 \begin{array}{c} \mathsf{0} \end{array}
 & \begin{array}{c} \mathsf{\Delta} \mathsf{H} \end{array}
 & \begin{array}{c} \mathsf{0} \end{array}\\
 \hline
 \begin{array}{c} \mathsf{0} \end{array}
 & \begin{array}{c} \mathsf{0} \end{array}
 & \begin{array}{c} \mathsf{\Delta} \mathsf{H} \end{array}\\
\end{array}\right],
\label{1Sx0}
\end{equation}
and prove the next lemma.

\begin{lemma}
Let $({\bf h},\boldsymbol{\gamma}) \in \mathbb{R}_+^n \times \mathbb{R}_+^{*\ n-1}$. $\mathsf{S}_x^0({\bf h},\boldsymbol{\gamma})$ is positive-definite if and only if
\begin{equation}
\left\{
\begin{array}{lr}
h_i > 0,& \forall\ i \in [\![1,n]\!],\\[4pt]
1>\gamma_i > 0,&\forall\ i \in [\![1,n-1]\!].
\end{array}
\right.
\label{1Sx0posdef}
\end{equation}
\label{1lemSx0posdef}
\end{lemma}
\begin{proof}
First of all, it is clear $\mathsf{S}_x^0({\bf h},\boldsymbol{\gamma})$ is positive-definite if and only if $\mathsf{\Delta} \mathsf{\Gamma}$ and $\mathsf{\Delta} \mathsf{H}$ are positive-definite. Then, as $\mathsf{\Delta} \mathsf{H} :=\mathrm{diag}(\alpha_{n,i} h_i)$, it is positive-definite if and only if
\begin{equation}
\forall i \in [\![1,n]\!],\ \alpha_{n,i}h_i>0.
\label{1conddefposDeltaH}
\end{equation}
Moreover, using the {\em Sylvester}'s criterion, $\mathsf{\Delta} \mathsf{\Gamma}:=[\alpha_{n,\min(i,j)} ]_{(i,j)\in[\![1,n]\!]^2}$ is positive-definite if and only if all the leading principal minors are strictly positive:
\begin{equation}
\forall k \in [\![1,n]\!],\ m_k:=\det [\alpha_{n,\min(i,j)} ]_{(i,j)\in[\![1,k]\!]^2} >0.
\end{equation}

\noindent
Let $k \in [\![1,n]\!]$, performing the following elementary operations on the columns of $\det(\mathsf{\Delta} \mathsf{\Gamma})$: 
\begin{equation}
\forall\ i \in [\![1,n-1]\!],\ \mathrm{C}_i(\mathsf{\Delta} \mathsf{\Gamma}) \leftarrow \mathrm{C}_i(\mathsf{\Delta} \mathsf{\Gamma}) - \mathrm{C}_{i+1}(\mathsf{\Delta} \mathsf{\Gamma}),
\label{1elementopDeltaGamma}
\end{equation}
and expanding this determinant along the $1^{\mathrm{st}}$ line, we deduce the expression of $m_k$:
\begin{equation}
\left\{
\begin{array}{lr}
m_k=\alpha_{n,1},& \mathrm{if}\ k=1,\\[4pt]
m_k= \alpha_{n,1} \prod_{i=1}^{k-1} (\alpha_{n,i+1}-\alpha_{n,i}), & \mathrm{if}\ k\in[\![2,n]\!].
\end{array}
\right.
\label{1minors}
\end{equation}
Consequently, it is obvious that 
\begin{equation}
\forall k \in [\![1,n]\!],\ m_k >0 \Longleftrightarrow
\left\{
\begin{array}{lr}
\alpha_{n,1}>0, & \\[4pt]
\alpha_{n,k+1}-\alpha_{n,k}>0,& \forall\ k\in[\![2,n]\!].
\end{array}
\right.
\label{1Sylvestercrit}
\end{equation}
As for all $k\in[\![1,n]\!]$, $\alpha_{n,k}:=\frac{\rho_k}{\rho_n}$ and $\rho_n$ is assumed strictly positive, $\mathsf{S}_x^0$ is positive-definite if and only if conditions (\ref{1Sx0posdef}) are verified.
\end{proof}

\noindent
Finally, using lemmata \ref{1wellposedrotinv}--\ref{1lemSx0posdef}, we can prove the theorem \ref{1condwellposedMLSW}. One can check that $\mathsf{S}_x(\mathbf{u},\boldsymbol{\gamma},u_0) $ and $\mathsf{S}_x(\mathbf{u},\boldsymbol{\gamma},u_0) \mathsf{A}_x({\bf u},\boldsymbol{\gamma})$ are unconditionally symmetric:
\begin{equation}
\mathsf{S}_x(\mathbf{u},\boldsymbol{\gamma},u_0)\mathsf{A}_x(\mathbf{u},\boldsymbol{\gamma}) :=  \left[\begin{array}{c|c|c}
 \begin{array}{c} \mathsf{\Delta} \mathsf{\Gamma}(2\mathsf{V}_x-u_0\mathsf{I_n})  \end{array}
 & \begin{array}{c} \mathsf{S}_x^1(\mathbf{u},\boldsymbol{\gamma},u_0) \end{array}
 & \begin{array}{c} \mathsf{0} \end{array}\\
 \hline
 \begin{array}{c} \mathsf{S}_x^1(\mathbf{u},\boldsymbol{\gamma},u_0) \end{array}
 & \begin{array}{c} \mathsf{\Delta} \mathsf{H}(2\mathsf{V}_x-u_0\mathsf{I_n}) \end{array}
 & \begin{array}{c} \mathsf{0} \end{array}\\
 \hline
 \begin{array}{c} \mathsf{0} \end{array}
 & \begin{array}{c} \mathsf{0} \end{array}
 & \begin{array}{c} \mathsf{\Delta} \mathsf{H}\mathsf{V}_x \end{array}\\
\end{array}\right],
\label{1SxAx}
\end{equation}
where $\mathsf{S}_x^1(\mathbf{u},\boldsymbol{\gamma},u_0):= \mathsf{\Delta} (\mathsf{\Gamma} \mathsf{H} +(\mathsf{V}_x-u_0\mathsf{I_n})\mathsf{V}_x)$. As we need to chose a reference velocity $u_0$, we decide to set $u_0:=\bar{u}$, the average velocity. Moreover, if $({\bf u},\boldsymbol{\gamma}) \in \mathcal{H}^s(\mathbb{R}^2) \times \mathbb{R}_+^{*\ n-1}$ are such that
\begin{equation}
\left\{
\begin{array}{lr}
u_i(X)=\bar{u}(X),& \forall X \in \mathbb{R}^2,\ \forall i \in [\![1,n]\!],\\[4pt]
\inf_{X \in \mathbb{R}^2} h_i(X) > 0,& \forall i \in [\![1,n]\!],\\[4pt]
1>\gamma_i > 0,& \forall i \in [\![1,n-1]\!],
\end{array}
\right.
\label{1particularu}
\end{equation}
then, $\mathsf{S}_x(\mathbf{u},\boldsymbol{\gamma},\bar{u})=\mathsf{S}_x^0({\bf h},\boldsymbol{\gamma})$ and, according to the lemma \ref{1lemSx0posdef},
\begin{equation}
\forall X \in \mathbb{R}^2,\ \mathsf{S}_x(\mathbf{u}(X),\boldsymbol{\gamma},\bar{u}(X))>0.
\label{1Sxdefpos1}
\end{equation}

\noindent
Then, if $({\bf u},\boldsymbol{\gamma})$ verifies
\begin{equation}
\left\{
\begin{array}{lr}
\inf_{X \in \mathbb{R}^2} h_i(X) > 0,& \forall\ i \in [\![1,n]\!],\\[4pt]
1>\gamma_i > 0,& \forall\ i \in [\![1,n-1]\!],\\
\end{array}
\right.
\label{1Sx0posdef2}
\end{equation}
as all the eigenvalues of $\mathsf{S}_x(\mathbf{u},\boldsymbol{\gamma},\bar{u})$  depend continuously on the parameter ${\bf u}$, the matrix $\mathsf{S}_x(\mathbf{u},\boldsymbol{\gamma},\bar{u})$ remains positive-definite if $u_i - \bar{u}$ is sufficiently small, for all $i \in [\![1,n]\!]$: this insures the existence of the sequence $\left(\delta_i({\bf h},\boldsymbol{\gamma})\right)_{i \in [\![1,n]\!]} \subset \mathbb{R}_+^*$ such that
\begin{equation}
\forall i \in [\![1,n]\!],\ \inf_{X \in \mathbb{R}^2} \delta_i({\bf h}(X),\boldsymbol{\gamma}) - |u_i(X)-\bar{u}(X)|^2 >0 \Rightarrow \mathsf{S}_x(\mathbf{u}(X),\boldsymbol{\gamma},\bar{u}(X))>0.
\label{1newconditionsym}
\end{equation}
Moreover, these quantities depend only on the parameters of $\mathsf{S}_x^0({\bf h},\boldsymbol{\gamma})$: ${\bf h}$ and $\boldsymbol{\gamma}$. In order to use the lemma \ref{1wellposedrotinv}, we remark that if for all $\theta \in [0,2\pi]$,
\begin{equation}
\inf_{X \in \mathbb{R}^2} \delta_i({\bf h}(X),\boldsymbol{\gamma}) - \left[\cos(\theta)(u_i(X)-\bar{u}(X)) -\sin(\theta)(v_i(X)-\bar{v}(X))\right]^2 >0
\label{1condsymPtheta}
\end{equation}
then,
\begin{equation}
\forall (X,\theta) \in \mathbb{R}^2 \times [0,2\pi],\ \mathsf{S}_x(\mathsf{P}(\theta)\mathbf{u}(X),\boldsymbol{\gamma},\mathsf{P}(\theta)\bar{u}(X))>0.
\end{equation}

\noindent
As this last condition must be verified for all $\theta \in [0,2\pi]$ and
\begin{equation}
\forall (\alpha,\beta) \in \mathbb{R}^2,\ \max_{\theta \in [0,2\pi]} \left[\cos(\theta)\alpha+\sin(\theta)\beta\right]^2 = \alpha^2 + \beta^2,
\label{maxPthetau}
\end{equation}
then, if $({\bf u^0},\boldsymbol{\gamma})$ is such that
\begin{equation}
\inf_{X \in \mathbb{R}^2} \delta_i({\bf h}(X),\boldsymbol{\gamma}) - |u_i(X)-\bar{u}(X)|^2 - |v_i(X)-\bar{v}(X)|^2 >0,
\label{1condsymbsymm2}
\end{equation}
then, $\mathsf{S}_x(\mathsf{P}(\theta)\mathbf{u}(X),\boldsymbol{\gamma},\mathsf{P}(\theta)\bar{u}(X))$ is positive-definite for all $\theta \in [0,2\pi]$.

\noindent
Finally, using the lemma \ref{1wellposedrotinv}, if $({\bf u^0},\boldsymbol{\gamma}) \in \mathcal{H}^s(\mathbb{R}^2) \times ]0,1[^{n-1}$ verifies conditions (\ref{1condwellposed}), the mapping 
\begin{equation}
\mathsf{S}: ({\bf u},\boldsymbol{\gamma},\theta) \mapsto \mathsf{P}(\theta)^{\mathsf{-1}} \mathsf{S}_x(\mathsf{P}(\theta) {\bf u},\boldsymbol{\gamma}) \mathsf{P}(\theta)
\label{1mappingsymbolsym}
\end{equation}
is a symbolic-simmetrizer and, according to the proposition \ref{1propsymm}, the {\em Cauchy} problem, associated with (\ref{1systemmultilayer}) and the initial data ${\bf u^0}$, is hyperbolic, locally well-posed in $\mathcal{H}^s(\mathbb{R}^2)^{3n}$ and the unique solution verifies conditions \eqref{1strongsol}.
\end{proof}

\noindent
To conclude, considering $\boldsymbol{\gamma} \in ]0,1[^{n-1}$ and $s > 2$, we define $\mathcal{S}^s_{\boldsymbol{\gamma}} \subset \mathcal{H}^s(\mathbb{R}^2)^{3n}$, an open subset of initial conditions such that the model (\ref{1systemmultilayer}) is symmetrizable:
\begin{equation}
\mathcal{S}^s_{\gamma}:= \left\{ {\bf u^0} \in \mathcal{H}^s(\mathbb{R}^2)^{3n} / {\bf u^0}\ \mathrm{verifies}\ \mathrm{conditions}\ {\rm \eqref{1condwellposed}}  \right\}.
\label{1Sgamma}
\end{equation}

\noindent
{\em Remark:} The condition of symmetrizability expressed in \cite{duchene2013note}, with the multi-layer shallow water model with free surface in one dimension, is a little different from \eqref{1condwellposed}. Indeed, there is no need of a velocity reference but even if it seems to be a weaker criterion, it is not possible to assure it, as there is no explicit estimations of this criterion.

\subsection{Lower bounds of $\boldsymbol{\delta_i}$} In this subsection, we do not estimate exactly the sequence $(\delta_i({\bf h},\boldsymbol{\gamma}))_{i \in [\![1,n]\!]} \subset \mathbb{R}_+^*$, but a lower bound of each element $\delta_i({\bf h},\boldsymbol{\gamma})$. The proof is based on the next proposition, where $\lambda^{\min}$ and $\lambda^{\max}$ denote respectively the smallest and the largest eigenvalues.
\begin{proposition}
Let $\mathcal{S}_n(\mathbb{R})$ the space-vector of $n \times n$  symmetric matrices, with real coefficients. Then, $\lambda^{\min}: \mathcal{S}_n(\mathbb{R}) \rightarrow \mathbb{R}$ is a concave function and $\lambda^{\max}: \mathcal{S}_n(\mathbb{R}) \rightarrow \mathbb{R}$ is a convex one:
\begin{equation}
\forall (\mathsf{A},\mathsf{B}) \in \mathcal{S}_n(\mathbb{R})^2,\
\left\{
\begin{array}{l}
\lambda^{\min} (\mathsf{A}+\mathsf{B}) \ge \lambda^{\min}(\mathsf{A})+\lambda^{\min}(\mathsf{B}),\\[4pt]
\lambda^{\max} (\mathsf{A}+\mathsf{B}) \le \lambda^{\max}(\mathsf{A})+\lambda^{\max}(\mathsf{B}),
\end{array}
\right.
\label{1ConcConv}
\end{equation}
\label{1proplambdamin}
\end{proposition}

\noindent
Using this last proposition, we can extract conditions which maintain $\mathsf{S}_x({\bf u},\boldsymbol{\gamma},\bar{u})$ positive-definite.
\begin{proposition}
Let $\boldsymbol{\gamma} \in ]0,1[^{n-1}$ and ${\bf h} \in \mathbb{R}_+^{*\ n}$. Then, $\lambda^{\min}(\mathsf{S}_x^0({\bf h},\boldsymbol{\gamma}))>0$ and a lower bound of $\delta_i({\bf h},\boldsymbol{\gamma})$, for every $i \in [\![1,n]\!]$, is
\begin{equation}
\delta_i({\bf h},\boldsymbol{\gamma}) \ge \left(\frac{\lambda^{\min}(\mathsf{S}_x^0({\bf h},\boldsymbol{\gamma}))}{\alpha_{n,i}}\right)^2
\label{1estimdeltai}
\end{equation}
\label{1propestimdeltai}
\end{proposition}
\begin{proof}
We remind that $\left(\delta_i({\bf h},\boldsymbol{\gamma})\right)_{i \in [\![1,n]\!]}$ is the sequence that remains $\mathsf{S}_x({\bf u},\boldsymbol{\gamma},\bar{u})$ positive-definite ({\em i.e.} $\lambda^{\min}\left( \mathsf{S}_x({\bf u},\boldsymbol{\gamma},\bar{u}) \right) >0$). We decompose $\mathsf{S}_x({\bf u},\boldsymbol{\gamma},\bar{u})$ as $\mathsf{S}_x^0({\bf h},\boldsymbol{\gamma})+ \mathsf{S}_x({\bf u},\boldsymbol{\gamma},\bar{u})-\mathsf{S}_x^0({\bf h},\boldsymbol{\gamma})$. Then, according to the proposition \ref{1proplambdamin},  a condition to insure $\mathsf{S}_x({\bf u},\bar{u})$ positive-definite is
\begin{equation}
\lambda^{\min}(\mathsf{S}_x^0({\bf h},\boldsymbol{\gamma}))+\lambda^{\min}(\mathsf{S}_x({\bf u},\boldsymbol{\gamma},\bar{u})-\mathsf{S}_x^0({\bf h},\boldsymbol{\gamma})) > 0.
\label{1lambdaminSx}
\end{equation}
As the spectrum of $\mathsf{S}_x({\bf u},\boldsymbol{\gamma},\bar{u})-\mathsf{S}_x^0({\bf h},\boldsymbol{\gamma})$ is explicit
\begin{equation}
\sigma \left( \mathsf{S}_x({\bf u},\boldsymbol{\gamma},\bar{u})-\mathsf{S}_x^0({\bf h},\boldsymbol{\gamma}) \right) = \left(\pm \alpha_{n,i}(u_i-\bar{u})\right)_{i \in [\![1,n]\!]},
\label{1spectra1}
\end{equation}
it is obvious that $\lambda^{\min}(\mathsf{S}_x({\bf u},\boldsymbol{\gamma},\bar{u})-\mathsf{S}_x^0({\bf h},\boldsymbol{\gamma}))=-\max_{j \in [\![1,n]\!]} \left(\alpha_{n,j} | u_j-\bar{u}| \right)$ and the matrix $\mathsf{S}_x({\bf u},\boldsymbol{\gamma},\bar{u})$ remains positive-definite if
\begin{equation}
 \forall i \in [\![1,n]\!],\ \lambda^{\min}(\mathsf{S}_x^0({\bf h},\boldsymbol{\gamma})) \ge \alpha_{n,i} |u_i-\bar{u}|.
\label{1deltai2}
\end{equation}
Finally, the lower bound of $\delta_i({\bf h},\boldsymbol{\gamma})$, for $i \in [\![1,n]\!]$, is straightforward obtained with the definition of $\delta_i({\bf h},\boldsymbol{\gamma})$ in theorem \ref{1condwellposedMLSW}.
\end{proof}

\noindent
As the lower bound (\ref{1estimdeltai}) is not explicit in ${\bf h}$ and $\boldsymbol{\gamma}$, we give, in the next proposition, an explicit lower bound of $\delta_i({\bf h},\boldsymbol{\gamma})$, for $i \in [\![1,n]\!]$.

\begin{proposition}
Let $\boldsymbol{\gamma} \in ]0,1[^{n-1}$ and ${\bf h} \in \mathbb{R}_+^{*\ n}$. A lower bound of $\delta_i({\bf h},\boldsymbol{\gamma})$, for $i \in [\![1,n]\!]$, is
\begin{equation}
\delta_i({\bf h},\boldsymbol{\gamma}) \ge \frac{1}{\left(\alpha_{n,i} a({\bf h},\boldsymbol{\gamma}) \right)^2}
\label{1estimdeltai2}
\end{equation}
where 
\begin{equation}
a({\bf h},\boldsymbol{\gamma}):=\max\left(\left(\frac{\alpha_{n,2}}{\alpha_{n,1}}+1 \right)p_1,2\max_{k \in [\![1,n-2]\!]}\left( p_k+p_{k+1} \right), \max_{i \in [\![1,n]\!]} \left(\alpha_{n,i}^{-1} h_i^{-1}\right) \right),
\label{1ahgamma}
\end{equation}
and for all $k \in [\![1,n-1]\!]$, $p_k:=\frac{1}{\alpha_{n,k+1}-\alpha_{n,k}}=\frac{\rho_n}{\rho_{k+1}-\rho_k}$.
\label{1propestimdeltai2}
\end{proposition}
\begin{proof}
First, in order to provide an explicit lower bound of $\lambda^{\min}(\mathsf{S}_x^0({\bf h},\boldsymbol{\gamma}))$, an upper bound of the spectral radius of $\mathsf{S}_x^0({\bf h},\boldsymbol{\gamma})^{\mathsf{-1}}$ is sufficient and is proved in the next lemma.
\begin{lemma}
Let $\boldsymbol{\gamma} \in ]0,1[^{n-1}$ and ${\bf h} \in \mathbb{R}_+^{*\ n}$. Then, the next inequality is verified
\begin{equation}
\lambda^{\max}(\mathsf{S}_x^0({\bf h},\boldsymbol{\gamma})^{\mathsf{-1}}) \le a({\bf h},\boldsymbol{\gamma}).
\label{1rayonspectralSx0m1}
\end{equation}
\label{1lemlambdamaxSx0m1}
\end{lemma}
\begin{proof}
We remind $\mathsf{S}_x^0({\bf h},\boldsymbol{\gamma}))$ is positive-definite under conditions (\ref{1Sx0posdef}). Then, the inverse of $\mathsf{S}_x^0({\bf h},\boldsymbol{\gamma})$ is
\begin{equation}
\mathsf{S}_x^0({\bf h},\boldsymbol{\gamma})^{\mathsf{-1}} :=  \left[\begin{array}{c|c|c}
 \begin{array}{c} \mathsf{\Delta}^{\mathsf{-1}} \mathsf{\Gamma}^{\mathsf{-1}}  \end{array}
 & \begin{array}{c} \mathsf{0} \end{array}
 & \begin{array}{c} \mathsf{0} \end{array}\\
 \hline
 \begin{array}{c} \mathsf{0} \end{array}
 & \begin{array}{c} \mathsf{\Delta}^{\mathsf{-1}} \mathsf{H}^{\mathsf{-1}} \end{array}
 & \begin{array}{c} \mathsf{0} \end{array}\\
 \hline
 \begin{array}{c} \mathsf{0} \end{array}
 & \begin{array}{c} \mathsf{0} \end{array}
 & \begin{array}{c} \mathsf{\Delta}^{\mathsf{-1}} \mathsf{H}^{-1} \end{array}\\
\end{array}\right],
\label{1Sx0-1}
\end{equation}
where $\mathsf{\Delta^{-1}}\mathsf{H}^{\mathsf{-1}}=\mathrm{diag}(\alpha_{n,1}^{-1}h_1^{-1},\ldots,\alpha_{n,n-1}^{-1}h_{n-1}^{-1},h_n^{-1})$. Moreover, one can verified that $\mathsf{\Delta}^{\mathsf{-1}} \mathsf{\Gamma}^{\mathsf{-1}}$ is a $n$-{\em Toeplitz} symmetric matrix ({\em i.e.} a tridiagonal symmetric matrix), defined by $\boldsymbol{p_1} \in \mathbb{R}^n$ on the diagonal and $\boldsymbol{p_2} \in \mathbb{R}^{n-1}$ just above and below this diagonal, with
\begin{equation}
\left\{
\begin{array}{l}
\boldsymbol{p_1}:={}^{\top}\left(\frac{\alpha_{n,2}}{\alpha_{n,1}}p_1,p_1+p_2,p_2+p_3,\ldots,p_{n-2}+p_{n-1},p_{n-1}  \right),\\[4pt]
\boldsymbol{p_2}:={}^{\top}\left(-p_1,-p_2,p_3,\ldots,-p_{n-1}  \right).
\end{array}
\right.
\label{1gamma12}
\end{equation}
Then, using the {\em Gerschgorin}'s theorem, there exists $k \in [\![1,n]\!]$ such that
\begin{equation}
\lambda^{\max}(\mathsf{\Delta}^{\mathsf{-1}} \mathsf{\Gamma}^{\mathsf{-1}}) \in \mathcal{D}_k(\mathsf{\Delta}^{\mathsf{-1}} \mathsf{\Gamma}^{\mathsf{-1}}),
\label{1lambdaMax}
\end{equation}
with the subsets $\left(\mathcal{D}_k\right)_{k \in [\![1,n]\!]} \subset \mathbb{C}$ defined by
\begin{equation}
\forall k \in [\![1,n]\!],\ \forall \mathsf{A}:=\left[ \mathsf{A}_{i,j} \right]_{(i,j) \in [\![1,n]\!]^2},\ \mathcal{D}_k\left(\mathsf{A}\right):=\left\{z \in \mathbb{C},\ |z-\mathsf{A}_{k,k}| \le \sum_{j\not=k} |\mathsf{A}_{k,j}| \right\}.
\label{1gerschgorincricle}
\end{equation}
Then, as $\mathsf{\Delta}^{\mathsf{-1}} \mathsf{\Gamma}^{\mathsf{-1}}$ is symmetric, real and positive-definite, $\lambda^{\max}(\mathsf{\Delta}^{\mathsf{-1}} \mathsf{\Gamma}^{\mathsf{-1}}) \in \mathbb{R}_+^*$ and using the tridiagonal structure of $\mathsf{\Delta}^{\mathsf{-1}} \mathsf{\Gamma}^{\mathsf{-1}}$
\begin{equation}
\lambda^{\max}(\mathsf{\Delta}^{\mathsf{-1}} \mathsf{\Gamma}^{\mathsf{-1}}) \le a_1({\bf h},\boldsymbol{\gamma}),
\label{1lambdamaxupperb}
\end{equation}
with $a_1({\bf h},\boldsymbol{\gamma}):=\max \left( \left(\frac{\alpha_{n,2}}{\alpha_{n,1}}+1 \right)p_1,2p_{n-1},2\max_{k \in [\![1,n-2]\!]}\left( p_i+p_{i+1} \right) \right)$. Finally, as we have for all $i \in [\![1,n-1]\!]$, $\gamma_i \in ]0,1[$, it implies that
\begin{equation}
\forall i \in [\![1,n-1]\!],\ p_i >0.
\label{1majorp1p2}
\end{equation}
Consequently, the inequality \eqref{1rayonspectralSx0m1} is proved, using the structure of $\mathsf{S}_x^0({\bf h},\boldsymbol{\gamma})^{\mathsf{-1}}$:
\begin{equation}
\lambda^{\max}(\mathsf{S}_x^0({\bf h},\boldsymbol{\gamma})^{\mathsf{-1}}) \le \max\left(\left(\frac{\alpha_{n,2}}{\alpha_{n,1}}+1 \right)p_1,2\max_{k \in [\![1,n-2]\!]}\left( p_k+p_{k+1} \right), \max_{i \in [\![1,n]\!]} \left(\alpha_{n,i}^{-1} h_i^{-1}\right) \right),
\label{1rayonspectralSx0m12}
\end{equation}
and $a({\bf h},\boldsymbol{\gamma})$ is an explicit upper bound of $\lambda^{\max}(\mathsf{S}_x^0({\bf h},\boldsymbol{\gamma})^{\mathsf{-1}})$.
\end{proof}

\noindent
Finally, we can prove the explicit lower bound (\ref{1estimdeltai2}). As $\mathsf{S}_x^0({\bf h},\boldsymbol{\gamma})$ is positive-definite if $\boldsymbol{\gamma} \in ]0,1[^{n-1}$ and ${\bf h} \in \mathbb{R}_+^{*\ n}$, its smallest eigenvalue is the inverse of the greatest eigenvalue of $\mathsf{S}_x^0({\bf h},\boldsymbol{\gamma})^{\mathsf{-1}}$. Then, according to the lemma \ref{1lemlambdamaxSx0m1},
\begin{equation}
\lambda^{\min}(\mathsf{S}_x^0({\bf h},\boldsymbol{\gamma})) \ge \frac{1}{a({\bf h},\boldsymbol{\gamma})},
\label{1lambdamin}
\end{equation}
and using the proposition \ref{1propestimdeltai}, the explicit lower bound (\ref{1estimdeltai2}) of $\delta_i({\bf h},\boldsymbol{\gamma})$ is insured.
\end{proof}

\noindent
{\em Remark:} The lower bound \eqref{1estimdeltai2} is explicit but rougher than the lower bound \eqref{1estimdeltai}. However, the main loss was due to the concave-inequality \eqref{1lambdaminSx}.

\section{Hyperbolicity of particular cases}
According to the previous section, the system \eqref{1systemmultilayer}, with initial data ${\bf u^0} \in \mathcal{H}^s(\mathbb{R}^2)^{3n}$, $s>2$, is hyperbolic if ${\bf u^0} \in \mathcal{S}^s_{\boldsymbol{\gamma}}$. However, this was just a sufficient condition of hyperbolicity. The aim of this section is to analyse the eigenstructure of particular cases and to obtain an explicit criterion of hyperbolicity ({\em i.e.} weaker than the lower bound of $\delta_{i_0}({\bf h}(X),\boldsymbol{\gamma})$ in \eqref{1estimdeltai2}, for one $i_0 \in [\![1,n]\!]$). This will provide another necessary criterion for initial conditions to be in the set of hyperbolicity of the model \eqref{1systemmultilayer}: $\mathcal{H}_{\boldsymbol{\gamma}}$, defined by
\begin{equation}
\mathcal{H}_{\boldsymbol{\gamma}}:= \left\{ {\bf u^0} \in \mathcal{L}^2(\mathbb{R}^2)^{3n} / {\bf u^0}\ \mathrm{verifies}\ \mathrm{conditions}\ {\rm (\ref{1spectrumAutheta})} \right\}
\label{1defHgamma}
\end{equation}

\noindent
To succeed, $\boldsymbol{\gamma} \in ]0,1[^{n-1}$ is set and only for one $i_0 \in [\![1,n]\!]$, the asymptotic case $1-\gamma_{i_0} \rightarrow 0$ is studied, in order to extract the criterion of hyperbolicity. The technique is based on the analysis performed for the two-layer model in \cite{monjarret2014local}. 

\subsection{Eigenstructure of $\boldsymbol{\mathsf{A}({\bf u},\boldsymbol{\gamma},\theta)}$}
Using the rotational invariance \eqref{1rotinv}, the eigenstructure of $\mathsf{A}({\bf u},\boldsymbol{\gamma},\theta)$ is deduced from the one of $\mathsf{A}_x({\bf u},\boldsymbol{\gamma})$. Moreover, as the eigenstructure of $\mathsf{A}_x({\bf u},\boldsymbol{\gamma})$ will be analyzed, the canonical basis of $\mathbb{R}^{3n}$ will be necessary and denoted by $(\boldsymbol{e_i})_{i \in [\![1,3n]\!]}$. For every eigenvalue $\lambda \in \sigma\left(\mathsf{A}_x({\bf u},\boldsymbol{\gamma}) \right)$, the associated eigenspace will be noted $\mathcal{E}_{\lambda}({\bf u},\boldsymbol{\gamma}):=\ker\left(\mathsf{A}_x({\bf u},\boldsymbol{\gamma})-\lambda \mathsf{I_{3n}} \right)$; the geometric multiplicity will be denoted by $\mu_{\lambda}({\bf u},\boldsymbol{\gamma}):=\dim\mathcal{E}_{\lambda}({\bf u},\boldsymbol{\gamma})$; the associated right eigenvector will be noted $\boldsymbol{r_x^{\lambda}}({\bf u},\boldsymbol{\gamma})$ and the left one $\boldsymbol{l_x^{\lambda}}({\bf u},\boldsymbol{\gamma})$. First, we prove the next proposition:
\begin{proposition}
The characteristic polynomial of $\mathsf{A}_x({\bf u},\boldsymbol{\gamma})$ is equal to
\begin{equation}
\det\left(\mathsf{A}_x({\bf u},\boldsymbol{\gamma})-\lambda \mathsf{I_{3n}} \right)= \det\left( \mathsf{M}_x(\lambda,{\bf u},\boldsymbol{\gamma})  \right) \prod_{i=1}^n (u_i-\lambda),
\label{1caractpol}
\end{equation}
where the $n \times n$ matrix $\mathsf{M}_x(\lambda,{\bf u},\boldsymbol{\gamma}):=(\mathsf{V}_x-\lambda \mathsf{I_n})^2-\mathsf{\Gamma} \mathsf{H}$\label{1propcharactpoly}
\end{proposition}
\begin{proof}
First of all, according to the block-structure of $\mathsf{A}_x({\bf u},\boldsymbol{\gamma})$, it is clear that
\begin{equation}
\det\left(\mathsf{A}_x({\bf u},\boldsymbol{\gamma})-\lambda \mathsf{I_{3n}} \right)=\det\left( \mathsf{A}_x^1({\bf u},\boldsymbol{\gamma})-\lambda \mathsf{I_{2n}}  \right) \prod_{i=1}^n (u_i-\mu),
\label{1detblockstruct}
\end{equation}
where the $2n \times 2n$ matrix $\mathsf{A}_x^1({\bf u},\boldsymbol{\gamma})$ is defined by
\begin{equation}
\mathsf{A}_x^1({\bf u},\boldsymbol{\gamma}) :=  \left[\begin{array}{c|c}
 \begin{array}{c} \mathsf{V}_x  \end{array}
 & \begin{array}{c} \mathsf{H} \end{array}\\
 \hline
 \begin{array}{c} \mathsf{\Gamma} \end{array}
 & \begin{array}{c} \mathsf{V}_x \end{array}\\
\end{array}\right].
\label{1Ax1}
\end{equation}
Then, as all the blocks of $\mathsf{A}_x^1({\bf u},\boldsymbol{\gamma})$ commute, the characteristic polynomial of $\mathsf{A}_x^1({\bf u},\boldsymbol{\gamma})$ is equal to $\det\left(\mathsf{M}_x(\lambda,{\bf u},\boldsymbol{\gamma}) \right)$.
\end{proof}

\noindent
According to the expression of the characteristic polynomial of $\mathsf{A}_x({\bf u},\boldsymbol{\gamma})$ in \eqref{1caractpol}, we denote the spectrum of this matrix by 
\begin{equation}
\sigma(\mathsf{A}_x({\bf u},\boldsymbol{\gamma})):=\left(\lambda_i^{\pm}({\bf u},\boldsymbol{\gamma})\right)_{i \in [\![1,n]\!]} \cup \left( \lambda_{2n+i}({\bf u},\boldsymbol{\gamma})\right)_{i \in [\![1,n]\!]},
\label{1sigmaAx}
\end{equation}
where $\left(\lambda_i^{\pm}({\bf u},\boldsymbol{\gamma})\right)_{i \in [\![1,n]\!]}=:\sigma(\mathsf{A}_x^1({\bf u},\boldsymbol{\gamma}))$ and 
\begin{equation}
\forall i \in [\![1,n]\!],\ \lambda_{2n+i}({\bf u},\boldsymbol{\gamma}):=u_i.
\label{1lambda2npi}
\end{equation}

\noindent
{\em Remarks:} 1) Using the rotational invariance \eqref{1rotinv}, the spectrum of $\mathsf{A}({\bf u},\boldsymbol{\gamma},\theta)$ will be
\begin{equation}
\sigma \left( \mathsf{A}({\bf u},\boldsymbol{\gamma},\theta) \right)=\left(\lambda_i^{\pm}(\mathsf{P}(\theta){\bf u},\boldsymbol{\gamma})\right)_{i \in [\![1,n]\!]} \cup \left( \lambda_{2n+i}(\mathsf{P}(\theta){\bf u},\boldsymbol{\gamma})\right)_{i \in [\![1,n]\!]} .
\label{1spectrAtheta}
\end{equation}
\noindent
2) The eigenvalues $\left(\lambda_i^{\pm}({\bf u},\boldsymbol{\gamma})\right)_{i \in [\![1,n-1]\!]}$ will be called the baroclinic eigenvalues and $\lambda_n^{\pm}({\bf u},\boldsymbol{\gamma})$ will be called the barotropic eigenvalues.

\noindent
As the eigenstructure associated to $\left( \lambda_{2n+i}({\bf u},\boldsymbol{\gamma})\right)_{i \in [\![1,n]\!]}$ is entirely known
\begin{equation}
\forall i \in [\![1,n]\!],\
\left\{
\begin{array}{l}
\boldsymbol{r_x^{\lambda_{2n+i}}}({\bf u},\boldsymbol{\gamma})=\boldsymbol{ e_{2n+i}},\\[4pt]
\boldsymbol{l_x^{\lambda_{2n+i}}}({\bf u},\boldsymbol{\gamma})={}^{\top} \boldsymbol{ e_{2n+i}},
\end{array}
\right.
\label{1vectpropVx}
\end{equation}
the following study is only focused on $\sigma(\mathsf{A}_x^1({\bf u},\boldsymbol{\gamma}))$. Moreover, as
\begin{equation}
\mathsf{M}_x(\lambda,{\bf u},\boldsymbol{\gamma})=(\mathsf{V}_x-\bar{u}\mathsf{I_n}-(\lambda-\bar{u}) \mathsf{I_n})^2-\mathsf{\Gamma} \mathsf{H},
\end{equation}
the analysis will be performed with the rescaling
\begin{equation}
\forall i \in [\![1,n]\!],\
\left\{
\begin{array}{l}
\tilde{\lambda}_i^{\pm} := \lambda_i^{\pm}-\bar{u},
\tilde{u}_i:=u_i-\bar{u}.
\end{array}
\right.
\end{equation}

\noindent
In this part, we will remove the $\tilde{ }$ and we consider ${\bf u}$ such that $\bar{u}=0$. In the following study, we set $f_n: \mathbb{R} \times \mathbb{R}^{3n} \times \mathbb{R}^{n-1} \rightarrow \mathbb{R}$ such that
\begin{equation}
\forall (\lambda,{\bf u},\boldsymbol{\gamma}) \in \mathbb{R} \times \mathbb{R}^{3n} \times \mathbb{R}^{n-1},\  f_n(\lambda,{\bf u},\boldsymbol{\gamma}):=\det\left( \mathsf{M}_x(\lambda,{\bf u},\boldsymbol{\gamma})  \right)
\label{1f}
\end{equation}

\subsection{A $\boldsymbol{1^{\mathrm{st}}}$ case: the single-layer model}\label{1sectioncase1}
The single-layer model with free surface is characterized by $({\bf u},\boldsymbol{\gamma})=({\bf u^0},\boldsymbol{\gamma^0})$, where ${\bf u}^{\boldsymbol{0}}$ and $\boldsymbol{\gamma^0}$ are defined by
\begin{equation}
\forall i \in [\![1,n-1]\!],\
\left\{
\begin{array}{l}
u_{i}=u_{i+1}=0\\[4pt]
\gamma_i=1.
\end{array}
\right.
\label{1conddegen1}
\end{equation}
In that case, the spectrum of $\mathsf{A}_x^1({\bf u^0},\boldsymbol{\gamma^0})$ is always real and is such that
\begin{equation}
\left\{
\begin{array}{lr}
\lambda_i^{\pm}({\bf u^0},\boldsymbol{\gamma^0})=\lambda_0, & \forall i \in [\![1,n-1]\!],\\[4pt]
\lambda_{n}^{\pm}({\bf u^0},\boldsymbol{\gamma^0})=\pm \sqrt{H},&
\end{array}
\right.
\label{1onelayerlambda}
\end{equation}
where $\lambda_0:=0$. The geometric multiplicity associated to $\lambda_0$ and $\lambda_{n}^{\pm}$ are respectively $\mu_{\lambda_0}=n-1$ and $\mu_{\lambda_{n}^{\pm}}=1$. The eigenvectors associated to this spectrum are
\begin{equation}
\forall  i \in [\![1,n-1]\!],\
\left\{
\begin{array}{l}
\boldsymbol{r_x^{\lambda_i^{\pm}}}({\bf u^0},\boldsymbol{\gamma^0})= \boldsymbol{ e_i} - \boldsymbol{ e_{i+1}},\\[4pt]
\boldsymbol{l_x^{\lambda_i^{\pm}}}({\bf u^0},\boldsymbol{\gamma^0})= {}^{\top} \boldsymbol{ e_{n+i}} - {}^{\top} \boldsymbol{ e_{n+i+1}},
\end{array}
\right.
\label{1vi}
\end{equation}
\begin{equation}
\left\{
\begin{array}{ll}
\boldsymbol{r_x^{\lambda_n^{\pm}}}({\bf u^0},\boldsymbol{\gamma^0})=\sum_{k=1}^n \boldsymbol{ e_{n+k}}\pm \frac{h_k}{\sqrt{H}}\boldsymbol{ e_k}\\[4pt]
\boldsymbol{l_x^{\lambda_n^{\pm}}}({\bf u^0},\boldsymbol{\gamma^0})=\sum_{k=1}^n {}^{\top} \boldsymbol{ e_{k}}\pm \frac{h_k}{\sqrt{H}}{}^{\top} \boldsymbol{ e_{n+k}}.
\end{array}
\right.
\label{1vn}
\end{equation}
To conclude, in the single-layer case, the model is hyperbolic but there is no eigenbasis of $\mathbb{R}^{3n}$.

\subsection{A $\boldsymbol{2^{\mathrm{nd}}}$ case: the merger of two layers}\label{1sectioncase2}

The merger of two layers is characterized by the equality of the parameters of two neighboring layers: $i \in [\![1,n-1]\!]$ such that $({\bf u},\boldsymbol{\gamma})=({\bf u^{\boldsymbol{i}}},\boldsymbol{\gamma^{i}})$, where ${\bf u}^{\boldsymbol{i}}$ and $\boldsymbol{\gamma^{i}}$ are defined by
\begin{equation}
\left\{
\begin{array}{lr}
u_j^2 \le \delta_j({\bf h},\boldsymbol{\gamma}), & \forall j \in [\![1,n]\!]\setminus \{i\},\\[4pt]
0 < \gamma_{j} < 1, & \forall j \in [\![1,n-1]\!]\setminus \{i\}. 
\end{array}
\right.
\label{1conddegen21}
\end{equation}
and for $i \in [\![1,n-1]\!]$,
\begin{equation}
\left\{
\begin{array}{l}
u_{i}=u_{i+1}\\[4pt]
\gamma_{i}=1.
\end{array}
\right.
\label{1conddegen22}
\end{equation}
Then, according to theorem \ref{1condwellposedMLSW}, it is hyperbolic and the spectrum of $\mathsf{A}_x^1({\bf u^{\boldsymbol{i}}},\boldsymbol{\gamma^{i}})$ is always a subset of $\mathbb{R}$. However there is no recursive method nor explicit expression to determine entirely this spectrum. Moreover, as the next equality on the columns is obvious,
\begin{equation}
\mathrm{C}_{i}\left(\mathsf{A}_x^1({\bf u^{\boldsymbol{i}}},\boldsymbol{\gamma^{i}})-u_{i} \mathsf{I_{2n}} \right)=\mathrm{C}_{i+1}\left(\mathsf{A}_x^1({\bf u^{\boldsymbol{i}}},\boldsymbol{\gamma^{i}})-u_{i} \mathsf{I_{2n}} \right),
\label{1equalcolumn2case}
\end{equation}
there is only one trivial value for the eigenvalues $\lambda_{i}^{\pm}= u_{i}$. And for this eigenvalues, the eigenvectors associated are
\begin{equation}
\left\{
\begin{array}{l}
\boldsymbol{r_x^{\lambda_{i}^{\pm}}}({\bf u^{\boldsymbol{i}}},\boldsymbol{\gamma^{i}})= \boldsymbol{ e_{i}} - \boldsymbol{ e_{i+1}},\\[4pt]
\boldsymbol{l_x^{\lambda_{i}^{\pm}}}({\bf u^{\boldsymbol{i}}},\boldsymbol{\gamma^{i}})= {}^{\top}\boldsymbol{ e_{n+i}} - {}^{\top}\boldsymbol{ e_{n+i+1}},
\end{array}
\right.
\label{1vi2}
\end{equation}
To conclude, as in the previous case, the model, with the merger of two layer, remains hyperbolic but there is no eigenbasis of $\mathbb{R}^{3n}$.

\subsection{The asymptotic expansion of the merger of two layers}
With the same notations as the previous subsection, we consider the merger of two layer: there exists $i \in [\![1,n-1]\!]$, such that conditions (\ref{1conddegen21}--\ref{1conddegen22}) are verified. As it was explained before, the eigenvalue $\lambda_{i}^{\pm}({\bf u},\boldsymbol{\gamma})$ is explicit but does not provide two distinct eigenvalues associated to the interface $i$, in order to get two distinct right eigenvectors. Indeed, proving the existence of two distinct right eigenvectors would be a first step to prove the diagonalizability of the matrix $\mathsf{A}({\bf u},\boldsymbol{\gamma},\theta)$, in order to apply proposition \ref{1diagimpliqsym}.

\begin{proposition}
Let $i \in [\![1,n-1]\!]$, $({\bf u},\boldsymbol{\gamma}) \in \mathbb{R}^{3n} \times ]0,1[^{n-1}$ such that $1-\gamma_i$ and $(u_j)_{j \in [\![1,n]\!]}$ are sufficiently small. Then, an expansion of $\lambda_i^{\pm}({\bf u},\boldsymbol{\gamma})$ is
\begin{equation}
\begin{array}{ll}
\lambda_i^{\pm}({\bf u},\boldsymbol{\gamma})=& \frac{u_i h_{i+1}+u_{i+1}h_i}{h_i+h_{i+1}} \pm \left[\frac{ h_i h_{i+1}}{h_i+h_{i+1}}\left(1-\gamma_i-\frac{(u_{i+1}-u_i)^2}{h_i+h_{i+1}}\right)\right]^{\frac{1}{2}}\\[6pt]
& +\mathcal{O}((1-\gamma_i),(u_j^2)_{j \in [\![1,n]\!]}).
\end{array}
\label{1approxlambdaipm}
\end{equation}
\label{1thmapproxlambdaipm}
\end{proposition}

\begin{proof}
In order to obtain an asymptotic expansion of $\lambda_{i}^{\pm}$, we perform a $2^{\mathrm{nd}}$ order {\em Taylor} expansion of $f_n$, about a state mixing the two cases analyzed in \S \ref{1sectioncase1} and \S \ref{1sectioncase2}:
\begin{equation}
\left\{
\begin{array}{l}
\lambda=0,\\
{\bf u}={\bf u^{\boldsymbol{0}}},\\
\boldsymbol{\gamma}=\boldsymbol{\gamma}^{i}
\end{array}
\right.
\label{1statedegen}
\end{equation}

\noindent
Then, we have

\begin{eqnarray}
f_n(\lambda,\boldsymbol{u},\boldsymbol{\gamma})& &= f_n(0,{\bf u^0},\boldsymbol{\gamma^i})+\lambda\frac{\partial f_n}{\partial \lambda}(0,{\bf u^0},\boldsymbol{\gamma^i})+(\gamma_i-1)\frac{\partial f_n}{\partial \gamma_i}(0,{\bf u^0},\boldsymbol{\gamma^i})\nonumber\\
&& +\sum_{j=1}^n u_j\frac{\partial f_n}{\partial u_j}(0,{\bf u^0},\boldsymbol{\gamma^i})+\frac{1}{2}\lambda^2 \frac{\partial^2 f_n}{\partial \lambda^2}(0,{\bf u^0},\boldsymbol{\gamma^i})\nonumber\\
&& +\sum_{j=1}^n \frac{1}{2}u_j^2 \frac{\partial^2 f_n}{\partial u_j^2}(0,{\bf u^0},\boldsymbol{\gamma^i})+\frac{1}{2}(\gamma_i - 1)^2 \frac{\partial^2 f_n}{\partial \gamma_i^2}(0,{\bf u^0},\boldsymbol{\gamma^i})\nonumber\\
&& +\sum_{j=1}^n u_j\left(\lambda \frac{\partial^2 f_n}{\partial u_j \partial \lambda}(0,{\bf u^0},\boldsymbol{\gamma^i})+(\gamma_i-1)\frac{\partial^2 f_n}{\partial u_j \partial \gamma_i}(0,{\bf u^0},\boldsymbol{\gamma^i})\right)\nonumber\\
&& +\sum_{j\not=k}u_ju_k\frac{\partial^2 f_n}{\partial u_j \partial u_k}(0,{\bf u^0},\boldsymbol{\gamma^i})+\lambda(\gamma_i-1)\frac{\partial^2 f_n}{\partial \lambda \partial \gamma_i}(0,{\bf u^0},\boldsymbol{\gamma^i})\nonumber\\
&&+o(\lambda^2,1-\gamma_i,\left(u_j^2\right)_{j \in [\![1,n]\!]}),\nonumber
\label{1Taylor}
\end{eqnarray}
To calculate all these derivatives, we use the following lemmata:

\begin{lemma}Let $\boldsymbol{\gamma} \in ]0,1[^{n-1}$ and ${\bf h} \in \mathbb{R}_+^{*\ n}$, then
\begin{equation}
f_n(0,{\bf u^0},\boldsymbol{\gamma})=(-1)^n h_n \prod_{i=1}^{n-1} h_i (1-\gamma_i)
\end{equation}
\label{1lemmagamma1}
\end{lemma}
\begin{proof}
First, we perform the next operations to the columns of $\mathsf{M}_x(\lambda,{\bf u^0},\boldsymbol{\gamma})$: for all $k \in [\![1,n-1]\!]$,
\begin{equation}
\mathrm{C}_{k}\left(\mathsf{M}_x(\lambda,{\bf u^0},\boldsymbol{\gamma}) \right) \leftarrow \mathrm{C}_{k}\left(\mathsf{M}_x(\lambda,{\bf u^0},\boldsymbol{\gamma}) \right)-\mathrm{C}_{k+1}\left(\mathsf{M}_x(\lambda,{\bf u^0},\boldsymbol{\gamma}) \right)
\label{1columnslammegamma1}
\end{equation}
Finally, with an expansion of the determinant obtained, about the $1^{\mathrm{st}}$ line, the lemma \ref{1lemmagamma1} is proved.
\end{proof}

\noindent
In the next lemma, for $k \in [\![1,n-1]\!]$, we denote by $\mathsf{M}_x^k(\lambda,{\bf u},\boldsymbol{\gamma})$, the $n-1 \times n-1$ matrix obtained with $\mathsf{M}_x(\lambda,{\bf u},\boldsymbol{\gamma})$ with the $k^{\mathrm{th}}$ column and $k^{\mathrm{th}}$ line removed; and by $f^k_n: \mathbb{R} \times \mathbb{R}^{3n} \times \mathbb{R}^{n-1} \rightarrow \mathbb{R}$ such that $f^k_n(\lambda,{\bf u},\boldsymbol{\gamma})$ is the $k^{\mathrm{th}}$ first minor of $\mathsf{M}_x(\lambda,{\bf u},\boldsymbol{\gamma})$
\begin{equation}
f^k_n(\lambda,{\bf u},\boldsymbol{\gamma}):= \det \left( \mathsf{M}_x^k(\lambda,{\bf u},\boldsymbol{\gamma}) \right).
\label{1fkdef}
\end{equation}

\begin{lemma}
Let $k \in [\![1,n]\!]$, $\boldsymbol{\gamma} \in ]0,1[^{n-1}$ and ${\bf h} \in \mathbb{R}_+^{*\ n}$, then
\begin{equation}
f^k_n(0,{\bf u^0},\boldsymbol{\gamma})=
\left\{
\begin{array}{lr}
(-1)^{n-1} \prod_{j=2}^n h_j \prod_{j=2}^{n-1} (1-\gamma_j),& \mathrm{if}\ k=1,\\[4pt]
(-1)^{n-1}  \eta_k \prod_{j=1,j \not =k}^{n} h_j\prod_{j=1,j \not \in \{k-1,k\}}^{n-1} (1-\gamma_j),& \mathrm{if}\ k \in [\![2,n-1]\!],\\[4pt]
(-1)^{n-1} \prod_{j=1}^{n-1} h_j \prod_{j=1}^{n-2}(1-\gamma_j),& \mathrm{if}\ k=n,
\end{array}
\right.
\label{1fk}
\end{equation}
where $\forall k \in [\![2,n-1]\!]$, $\eta_k:=1-\gamma_{k-1}\gamma_k$.
\label{lemmalinecolumn}
\end{lemma}
\begin{proof}
First, we just remark that
\begin{equation}
f^k_n(0,{\bf u^0},\boldsymbol{\gamma})=f_{n-1}(0,{\bf u^0}(k),\boldsymbol{\gamma}(k)),
\label{1fkdet}
\end{equation}
where ${\bf u^0}(k) \in \mathbb{R}^{3n-3}$ is the vector ${\bf u^0}$, where $h_k$, $u_k$ and $v_k$ have been removed; and $\boldsymbol{\gamma}(k) \in \mathbb{R}^{n-1}$ is defined by
\begin{equation}
\boldsymbol{\gamma}(k):=
\left\{
\begin{array}{lr}
{}^{\top} (\gamma_2,\ldots,\gamma_{n-1}), & \mathrm{if}\ k=1,\\[4pt]
{}^{\top} (\gamma_1,\ldots,\gamma_{k-2},\gamma_{k-1} \gamma_k,\gamma_{k+1},\ldots,\gamma_{n-1}), & \mathrm{if}\ k \in [\![2,n-1]\!],\\[4pt]
{}^{\top} (\gamma_1,\ldots,\gamma_{n-2}),& \mathrm{if}\ k=n.
\end{array}
\right.
\label{1gammak}
\end{equation}
Then, the lemma \ref{lemmalinecolumn} is straightforward deduced, as a direct application of the lemma \ref{1lemmagamma1}. 
\end{proof}

\noindent
Furthermore, using the lemma \ref{1lemmagamma1} and reminding that $\boldsymbol{\gamma^i}$ is defined such that $\gamma_i=1$, then it is clear that
\begin{equation}
f_n(0,{\bf u^0},\boldsymbol{\gamma^i})=0.
\label{1fgammai}
\end{equation}
Consequently, all the derivatives of the $2^{\mathrm{nd}}$ order {\em Taylor} expansion of $f_n$, about the state $\lambda=0$, ${\bf u}={\bf u^0}$ and $\boldsymbol{\gamma}=\boldsymbol{\gamma}^{i}$, are deduced from the particular structure of $\mathsf{M}_x(\lambda,{\bf u^0},\boldsymbol{\gamma})$ and lemmata \ref{1lemmagamma1}--\ref{lemmalinecolumn}.

\begin{lemma}
The $1^{\mathrm{st}}$ order partial derivatives are such that
\begin{equation}
\left\{
\begin{array}{llr}
\frac{\partial f_n}{\partial \lambda}(0,{\bf u^{\boldsymbol{0}}},\boldsymbol{\gamma^i})&=0,&\\[4pt]
\frac{\partial f_n}{\partial u_j}(0,{\bf u^{\boldsymbol{0}}},\boldsymbol{\gamma^i})&=0, &\forall j \in [\![1,n]\!]\\[4pt]
\frac{\partial f_n}{\partial \gamma_i}(0,{\bf u^{\boldsymbol{0}}},\boldsymbol{\gamma^i})&=(-1)^{n+1} h_n h_i \prod_{j=1,j\not=i}^{n-1} h_j (1-\gamma_j).&
\end{array}
\right.
\label{1derivative1}
\end{equation}
\label{1lemmaderivee1}
\end{lemma}
\begin{proof}
Remarking
\begin{equation}
\left\{
\begin{array}{ll}
\frac{\partial f_n}{\partial \lambda}(\lambda,{\bf u},\boldsymbol{\gamma})=\sum_{k=1}^n -2(u_k-\lambda) f^k_n(\lambda,{\bf u},\boldsymbol{\gamma}),& \\[4pt]
\frac{\partial f_n}{\partial u_j}(\lambda,{\bf u},\boldsymbol{\gamma})=2(u_j-\lambda) f^k_n(\lambda,{\bf u},\boldsymbol{\gamma}),& \forall j \in [\![1,n]\!].
\end{array}
\right.
\label{1dfdlambda}
\end{equation}
and, according to the definition of ${\bf u^{\boldsymbol{0}}}$ in (\ref{1conddegen1}): $\forall k \in [\![1,n]\!]$, $u_k=0$, it is straightforward to prove the two $1^{\mathrm{st}}$ derivatives:
\begin{equation}
\left\{
\begin{array}{lr}
\frac{\partial f_n}{\partial \lambda}(0,{\bf u^{\boldsymbol{0}}},\boldsymbol{\gamma^i})=0,& \\[4pt]
\frac{\partial f_n}{\partial u_j}(0,{\bf u^{\boldsymbol{0}}},\boldsymbol{\gamma^i})=0, & \forall j \in [\![1,n]\!].
\end{array}
\right.
\label{1dfdlambda2}
\end{equation}

\noindent
The $3^{\mathrm{rd}}$ one is obtained remarking that in each column of $\mathsf{M}_x(\lambda,{\bf u},\boldsymbol{\gamma})$, the terms in $\boldsymbol{\gamma}$ are not correlated with the terms in $\lambda$ and ${\bf u}$. Then, the result is proved, applying the lemma \ref{1lemmagamma1}.
\end{proof}

\begin{lemma}
The $2^{\mathrm{nd}}$ order partial derivatives are
\begin{equation}
\left\{
\begin{array}{llr}
\frac{\partial^2 f_n}{\partial \lambda \partial \gamma_i}(0,{\bf u^{\boldsymbol{0}}},\boldsymbol{\gamma^i})&=0,&\\[4pt]
\frac{\partial^2 f_n}{\partial u_j \partial u_k}(0,{\bf u^{\boldsymbol{0}}},\boldsymbol{\gamma^i})&=0,& \forall j \not =k,\\[4pt]
\frac{\partial^2 f_n}{\partial u_j \partial \gamma_i}(0,{\bf u^{\boldsymbol{0}}},\boldsymbol{\gamma^i})&=0,& \forall j \in [\![2,n]\!],\\[4pt]
\frac{\partial^2 f_n}{\partial \gamma_i^2}(0,{\bf u^{\boldsymbol{0}}},\boldsymbol{\gamma^i})&=0,&\\[4pt]
\frac{\partial^2 f_n}{\partial u_j^2}(0,{\bf u^{\boldsymbol{0}}},\boldsymbol{\gamma^i})&=0, &\forall j \not \in \{i,i+1\},\\[4pt]
\frac{\partial^2 f_n}{\partial \lambda \partial u_j}(0,{\bf u^{\boldsymbol{0}}},\boldsymbol{\gamma^i})&=0,& \forall j \not \in \{i,i+1\},
\end{array}
\right.
\label{1derivative3}
\end{equation}
\begin{equation}
\left\{
\begin{array}{ll}
\frac{\partial^2 f_n}{\partial \lambda^2}(0,{\bf u^{\boldsymbol{0}}},\boldsymbol{\gamma^i})&=(h_i+h_{i+1})\kappa_i,\\[4pt]
\frac{\partial^2 f_n}{\partial \lambda \partial u_{i+1}}(0,{\bf u^{\boldsymbol{0}}},\boldsymbol{\gamma^i})&=-h_i \kappa_i,\\[4pt]
\frac{\partial^2 f_n}{\partial u_{i+1}^2}(0,{\bf u^{\boldsymbol{0}}},\boldsymbol{\gamma^i})&=h_i \kappa_i,\\[4pt]
\frac{\partial^2 f_n}{\partial \lambda \partial u_i}(0,{\bf u^{\boldsymbol{0}}},\boldsymbol{\gamma^i})&=-h_{i+1} \kappa_i,\\[4pt]
\frac{\partial^2 f_n}{\partial u_i^2}(0,{\bf u^{\boldsymbol{0}}},\boldsymbol{\gamma^i})&=h_{i+1}\kappa_i,
\end{array}
\right.
\label{1derivative2}
\end{equation}
where $\kappa_i:=2(-1)^{n-1}\prod_{j=1,j\not \in\{i,i+1\}}^n h_j\prod_{j=1,j\not=i}^{n-1} (1-\gamma_j)$.
\label{1derivee2}
\end{lemma}
\begin{proof}
We note that
\begin{equation}
\left\{
\begin{array}{lr}
\frac{\partial^2 f_n}{\partial \lambda^2}(\lambda,{\bf u},\boldsymbol{\gamma})= \sum_{k=1}^n 2 f^k_n(\lambda,{\bf u},\boldsymbol{\gamma}),& \\[4pt]
\frac{\partial^2 f_n}{\partial u_j^2}(\lambda,{\bf u},\boldsymbol{\gamma})=2 f^j_n(\lambda,{\bf u},\boldsymbol{\gamma}),& \forall j \in [\![1,n]\!],\\[4pt]
\frac{\partial^2 f_n}{\partial \lambda \partial u_j}(\lambda,{\bf u},\boldsymbol{\gamma})=-2 f^j_n(\lambda,{\bf u},\boldsymbol{\gamma}),& \forall j \in [\![1,n]\!],
\end{array}
\right.
\label{1sndderivative}
\end{equation}
and as it was noticed before
\begin{equation}
\left\{
\begin{array}{lr}
\frac{\partial^2 f_n}{\partial \gamma_i^2}(\lambda,{\bf u},\boldsymbol{\gamma})=\frac{\partial^2 f_n}{\partial \gamma_i^2}(0,{\bf u^{\boldsymbol{0}}},\boldsymbol{\gamma}),& \\[4pt]
\frac{\partial^2 f_n}{\partial \lambda \partial \gamma_i}(\lambda,{\bf u},\boldsymbol{\gamma})=0,& \\[4pt]
\frac{\partial^2 f_n}{\partial u_j \partial \gamma_i}(\lambda,{\bf u},\boldsymbol{\gamma})=0,& \forall j \in [\![1,n]\!].
\end{array}
\right.
\label{1deriv2gamma}
\end{equation}
Moreover, according to the definition of $\boldsymbol{\gamma^i}$ in (\ref{1conddegen21}--\ref{1conddegen22}) and the expression of $f^j_n(0,{\bf u^0},\boldsymbol{\gamma})$ in (\ref{1fk}),
\begin{equation}
\forall j \in [\![1,n]\!],\ f^j_n(0,{\bf u^0},\boldsymbol{\gamma^i})=
\left\{
\begin{array}{lr}
0, & \mathrm{if}\ j \not= i,i+1,\\[4pt]
\frac{1}{2} \kappa_i h_{i+1}, & \mathrm{if}\ j=i,\\[4pt]
\frac{1}{2} \kappa_i h_{i}, & \mathrm{if}\ j=i+1,
\end{array}
\right.
\label{1fkexpress}
\end{equation}
the other derivatives are directly calculated.
\end{proof}

\noindent
Using lemmata \ref{1lemmagamma1}--\ref{1derivee2}, the $2^{\mathrm{nd}}$ order {\em Taylor} expansion of $f$, about the state (\ref{1statedegen}), becomes
\begin{equation}
\begin{array}{ll}
0=& \kappa_i \bigg[ \frac{1}{2}(\gamma_i-1) h_i h_{i+1}  +\frac{1}{2}\lambda^2 (h_i+h_{i+1})  + \frac{1}{2}u_i^2 h_{i+1}  + \frac{1}{2}u_{i+1}^2 h_i \\
& - \lambda u_i h_{i+1}  - \lambda u_{i+1} h_{i}  \bigg] +o(\lambda^2,1-\gamma_i,\left(u_j^2\right)_{j \in [\![1,n]\!]}).\\
\end{array}
\label{1secondtaylorderivecalculee}
\end{equation}
Finally, as $\kappa_i \not=0$, we apply the implicit function theorem and obtain the expression (\ref{1approxlambdaipm}).
\end{proof}

\begin{theorem}
Let $i \in [\![1,n-1]\!]$, $({\bf u},\boldsymbol{\gamma}) \in \mathbb{R}^{3n} \times ]0,1[^{n-1}$ such that ${\bf h}>0$, and $1-\gamma_i$, $(u_j)_{j \in [\![1,n]\!]}$ and $(v_j)_{j \in [\![1,n]\!]}$ are sufficiently small. Then, a necessary condition of hyperbolicity for the model {\em (\ref{1systemmultilayer})} is
\begin{equation}
(u_{i+1}-u_i)^2+(v_{i+1}-v_i)^2 \le (h_i+h_{i+1})(1-\gamma_i).
\label{1condnechyp}
\end{equation}
\label{1corolcaluderiv1}
\end{theorem}
\begin{proof}
To verify the hyperbolicity of the system (\ref{1systemmultilayer}), all the eigenvalues of $\mathsf{A}({\bf u},\boldsymbol{\gamma},\theta)$ need to be real. According to the rotational invariance (\ref{1rotinv}) and the proposition \ref{1thmapproxlambdaipm}, if $1-\gamma_i$, $(u_j)_{j \in [\![1,n]\!]}$ and $(v_j)_{j \in [\![1,n]\!]}$ are sufficiently small, the asymptotic expansion of $\lambda_i^{\pm}(\mathsf{P}(\theta){\bf u},\boldsymbol{\gamma}) \in \sigma \left( \mathsf{A}({\bf u},\boldsymbol{\gamma},\theta \right)$ is
\begin{equation}
\begin{array}{ll}
\lambda_i^{\pm}(\mathsf{P}(\theta){\bf u},\boldsymbol{\gamma})=& \cos(\theta) \frac{u_i h_{i+1}+u_{i+1}h_i}{h_i+h_{i+1}}+ \sin(\theta) \frac{v_i h_{i+1}+v_{i+1}h_i}{h_i+h_{i+1}}\\[4pt]
& \pm \left[\frac{ h_i h_{i+1}}{h_i+h_{i+1}}\left(1-\gamma_i- \frac{(\cos(\theta) (u_{i+1}-u_i)+\sin(\theta) (v_{i+1}-v_i))^2}{h_i+h_{i+1}}\right)\right]^{\frac{1}{2}}\\[6pt]
&  +\mathcal{O}((1-\gamma_i),(u_j^2)_{j \in [\![1,n]\!]}).
\end{array}
\label{1approxlambdaipmtheta}
\end{equation}
\noindent
Then, as $h_i >0$ for all $i \in [\![1,n]\!]$, a necessary condition to have $\lambda_i^{\pm}(\mathsf{P}(\theta){\bf u},\boldsymbol{\gamma}) \in \mathbb{R}$, for all $\theta \in [0,2\pi]$ is
\begin{equation}
\forall \theta \in [0,2\pi],\  1-\gamma_i- \frac{(\cos(\theta) (u_{i+1}-u_i)+\sin(\theta) (v_{i+1}-v_i))^2}{h_i+h_{i+1}} \ge 0.
\label{1condneclambdarealtheta}
\end{equation}
Finally, using (\ref{maxPthetau}), the necessary condition of hyperbolicity \eqref{1condnechyp} is obtained.
\end{proof}

\noindent
With the asymptotic expansion (\ref{1approxlambdaipm}), we can deduce an asymptotic expansion of the eigenvectors associated to $\lambda_i^{\pm}({\bf u},\boldsymbol{\gamma})$.

\begin{proposition}
Let $i \in [\![1,n-1]\!]$, $({\bf u},\boldsymbol{\gamma}) \in \mathbb{R}^{3n} \times ]0,1[^{n-1}$ such that $1-\gamma_i$ and $(| u_j | )_{j \in [\![1,n]\!]}$ are sufficiently small. Then, the asymptotic expansion of the right eigenvector associated to $\lambda_i^{\pm}({\bf u},\boldsymbol{\gamma})$, with precision in $\mathcal{O}((1-\gamma_i),(u_j^2)_{j \in [\![1,n]\!]})$, is such that
\begin{equation}
\begin{array}{ll}
\boldsymbol{r_x^{\lambda_{i}^{\pm}}}({\bf u},\boldsymbol{\gamma})= & \boldsymbol{ e_{i}} - \boldsymbol{ e_{i+1}}+\frac{u_{i+1}-u_i}{h_i+h_{i+1}}({\boldsymbol{e_{n+i}}}+{\boldsymbol{e_{n+i+1}}})\\[4pt]
& \pm \left[ \frac{h_i h_{i+1}}{h_i + h_{i+1}} \left( 1-\gamma_i - \frac{(u_{i+1} - u_i)^2}{h_i + h_{i+1}} \right) \right]^{\frac{1}{2}}(\frac{{\boldsymbol{e_{n+i}}}}{h_i}-\frac{{\boldsymbol{e_{n+i+1}}}}{h_{i+1}})\\[6pt]
& +\mathcal{O}((1-\gamma_i),(u_j^2)_{j \in [\![1,n]\!]}),
\end{array}
\label{1eigvectlambdipmr}
\end{equation}
and the asymptotic expansion of the left eigenvector associated to $\lambda_i^{\pm}({\bf u},\boldsymbol{\gamma})$, with precision in $\mathcal{O}((1-\gamma_i),(u_j^2)_{j \in [\![1,n]\!]})$, is such that
\begin{equation}
\begin{array}{ll}
\boldsymbol{l_x^{\lambda_{i}^{\pm}}}({\bf u},\boldsymbol{\gamma})= & {}^{\top}\boldsymbol{ e_{n+i}} - {}^{\top}\boldsymbol{ e_{n+i+1}}+\frac{u_{i+1}-u_i}{h_i+h_{i+1}}({}^{\top}{\boldsymbol{e_{i}}}+{}^{\top}{\boldsymbol{e_{i+1}}})\\[4pt]
& \pm \left[ \frac{h_i h_{i+1}}{h_i + h_{i+1}} \left( 1-\gamma_i - \frac{(u_{i+1} - u_i)^2}{h_i + h_{i+1}} \right) \right]^{\frac{1}{2}}(\frac{{}^{\top}{\boldsymbol{e_{i}}}}{h_i}-\frac{{}^{\top}{\boldsymbol{e_{i+1}}}}{h_{i+1}})\\[6pt]
& +\mathcal{O}((1-\gamma_i),(u_j^2)_{j \in [\![1,n]\!]}).
\end{array}
\label{1eigvectlambdipml}
\end{equation}
\label{1thmeigvectlambdipm}
\end{proposition}
\begin{proof}
We consider $\lambda_i^{\pm}({\bf u},\boldsymbol{\gamma}) \in \mathbb{R}$:
\begin{equation}
(u_{i+1}-u_i)^2 \le (h_{i+1}+h_i)(1-\gamma_i).
\label{1condhyplambdaipm2}
\end{equation}
Then, we define $\pi_i \in [-(h_{i+1}+h_i)^{\frac{1}{2}},(h_{i+1}+h_i)^{\frac{1}{2}}]$ such that
\begin{equation}
(u_{i+1}-u_i)^2= \pi_i^2 (1-\gamma_i),
\label{1defpi}
\end{equation}

\begin{equation}
\lambda_i^{\pm}({\bf u},\boldsymbol{\gamma})=u_i+\chi_i^{\pm} (1-\gamma_i)^{\frac{1}{2}} +o((1-\gamma_i)^{\frac{1}{2}},\left(u_j \right)_{j \in [\![1,n]\!]} ),
\label{1reformullambdai}
\end{equation}
where $\chi_i^{\pm} := \pi_i \frac{h_i}{h_i+h_{i+1}} \pm \left[ \frac{h_i h_{i+1}}{h_i +h_{i+1}}\left( 1-\frac{\pi_i^2}{h_i +h_{i+1}} \right)  \right]^{\frac{1}{2}}$ and we will expand the eigenvectors $\boldsymbol{r_x^{\lambda_i^{\pm}}}({\bf u},\boldsymbol{\gamma})$ and $\boldsymbol{l_x^{\lambda_i^{\pm}}}({\bf u},\boldsymbol{\gamma})$ as
\begin{equation}
\forall  i \in [\![1,n-1]\!],\
\left\{
\begin{array}{l}
\boldsymbol{r_x^{\lambda_i^{\pm}}}({\bf u},\boldsymbol{\gamma})= \boldsymbol{r_{i,0}}({\bf u},\boldsymbol{\gamma})+(1-\gamma_i)^{\frac{1}{2}} \boldsymbol{r_{i,1}^{\pm}}({\bf u},\boldsymbol{\gamma}),\\[6pt]
\boldsymbol{l_x^{\lambda_i^{\pm}}}({\bf u},\boldsymbol{\gamma})= \boldsymbol{l_{i,0}}({\bf u},\boldsymbol{\gamma})+(1-\gamma_i)^{\frac{1}{2}} \boldsymbol{l_{i,1}^{\pm}}({\bf u},\boldsymbol{\gamma}),
\end{array}
\right.
\label{1eigvectoexpand}
\end{equation}
\noindent
where
\begin{equation}
\left\{
\begin{array}{l}
\boldsymbol{r_{i,0}}({\bf u},\boldsymbol{\gamma}):=\boldsymbol{ e_i} - \boldsymbol{ e_{i+1}},\\[4pt]
\boldsymbol{l_{i,0}}({\bf u},\boldsymbol{\gamma}):={}^{\top} \boldsymbol{ e_{n+i}} - {}^{\top} \boldsymbol{ e_{n+i+1}}.
\end{array}
\right.
\end{equation}

\noindent
Moreover, we have
\begin{equation}
\mathsf{A}_x({\bf u},\boldsymbol{\gamma})=\mathsf{A}_x({\bf u}^{\boldsymbol{i}},\boldsymbol{\gamma^i})+(u_{i+1}-u_i) \mathsf{A}_x^{i,1}+(1-\gamma_i)\mathsf{A}_x^{i,2}(\boldsymbol{\gamma}),
\label{1equalityA1A2}
\end{equation}
where the $3n \times 3n$ matrices, $\mathsf{A}_x^{i,1}:=\left[ \mathsf{A}_{l,k}^{i,1} \right]_{(l,k) \in [\![1,n]\!]^2}$ and $\mathsf{A}_x^{i,2}(\boldsymbol{\gamma}):=\left[ \mathsf{A}_{l,k}^{i,2} \right]_{(l,k) \in [\![1,n]\!]^2}$, are defined by
\begin{equation}
\mathsf{A}_{l,k}^{i,1}:=
\left\{
\begin{array}{ll}
1,& \mathrm{if}\ l=k\ \mathrm{and}\ l \in \{p n+ i+1 / p \in [\![0,2]\!]\},\\
0,& \mathrm{otherwise},
\end{array}
\right.
\label{1defA1}
\end{equation}

\begin{equation}
\mathsf{A}_{l,k}^{i,2}:=
\left\{
\begin{array}{ll}
0,& \mathrm{if}\ l \le n+i\ \mathrm{or}\ l \ge 2n+1,\\
0,& \mathrm{if}\ k \ge n+1,\\
0,& \mathrm{if}\ n+k \ge l,\\
-\alpha_{l-n-1,k},& \mathrm{otherwise},
\end{array}
\right.
\label{1defA2}
\end{equation}
 and
\begin{equation}
\mathsf{A}_x({\bf u},\boldsymbol{\gamma})=\mathsf{A}_x({\bf u}^{\boldsymbol{i}},\boldsymbol{\gamma^i})+\pi_i (1-\gamma_i)^{\frac{1}{2}} \mathsf{A}_x^{i,1}+(1-\gamma_i)\mathsf{A}_x^{i,2}(\boldsymbol{\gamma}).
\label{1equalityA1A22}
\end{equation}

\noindent
In the asymptotic regime $0<1-\gamma_i \ll 1$ and for every $j \in [\![1,n]\!], | u_j | \ll 1$, $\boldsymbol{r_x^{\lambda_i^{\pm}}}({\bf u},\boldsymbol{\gamma})$ and $\boldsymbol{l_x^{\lambda_i^{\pm}}}({\bf u},\boldsymbol{\gamma})$ are respectively the approximations of the right and left eigenvectors associated to $\lambda_x^{\pm}({\bf u},\boldsymbol{\gamma})$, with precision $\mathcal{O}(1-\gamma_i)$ if and only if $\boldsymbol{r_{i,1}^{\pm}}({\bf u},\boldsymbol{\gamma})$ and $\boldsymbol{l_{i,1}^{\pm}}({\bf u},\boldsymbol{\gamma})$ verify

\begin{equation}
\left\{
\begin{array}{l}
\big(\pi_i \mathsf{A}_x^{i,1}-\chi_i^{\pm} \mathsf{I}_{3n}\big)\boldsymbol{r_{i,0}}({\bf u},\boldsymbol{\gamma})=-\big(\mathsf{A}_x({\bf u}^{\boldsymbol{i}},\boldsymbol{\gamma^i})-u_i \mathsf{I}_{3n}\big)\boldsymbol{r_{i,1}^{\pm}}({\bf u},\boldsymbol{\gamma}),\\[6pt]
\boldsymbol{l_{i,0}}({\bf u},\boldsymbol{\gamma}) \big(\pi_i \mathsf{A}_x^{i,1}-\chi_i^{\pm} \mathsf{I}_{3n}\big)=-\boldsymbol{l_{i,1}^{\pm}}({\bf u},\boldsymbol{\gamma})\big(\mathsf{A}_x({\bf u}^{\boldsymbol{i}},\boldsymbol{\gamma^i})-u_i \mathsf{I}_{3n}\big),
\end{array}
\right.
\label{1eigsystem}
\end{equation}
\noindent
Finally, a solution of \eqref{1eigsystem} is
\begin{equation}
\left\{
\begin{array}{l}
\boldsymbol{r_{i,1}^{\pm}}({\bf u},\boldsymbol{\gamma})=\frac{\chi_i^{\pm}}{h_i}\boldsymbol{ e_{n+i}}+\frac{\pi_i-\chi_i^{\pm}}{h_{i+1}}\boldsymbol{ e_{n+i+1}},\\[6pt]
\boldsymbol{l_{i,1}^{\pm}}({\bf u},\boldsymbol{\gamma})=\frac{\chi_i^{\pm}}{h_i} {}^{\top} \boldsymbol{ e_{i}} - +\frac{\pi_i-\chi_i^{\pm}}{h_{i+1}} {}^{\top} \boldsymbol{ e_{i+1}},
\end{array}
\right.
\label{1eigsystemsolution}
\end{equation}
\noindent
and the approximations of the eigenvectors given in proposition \ref{1thmeigvectlambdipm} are verified.
\end{proof}

\noindent
To sum this section up, we succeeded to split the eigenvalues $\lambda_i^{\pm}$ into two distinct ones, for one $i \in [\![1,n-1]\!]$, in the asymptotic $1-\gamma_i \ll 1$ and for all $j \in [\![1,n]\!], |u_j | \ll 1$. Moreover, we managed to get approximations of the corresponding left and right eigenvectors. However, this study was done just for one $i \in [\![1,n]\!]$ and need to be proved for each one to deduce the diagonalizability of $\mathsf{A}({\bf u},\boldsymbol{\gamma},\theta)$ and the local well-posedness of the system \eqref{1systemmultilayer}.

\section{Asymptotic expansion of all the eigenvalues}
In the previous section, a bifurcation of one couple of eigenvalues (associated with one interface liquide/liquid) has been obtained, in the regime of the merger of two layers: for the interface where the density ratio is the closest to $1$, we managed to prove there exist two distinct eigenvalues with distinct eigenvectors. However, this analysis is not possible anymore if all the density ratios tend to $1$, without distinction on how they tend to. In this section, we will prove the expressions of the asymptotic expansions of all the eigenvalues of $\mathsf{A}({\bf u},\boldsymbol{\gamma})$ and give a criterion of hyperbolicity of the system \eqref{1systemmultilayer}, under a regime which distinguish how these density ratios tend to $1$.

\subsection{The asymptotic regime}
In order to get an asymptotic expansion of the eigenvalues and the eigenvectors, it is necessary to assume there exist a small parameter $\epsilon > 0$ and an injective function $\sigma : [\![1,n-1]\!] \rightarrow \mathbb{R}_+^*$ such that for all $i \in [\![1,n-1]\!]$
\begin{equation}
1-\gamma_i=\epsilon^{\sigma(i)}.
\label{1hypasympt}
\end{equation}
\noindent
Without loss of generality, we consider $\epsilon$ is such that
\begin{equation}
\min_{i \in [\![1,n-1]\!]} \sigma(i) = 1.
\end{equation}
\noindent
Moreover, we set the next notations:
\begin{equation}
\left\{
\begin{array}{l}
u_{i+1}-u_i:=\pi_i \epsilon^{\frac{\sigma(i)}{2}},\\[4pt]
h_{i}:=\varpi_i h_{i+1}.
\end{array}
\right.
\label{1hypasympt3}
\end{equation}
Another assumption will be made on the parameters $\boldsymbol{\pi}:=(\pi_i)_{i \in [\![1,n-1]\!]} \in \mathbb{R}^{n-1}$ and $\boldsymbol{\varpi}:=(\varpi_i)_{i \in [\![1,n-1]\!]} \in \mathbb{R}_+^{*\ n-1}$:
\begin{equation}
\forall j \in [\![1,n-1]\!],\
\left\{
\begin{array}{l}
\pi_j^2= \mathcal{O}(h_{j+1} + h_{j}),\\[6pt]
\varpi_j=\mathcal{O}(1).
\end{array}
\right.
\label{1hypasympt2}
\end{equation}

\noindent
{\em Remark:} The assumption on $\boldsymbol{\pi}$ is in agreement with the necessary condition of hyperbolicity \eqref{1corolcaluderiv1}: we expect to get this type of condition for the hyperbolicity of the complete model. However, the assumption on $\boldsymbol{\varpi}$ is a particular case, where there is no preponderant layer.

\noindent
The density-stratification \eqref{1hypasympt} will permit to consider the multi-layer system as the two-layer system. We explain in this section how we figure it out: That is why we define the next subsets of $\mathbb{N}^*$, which provide a partition of $[\![1,n]\!]$:
\begin{equation}
\begin{array}{rl}
\Sigma_i^-&:=\{1 \le j \le i\ / \sigma([\![j,i]\!]) \subset\ [\sigma(i),+\infty[\},\\[4pt]
\Sigma_i^+&:=\{n \ge j > i\ / \sigma([\![i,j-1]\!]) \subset\ [\sigma(i),+\infty[\},\\[4pt]
\overline{\Sigma}_{i,1}^-&:=  \{j \not \in \Sigma_i^- \cup \Sigma_i^+ \ / \ \sigma(j) > \sigma(i)\ \mathrm{and}\ 1 \le j<i\},\\[4pt]
\overline{\Sigma}_{i,1}^+&:=  \{j \not \in \Sigma_i^- \cup \Sigma_i^+ \ / \ \sigma(j) > \sigma(i)\ \mathrm{and}\ n \ge j>i\},\\[4pt]
\overline{\Sigma}_{i,2}^-&:=  \{j \not \in \Sigma_i^- \cup \Sigma_i^+ \ / \ \sigma(j) < \sigma(i)\ \mathrm{and}\ 1 \le j<i\},\\[4pt]
\overline{\Sigma}_{i,2}^+&:=  \{j \not \in \Sigma_i^- \cup \Sigma_i^+ \ / \ \sigma(j) < \sigma(i)\ \mathrm{and}\ n \ge j>i\},
\end{array}
\label{1Sigmai}
\end{equation}
and
\begin{equation}
\left\{
\begin{array}{l}
m_i^-:=\mathrm{min}\ \Sigma_i^-,\\[4pt]
m_i^+:=\mathrm{max}\ \Sigma_i^+.
\end{array}
\right.
\label{1maxminSigmai}
\end{equation}

\noindent
Using the implicit function theorem and assumption \eqref{1hypasympt}, we will prove the eigenvalues associated to the interface $i$, $\lambda_i^{\pm}$, are influenced just by the layers with indices in 
\begin{equation}
\Sigma_i^- \cup \Sigma_i^+=[\![m_i^-,m_i^+]\!].
\end{equation}

\noindent
{\em Remark:} The interpretation of the indices $m_i^-$ and $m_i^+$ is: coming from the interface $i$, the interface $m_i^- -1$ is the first one, above the interface $i$, with a density ratio smaller than $\gamma_i$; the interface $m_i^+$ is the first one, below the interface $i$, with a density ratio smaller than $\gamma_i$:
\begin{equation}
\begin{array}{l}
m_i^-:=\mathrm{max} \{1 \le j \le i / \gamma_j \ge \gamma_i  \} ,\\[4pt]
m_i^+:=\mathrm{min} \{n \ge j > i / \gamma_{j-1} \ge \gamma_i  \}.
\end{array}
\label{1maxminSigmai2}
\end{equation}
\noindent
Then, in respect of the interface $i$, the interface $m_i^- -1$ has the same behavior as a free-surface and the interface $m_i^+$ as a bathymetry.

\subsection{The barotropic eigenvalues}
When all the densities and the velocities are equal, the barotropic eigenvalues degenerate to eigenvalues with simple multiplicity, so the asymptotic expansion is not necessary to prove the diagonalizability of the matrix $\mathsf{A}({\bf u},\boldsymbol{\gamma},\theta)$. However, using classical analysis, we can obtain more accurate expression of these eigenvalues, as it is proved in the next proposition. Thus, we may know the order of the perturbation under the asymptotic regime \eqref{1hypasympt}.

\begin{proposition}
Let $({\bf u},\boldsymbol{\gamma}) \in \mathbb{R}^{3n} \times ]0,1[^{n-1}$ such that $(1-\gamma_j)_{j \in [\![1,n-1]\!]}$ and $(u_j)_{j \in [\![1,n]\!]}$ are sufficiently small. Then, an asymptotic expansion of $\lambda_n^{\pm}({\bf u},\boldsymbol{\gamma})$ is
\begin{equation}
\begin{array}{ll}
\lambda_n^{\pm}({\bf u},\boldsymbol{\gamma})=& \bar{u} \pm \bigg[ \sqrt{H}  - \frac{1}{2 H^{\frac{3}{2}}} \sum_{j=1}^{n-1} (1-\gamma_j) \big( \sum_{k=1}^j h_k \big) \big(\sum_{k=j+1}^n h_k \big) \bigg]\\[10pt]
& +\mathcal{O}(\left((1-\gamma_j)^2\right)_{j \in [\![1,n-1]\!]}, \left((1-\gamma_j)u_k\right)_{(j,k) \in [\![1,n-1]\!] \times [\![1,n]\!]}, \left(u_k^2\right)_{k \in [\![1,n]\!]}).
\end{array}
\label{1approxlambdanpm}
\end{equation}
\label{1proplambdanexpand}
\end{proposition}

\begin{proof}
First, we prove two useful lemmata.
\begin{lemma}
Let $\boldsymbol{\alpha}:=(\alpha_1,\ldots,\alpha_n) \in \mathbb{R}^n$ and $\lambda \in \mathbb{R}$. We consider the matrix
\begin{equation}
\mathsf{N}(\boldsymbol{\alpha},\lambda):=\big[\lambda^2 \delta_{i}^j- \alpha_j ]_{(i,j) \in [\![1,n]\!]^2},
\nonumber
\end{equation}
where $\delta_{i}^j$ is the Kronecker symbol. Then
\begin{equation}
det(\mathsf{N}(\boldsymbol{\alpha},\lambda))= \lambda^{2n-2} \big(\lambda^2-\sum_{i=1}^n \alpha_i \big)
\end{equation}
\label{lemmadet1}
\end{lemma}
\begin{proof}
We define $q(\boldsymbol{\alpha},\lambda):=det(\mathsf{N}(\boldsymbol{\alpha},\lambda))$, which is a polynomial in $\lambda$:
\begin{equation}
g(\boldsymbol{\alpha},\lambda):=\sum_{i=1}^n a_i \lambda^{2i}
\end{equation}
with for all $i \in [\![0,n]\!], a_i = \frac{\partial^{2i} g}{\partial \beta^{2i}}(\boldsymbol{\alpha},0)$. One can prove recursively that
\begin{equation}
a_i=
\left\{
\begin{array}{lr}
0,&  \mathrm{if}\ i \in [\![0,n-2]\!],\\[6pt]
- \sum_{i=1}^n \alpha_i,& \mathrm{if}\  i=n-1,\\[6pt]
1,& \mathrm{if}\ i=n,
\end{array}
\right.
\nonumber
\end{equation}
and the lemma \ref{lemmadet1} is straightforward proved.
\end{proof}

\begin{lemma}
Let $\lambda \in \mathbb{R}^*$, $\boldsymbol{\gamma} \in ]0,1[^{n-1}$ and ${\bf h} \in \mathbb{R}_+^{*\ n}$, then
\begin{equation}
f_n(\lambda,{\bf u^0},\boldsymbol{\gamma})=-\lambda^{2n-2} (h_n \zeta_{n}-\lambda^2)
\label{1f2}
\end{equation}
where the sequence $\big(\zeta_{i}\big)_{i \in \mathbb{N}^*}$ is defined by
\begin{equation}
\left\{
\begin{array}{l}
\zeta_1=1,\\[6pt]
\zeta_{i+1}=\zeta_1+\frac{h_i}{\lambda^2}\zeta_i(\gamma_i-1+\frac{\lambda^2}{h_{i+1}})+\sum_{j=1}^{i-2}\frac{\rho_{j+1} h_j}{\rho_n \lambda^2}(\gamma_{j}-1)\zeta_j.
\end{array}
\right.
\label{1f3}
\end{equation}
\label{lemmalinecolumn2}
\end{lemma}
\begin{proof}
First, we factorize $f_n(\lambda,{\bf u^0},\boldsymbol{\gamma})$ by $(-1)^n\prod_{i=1}^n h_i$ and then we perform the next operations on the columns of $\mathsf{M}_x(\pm \sqrt{H},{\bf u^0},\boldsymbol{\gamma})$, for every $k \in [\![1,n-1]\!]$:
\begin{equation}
\mathrm{C}_k \leftarrow \mathrm{C}_k - \mathrm{C}_{k+1},
\label{1column2}
\end{equation}
and then for $i \in [\![1,n-1]\!]$,
\begin{equation}
\mathrm{C}_n \leftarrow \mathrm{C}_n + \frac{h_i}{\lambda^2}\zeta_i \mathrm{C}_{i},
\label{1column3}
\end{equation}
To finish, as the determinant becomes lower triangular, the lemma \ref{lemmalinecolumn2} is proved.
\end{proof}

\noindent
Then, we verify the next lemma to apply the implicit function theorem:
\begin{lemma}
The barotropic eigenvalues $\lambda_n^{\pm}({\bf u^0},\boldsymbol{\gamma^0}):=\pm \sqrt{H}$ verify
\begin{equation}
\left\{
\begin{array}{l}
f_n(\lambda_n^{\pm}({\bf u^0},\boldsymbol{\gamma^0}),{\bf u^0},\boldsymbol{\gamma^0})=0,\\[6pt]
\frac{\partial f_n}{\partial \lambda}(\lambda_n^{\pm}({\bf u^0},\boldsymbol{\gamma^0}),{\bf u^0},\boldsymbol{\gamma^0})\not=0,
\end{array}
\right.
\label{1derivlambdan}
\end{equation}
\label{1lemmafunctionimplicit}
\end{lemma}

\begin{proof}
According to the lemma \ref{lemmadet1} and the definition of $H$ in \eqref{1H}, it is clear that
\begin{equation}
f_n(\lambda_n^{\pm}({\bf u^0},\boldsymbol{\gamma^0}),{\bf u^0},\boldsymbol{\gamma^0})=0.
\end{equation}
Concerning the derivative, we have
\begin{equation}
\frac{\partial f_n}{\partial \lambda}(\lambda_n^{\pm}({\bf u^0},\boldsymbol{\gamma^0}),{\bf u^0},\boldsymbol{\gamma^0})=\sum_{k=1}^{n} 2 \lambda_n^{\pm}({\bf u^0},\boldsymbol{\gamma^0}) f^k_n(\lambda_n^{\pm}({\bf u^0},\boldsymbol{\gamma^0}),{\bf u^0},\boldsymbol{\gamma^0}).
\end{equation}
Then, with the lemma \ref{lemmadet1} and using the same argument as in \eqref{1fkdet},
\begin{equation}
\frac{\partial f_n}{\partial \lambda}(\lambda_n^{\pm}({\bf u^0},\boldsymbol{\gamma^0}),{\bf u^0},\boldsymbol{\gamma^0})=\pm 2 H^{n-\frac{1}{2}},
\end{equation}
and the lemma \ref{1lemmafunctionimplicit} is proved.
\end{proof}

\noindent
Then, as the lemma \ref{1lemmafunctionimplicit} is verified, it is possible to apply the implicit function theorem to get the approximation of $\lambda_n^{\pm}({\bf u},\boldsymbol{\gamma})$:
\begin{eqnarray}
0& =&(\lambda_n^{\pm}({\bf u},\boldsymbol{\gamma})-\lambda_n^{\pm}({\bf u^0},\boldsymbol{\gamma^0}))\frac{\partial f_n}{\partial \lambda}(\lambda_n^{\pm}({\bf u^0},\boldsymbol{\gamma^0}),{\bf u^0},\boldsymbol{\gamma^0})\nonumber\\[4pt]
& &  +\sum_{j=1}^{n-1}(\gamma_j-1)\frac{\partial f_n}{\partial \gamma_j}(\lambda_n^{\pm}({\bf u^0},\boldsymbol{\gamma^0}),{\bf u^0},\boldsymbol{\gamma^0})\nonumber\\[4pt]
&& +\sum_{j=1}^n u_j\frac{\partial f_n}{\partial u_j}(\lambda_n^{\pm}({\bf u^0},\boldsymbol{\gamma^0}),{\bf u^0},\boldsymbol{\gamma^0})\nonumber\\[4pt]
&&+\mathcal{O}(\left((1-\gamma_j)^2\right)_{j \in [\![1,n-1]\!]}, \left((1-\gamma_j)u_k\right)_{(j,k) \in [\![1,n-1]\!] \times [\![1,n]\!]}, \left(u_k^2\right)_{k \in [\![1,n]\!]}),\nonumber\\
\label{1Taylorlambdan}
\end{eqnarray}

\begin{lemma}
The $1^{\mathrm{st}}$ order partial derivatives are such that
\begin{equation}
\left\{
\begin{array}{llr}
\frac{\partial f_n}{\partial u_j}(\lambda_n^{\pm}({\bf u^0},\boldsymbol{\gamma^0}),{\bf u^0},\boldsymbol{\gamma^0})&=-2\big(\pm H^{\frac{1}{2}}\big) H^{n-2} h_j, &\forall j \in [\![1,n]\!],\\[6pt]
\frac{\partial f_n}{\partial \gamma_j}(\lambda_n^{\pm}({\bf u^0},\boldsymbol{\gamma^0}),{\bf u^0},\boldsymbol{\gamma^0})&= -H^{n-2} \big( \sum_{k=1}^j h_k \big) \big(\sum_{k=j+1}^n h_k \big), &\forall j \in [\![1,n-1]\!].
\end{array}
\right.
\label{1derivativelambdan}
\end{equation}
\label{1lemmaderiveelambdan}
\end{lemma}
\begin{proof}
The expressions of the $1^{\mathrm{st}}$ ones come from
\begin{equation}
\frac{\partial f_n}{\partial u_j}(\lambda_n^{\pm}({\bf u^0},\boldsymbol{\gamma^0}),{\bf u^0},\boldsymbol{\gamma^0})=-2 \lambda_n^{\pm}({\bf u^0},\boldsymbol{\gamma^0}) f_n^j({\bf u^0},\boldsymbol{\gamma^0}),{\bf u^0},\boldsymbol{\gamma^0}),
\end{equation}
the lemma \ref{lemmadet1} and using the same argument as in \eqref{1fkdet}. For the $2^{\mathrm{nd}}$ ones, we define a new sequence:
\begin{equation}
\forall (i,j) \in [\![1,n]\!] \times [\![1,n-1]\!],\ \eta_i^j:= \frac{\partial \zeta_i}{\partial \gamma_j}(\lambda_n^{\pm}({\bf u^0},\boldsymbol{\gamma^0}),{\bf u^0},\boldsymbol{\gamma^0}).
\end{equation}
Then, one can prove that for all $j \in [\![1,n-1]\!]$,
\begin{equation}
\eta_i^j=
\left\{
\begin{array}{ll}
\frac{h_{i-1}}{h_i}\eta_{i-1}^j, &i \in [\![1,k]\!],\\[6pt]
\frac{h_k}{H}\zeta_j+\frac{h_{i-1}}{h_i}\eta_{i-1}^j, &i \in [\![k+1,n]\!].
\end{array}
\right.
\end{equation}
which implies that
\begin{equation}
\forall j \in [\![1,n-1]\!],\ \eta_n^j=\frac{1}{H h_n}\big(\sum_{k=1}^j h_k \big)\big( \sum_{k=j+1}^n h_k \big).
\end{equation}
According to the lemma \ref{lemmalinecolumn2},
\begin{equation}
\frac{\partial f_n}{\partial \gamma_j}(\lambda_n^{\pm}({\bf u^0},\boldsymbol{\gamma^0}),{\bf u^0},\boldsymbol{\gamma^0})=-H^{n-1} h_n \eta_n^j,
\end{equation}
and the lemma \ref{1lemmaderiveelambdan} is proved.
\end{proof}

\noindent
Finally, using lemmata \ref{1lemmafunctionimplicit} and \ref{1lemmaderiveelambdan} in \eqref{1Taylorlambdan}, we have
\begin{eqnarray}
0& =& \pm 2(\lambda_n^{\pm}({\bf u},\boldsymbol{\gamma})-\lambda_n^{\pm}({\bf u^0},\boldsymbol{\gamma^0})) H^{n-\frac{1}{2}}\nonumber\\[4pt]
& &  + H^{n-2} \sum_{j=1}^{n-1}(1-\gamma_j)  \big( \sum_{k=1}^j h_k \big) \big(\sum_{k=j+1}^n h_k \big)\nonumber\\[4pt]
&& -2\big(\pm H^{\frac{1}{2}}\big) H^{n-2}  \sum_{j=1}^n u_j h_j\nonumber\\[4pt]
&&+\mathcal{O}(\left((1-\gamma_j)^2\right)_{j \in [\![1,n-1]\!]}, \left((1-\gamma_j)u_k\right)_{(j,k) \in [\![1,n-1]\!] \times [\![1,n]\!]}, \left(u_k^2\right)_{k \in [\![1,n]\!]}),\nonumber\\
\label{1Taylorlambdan2}
\end{eqnarray}
and the proposition \ref{1proplambdanexpand} is proved.
\end{proof}

\noindent
Then, we can deduce the expressions of the asymptotic expansions of the barotropic eigenvalues, in the particular asymptotic regime \eqref{1hypasympt}.
\begin{proposition}
Let $({\bf u},\boldsymbol{\gamma}) \in \mathbb{R}^{3n} \times ]0,1[^{n-1}$, $\epsilon >0$ and an injective function $\sigma \in \mathbb{R}_+^{*\ [\![1,n-1]\!]}$ such that $\boldsymbol{\gamma}$ verifies \eqref{1hypasympt}, ${\bf u}$ verifies \eqref{1hypasympt2} and
\begin{equation}
\epsilon \ll 1.
\label{1assumpteigenvectorlambdan2}
\end{equation}
Then, an asymptotic expansion of $\lambda_n^{\pm}({\bf u},\boldsymbol{\gamma})$ is
\begin{equation}
\begin{array}{ll}
\lambda_n^{\pm}({\bf u},\boldsymbol{\gamma})=& u_{m_{\sigma}^-}\pm \sqrt{H} +(u_{m_{\sigma}^-+1}-u_{m_{\sigma}^-})\frac{\sum_{k=m_{\sigma}^+ +1}^n h_k}{H}\\[6pt]
& +\mathcal{O}\left(1-\gamma_{m_{\sigma}^-}\right),
\end{array}
\label{1approxlambdanpm3}
\end{equation}
where $m_{\sigma}^- \in [\![1,n-1]\!]$ is defined by
\begin{equation}
\sigma(m_{\sigma}^-)=\min_{j \in [\![1,n-1]\!]} \sigma(j)=1,
\end{equation}
\label{1proplambdanexpandeigenvector2}
\end{proposition}

\noindent
or with other words, it is the indice of the interface liquid/liquid with the biggest density gap.

\begin{proof}
According to the assumption \eqref{1hypasympt}, \eqref{1hypasympt2} and the proposition \ref{1proplambdanexpand}, the expression \eqref{1approxlambdanpm3} is directly deduced,
\begin{equation}
\begin{array}{ll}
\lambda_n^{\pm}({\bf u},\boldsymbol{\gamma})&=u_{m_{\sigma}^-} \pm \sqrt{H}+\psi_{m_{\sigma}^-} \epsilon^{\frac{1}{2}} \pm \big( -h_{\sigma,m_{\sigma}^-}^- h_{\sigma,m_{\sigma}^-}^+ \big) \epsilon +\mathcal{O}(\epsilon^{\frac{3}{2}})\\[6pt]
&=u_{m_{\sigma}^-} \pm \sqrt{H}+\psi_{m_{\sigma}^-} \epsilon^{\frac{1}{2}} +\mathcal{O}(\epsilon),
\end{array}
\label{1reformullambdan22}
\end{equation}
where $\psi_{m_{\sigma}^-}:=\pi_{m_{\sigma}^-}\frac{\sum_{k=m_{\sigma}^+ +1}^n h_k}{H}$. In conclusion, the approximations of the eigenvalues given in proposition \ref{1proplambdanexpandeigenvector2} are verified.
\end{proof}

\noindent
{\em Remark:} The asymptotic expansion of $\lambda_n^{\pm}({\bf u},\boldsymbol{\gamma})$ in proposition \ref{1proplambdanexpand}  corresponds to the asymptotic expansion of $\lambda_1^{\pm}$ in the set of hyperbolicity $|F_x| < F_{crit}^-$ in \cite{monjarret2014local}, in the two-layer case. Moreover, it is in accordance with the expression of the internal eigenvalues in \cite{abgrall2009two}, \cite{barros2008hyperbolicity}, \cite{castro2011numerical}, \cite{kim2008two}, \cite{ovsyannikov1979two},  \cite{schijf1953theoretical} and \cite{stewart2012multilayer}.

To sum this subsection up, we managed to obtain an asymptotic expansion of the barotropic eigenvalues, $\lambda_n^{\pm}({\bf u},\boldsymbol{\gamma})$, with a precision in 
\begin{equation}
\mathcal{O} \left((1-\gamma_j)^2\right)_{j \in [\![1,n-1]\!]}, \left((1-\gamma_j)u_k\right)_{(j,k) \in [\![1,n-1]\!] \times [\![1,n]\!]}, \left(u_k^2\right)_{k \in [\![1,n]\!]}.
\end{equation}

\noindent
Moreover, we gave the asymptotic expansion, with the assumptions \eqref{1hypasympt}, \eqref{1hypasympt2} and \eqref{1assumpteigenvectorlambdan2}, with a precision about $\mathcal{O}(\epsilon)$.

\noindent
In the next subsection, we prove the expression of asymptotic expansion of the baroclinic eigenvalues and give a criterion of hyperbolicity, in the asymptotic regime \eqref{1hypasympt}.

\subsection{The baroclinic eigenvalues}
In the proposition \ref{1thmapproxlambdaipm}, we have proved the asymptotic expansion of the eigenvalues associated to an interface where the layers just above and below are almost merged. We prove, in this subsection, the asymptotic expansion of the baroclinic eigenvalues, for each interface.
\begin{proposition}
Let $({\bf u},\boldsymbol{\gamma}) \in \mathbb{R}^{3n} \times ]0,1[^{n-1}$, $\epsilon >0$ and an injective function $\sigma \in \mathbb{R}_+^{*\ [\![1,n-1]\!]}$ such that $\boldsymbol{\gamma}$ verifies \eqref{1hypasympt}, ${\bf u}$ verifies \eqref{1hypasympt2} and
\begin{equation}
\epsilon \ll 1.
\end{equation}
Then, for all $i \in [\![1,n-1]\!]$, the asymptotic expansion of $\lambda_i^{\pm}({\bf u},\boldsymbol{\gamma})$ is
\begin{equation}
\begin{array}{ll}
\lambda_i^{\pm}({\bf u},\boldsymbol{\gamma})=&\frac{u_{i+1}h_{\sigma,i}^-+u_i h_{\sigma,i}^+}{h_{\sigma,i}^- + h_{\sigma,i}^+} \pm \left[ \frac{h_{\sigma,i}^- h_{\sigma,i}^+}{(h_{\sigma,i}^- + h_{\sigma,i}^+)}\left(1-\gamma_i-\frac{(u_{i+1}-u_i)^2}{h_{\sigma,i}^- + h_{\sigma,i}^+}\right) \right]^{\frac{1}{2}}\\[7pt]
&+\mathcal{O}(\epsilon^{\frac{\sigma(i)+1}{2}})
\end{array}
\label{1lambdaipmasympt}
\end{equation}
where $h_{\sigma,i}^-:=\sum_{k=m_{i}^-}^i h_k$ and $h_{\sigma,i}^+:=\sum_{k=i+1}^{m_i^+} h_k$.
\label{1proplambdaipmapprox}
\end{proposition}

\noindent
{\em Remark:} $h_{\sigma,i}^{\pm}$ are the upper and lower layers influencing $\lambda_i^{\pm}$.

\begin{proof}
Let $i \in [\![1,n-1]\!]$, according to the corollary \ref{1corolcaluderiv1}, the eigenvalue $\lambda_i^{\pm}$ is assumed as
\begin{equation}
\left\{
\begin{array}{l}
\lambda_i^{\pm} - u_i := \tilde{\lambda}_i^{\pm} \epsilon^{\frac{\sigma(i)}{2}},\\[6pt]
\tilde{\lambda}_i^{\pm}=\mathcal{O}(\varpi_i).
\label{1rangelambdai}
\end{array}
\right.
\end{equation}

\noindent
First of all, we need to evaluate the order of each term of $\mathrm{det}(\mathsf{M}_x(\lambda_i^{\pm},{\bf u},\boldsymbol{\gamma}))$.
The next operations are performed to the columns of the determinant, without changing its value:
\begin{equation}
\forall j \in [\![1,n-1]\!],\ \mathrm{C}_j \leftarrow \mathrm{C}_j - \varpi_j \mathrm{C}_{j+1}.
\end{equation}
\noindent
Then, for all $j \in [\![1,n-1]\!]$, the new column $\mathrm{C}_j$ is expressed in $(\boldsymbol{f_i})_{i \in [|1,n|]}$, the canonical basis of $\mathbb{R}^n$
\begin{equation}
\left\{
\begin{array}{ll}
\mathrm{C}_j=&
\begin{array}{l}
h_j(\lambda_i^{\pm}-u_j)^2 \boldsymbol{f_j}+[h_j (1-\gamma_j)-\varpi_j(\lambda_i^{\pm}-u_{j+1})^2]\boldsymbol{f_{j+1}}\\
+h_j\sum_{k=j+2}^n \alpha_{k,j+1} (1-\gamma_j)\boldsymbol{f_k},
\end{array}\\[14pt]
\mathrm{C}_n=&h_n(\lambda_i^{\pm}-u_n)^2\boldsymbol{f_n}-h_n \sum_{k=1}^{n} \boldsymbol{f_k}.
\end{array}
\right.
\label{1Cj}
\end{equation}
Then, for all $j \in [\![1,n]\!]$, we denote by $o(i,j,\sigma)$, the order of the terms of the column $\mathrm{C}_j$:
\begin{equation}
\mathrm{C}_j = \mathcal{O}(\epsilon^{o(i,j,\sigma)}).
\end{equation}
\noindent
We provide the expression of $o(i,j,\sigma)$ in the next lemma:
\begin{lemma}
Let $(i,j) \in [\![1,n]\!] \times [\![1,n-1]\!]$,
\begin{equation}
o(i,j,\sigma)=
\left\{
\begin{array}{lr}
\sigma(i), & \mathrm{if}\ j \in \Sigma_i^- \cup \Sigma_i^+,\\[6pt]
\sigma_{i,j}^-,&  \mathrm{if}\ j \in \overline{\Sigma}_{i,1}^- \cup \overline{\Sigma}_{i,2}^-,\\[6pt]
\sigma_{i,j}^+,&  \mathrm{if}\ j \in \overline{\Sigma}_{i,1}^+ \cup \overline{\Sigma}_{i,2}^+,
\end{array}
\right.
\label{1minCj}
\end{equation}
\noindent
and
\begin{equation}
o(i,n,\sigma)=0,
\end{equation}
where for all $j \le i$,
\begin{equation}
\sigma_{i,j}^-:=\mathrm{min}\{ \sigma(k) / k \in [\![j,i]\!]\},
\end{equation}
and for all $j\ge i$,
\begin{equation}
\sigma_{i,j}^+:=\mathrm{min}\{ \sigma(k) / k \in [\![i,j]\!]\}.
\end{equation}
\label{1lemmijsigma}
\end{lemma}
\begin{proof}
According to the expression of $\mathrm{C}_n$ in \eqref{1minCj}, it is clear that the order of $\mathrm{C}_n$ is 1
\begin{equation}
o(i,n,\sigma)=0,
\end{equation}
\noindent
Moreover, we analyse each term of $\mathrm{C}_j$, for all $j \in [\![1,n-1]\!]$:
\begin{equation}
\forall k \ge j+2,\ h_j \alpha_{k,j+1} (1-\gamma_j) \sim h_j (1-\gamma_j),
\label{1orderCj1}
\end{equation}

\begin{equation}
\lambda_i^{\pm}-u_j \sim  \left\{
\begin{array}{lr}
 \lambda_i^{\pm}-u_i, & \mathrm{if}\ j \in \Sigma_i^-, \\[6pt]
 \lambda_i^{\pm}-u_{i+1}, & \mathrm{if}\ j \in \Sigma_i^+, \\[6pt]
 u_{\beta_{i,j}^- +1}-u_{\beta_{i,j}^-} , & \mathrm{if}\ j  \in \overline{\Sigma}_{i,1}^- \cup \overline{\Sigma}_{i,2}^-,\\[6pt]
 u_{\beta_{i,j}^+}-u_{\beta_{i,j}^+ +1} , & \mathrm{if}\ j  \in \overline{\Sigma}_{i,1}^+ \cup \overline{\Sigma}_{i,2}^+,
\end{array}
\right.
\label{1orderCj2}
\end{equation}

\begin{equation}
s_{i,j}
\sim \left\{
\begin{array}{lr}
 -\varpi_j(\lambda_i^{\pm}-u_{i})^2 , & \mathrm{if}\ j \in \Sigma_i^- \backslash \{i\},\\[6pt]
-\varpi_j (\lambda_i^{\pm}-u_{i+1})^2, & \mathrm{if}\ j \in \Sigma_i^+ \backslash \{m_i^+\},\\[6pt]
-\varpi_j (u_{\beta_{i,j+1}^- +1}-u_{\beta_{i,j+1}^-})^2, &\mathrm{if}\ j \in \overline{\Sigma}_{i,1}^-,\\[6pt]
-\varpi_j(u_{\beta_{i,j}^+ +1}-u_{\beta_{i,j}^+})^2, & \mathrm{if}\ j \in \overline{\Sigma}_{i,1}^+,\\[6pt]
\delta_{\beta_{i,j}^-}^{j} h_j(1-\gamma_j)-\delta_{\beta_{i,j}^-}^{\beta_{i,j+1}^-}\varpi_j (u_{\beta_{i,j+1}^- +1}-u_{\beta_{i,j+1}^-})^2, &\mathrm{if}\ j \in \overline{\Sigma}_{i,2}^-,\\[6pt]
\delta_{\beta_{i,j}^+}^{j} h_j(1-\gamma_j)-\varpi_j(u_{\beta_{i,j}^+ +1}-u_{\beta_{i,j}^+})^2, & \mathrm{if}\ j \in \overline{\Sigma}_{i,2}^+,\\[6pt]
h_j (1-\gamma_j)-\varpi_j(\lambda_i^{\pm}-u_{j+1})^2, & \mathrm{if}\ j=i,\\[6pt]
h_j (1-\gamma_j)-\varpi_j(u_j-u_{j+1})^2, & \mathrm{if}\ j=m_i^+,
\end{array}
\right.
\label{1orderCj3}
\end{equation}
where $\delta_{i}^{j}$ is the Kronecker symbol, $s_{i,j}$ is defined by
\begin{equation}
s_{i,j}:=h_j (1-\gamma_j) -\varpi_j(\lambda_i^{\pm}-u_{j+1})^2,  
\end{equation}
\noindent
and $\beta_{i,j}^{\pm} \in [\![1,n]\!]$ are defined such that
\begin{equation}
\sigma(\beta_{i,j}^{\pm}):=\sigma_{i,j}^{\pm}.
\end{equation} 
Then, according to \eqref{1hypasympt}, \eqref{1hypasympt2} and \eqref{1rangelambdai},
\begin{equation}
\forall k \ge j+2,\ h_j \alpha_{k,j+1} (1-\gamma_j) \sim h_j \epsilon^{\sigma(j)},
\label{1orderCj12}
\end{equation}
\begin{equation}
\lambda_i^{\pm}-u_j \sim  \left\{
\begin{array}{lr}
\tilde{\lambda}_i^{\pm} \epsilon^{\frac{\sigma(i)}{2}},& \mathrm{if}\ j \in \Sigma_i^-, \\[6pt]
(\tilde{\lambda}_i^{\pm}-\pi_i) \epsilon^{\frac{\sigma(i)}{2}}, & \mathrm{if}\ j \in \Sigma_i^+, \\[6pt]
 \pi_{\beta_{i,j}^-} \epsilon^{\frac{\sigma_{i,j}^-}{2}}, & \mathrm{if}\ j  \in \overline{\Sigma}_{i,1}^- \cup \overline{\Sigma}_{i,2}^-, \\[6pt]
-\pi_{\beta_{i,j}^+} \epsilon^{\frac{\sigma_{i,j}^+}{2}}, & \mathrm{if}\ j  \in \overline{\Sigma}_{i,1}^+ \cup \overline{\Sigma}_{i,2}^+,
\end{array}
\right.
\label{1orderCj22}
\end{equation}

\begin{equation}
s_{i,j}
\sim \left\{
\begin{array}{lr}
 -\varpi_j (\tilde{\lambda}_i^{\pm})^2 \epsilon^{\sigma(i)}, & \mathrm{if}\ j \in \Sigma_i^- \backslash \{i\},\\[6pt]
-\varpi_j (\tilde{\lambda}_i^{\pm}-\pi_i)^2 \epsilon^{\sigma(i)} & \mathrm{if}\ j \in \Sigma_i^+ \backslash \{m_i^+\},\\[6pt]
-\varpi_j \pi^2_{\beta_{i,j+1}^-} \epsilon^{\sigma_{i,j}^-}, &\mathrm{if}\ j \in \overline{\Sigma}_{i,1}^-,\\[6pt]
-\varpi_j \pi^2_{\beta_{i,j}^+} \epsilon^{\sigma_{i,j}^+}, & \mathrm{if}\ j \in \overline{\Sigma}_{i,1}^+,\\[6pt]
\delta_{\beta_{i,j}^-}^{j} h_j \epsilon^{\sigma(j)}-\delta_{\beta_{i,j}^-}^{\beta_{i,j+1}^-}\varpi_j \pi^2_{\beta_{i,j+1}^-} \epsilon^{\sigma_{i,j}^-}, &\mathrm{if}\ j \in \overline{\Sigma}_{i,2}^-,\\[6pt]
\delta_{\beta_{i,j}^+}^{j} h_j \epsilon^{\sigma(j)} -\varpi_j \pi^2_{\beta_{i,j}^+} \epsilon^{\sigma_{i,j}^+}, & \mathrm{if}\ j \in \overline{\Sigma}_{i,2}^+,\\[6pt]
(h_j- \varpi_j(\tilde{\lambda}_i^{\pm}-\pi_j)^2 )\epsilon^{\sigma(i)}, & \mathrm{if}\ j=i,\\[6pt]
(h_j - \varpi_j\pi_j^2) \epsilon^{\sigma(j)}, & \mathrm{if}\ j=m_i^+.
\end{array}
\right.
\label{1orderCj32}
\end{equation}
Finally, using that
\begin{equation}
\forall (i,j) \in [\![1,n-1]\!]^2,
\left\{
\begin{array}{l}
\sigma_{i,j}^- \le \sigma(j),\\[4pt]
\sigma_{i,j}^+ \le \sigma(j),\\[4pt]
\sigma_{i,m_i^- -1}^- = \sigma(m_i^- -1),
\end{array}
\right.
\end{equation}
the lemma \ref{1lemmijsigma} is proved.
\end{proof}

\noindent
Afterwards, we define for all $j \in [\![1,n]\!]$, $\tilde{\mathrm{C}}_j:=\frac{\mathrm{C}_j}{\epsilon^{o(i,j,\sigma)}} \big|_{\epsilon=0}$. If $j \in [\![1,n-1]\!]$, $\tilde{\mathrm{C}}_j$  is equal to 
\begin{equation}
\left\{
\begin{array}{lr}
(\tilde{\lambda}_i^{\pm})^2(\boldsymbol{f_j}-\varpi_j \boldsymbol{f_{j+1}}), & \mathrm{if}\ j \in \Sigma_i^-\backslash \{i\}, \\[3pt]
(\tilde{\lambda}_i^{\pm}-\pi_i)^2(\boldsymbol{f_j}-\varpi_j \boldsymbol{f_{j+1}}), & \mathrm{if}\ j \in \Sigma_i^+ \backslash \{m_i^+\},\\[3pt]
(\pi_{\beta_{i,j}^-})^2\boldsymbol{f_j}-\delta_{\sigma_{i,j}^-}^{\sigma_{i,j+1}^-}\varpi_j (\pi_{\beta_{i,j+1}^-})^2\boldsymbol{f_{j+1}}, & \mathrm{if}\ j \in \overline{\Sigma}_{i,1}^-,\\[10pt]
\delta_{\sigma_{i,j}^+}^{\sigma_{i,j+1}^+}(\pi_{\beta_{i,j-1}^+})^2\boldsymbol{f_j}-\varpi_j (\pi_{\beta_{i,j+1}^+})^2\boldsymbol{f_{j+1}}, & \mathrm{if}\ j \in \overline{\Sigma}_{i,1}^+,\\[10pt]
(\pi_{\beta_{i,j}^-})^2\boldsymbol{f_j}-\delta_{\sigma_{i,j}^-}^{\sigma_{i,j+1}^-} \varpi_j (\pi_{\beta_{i,j+1}^-})^2\boldsymbol{f_{j+1}} +\delta_{\sigma_{i,j}^-}^{\sigma(j)} h_j \sum_{k=j+1}^n \boldsymbol{f_k}, & \mathrm{if}\ j \in \overline{\Sigma}_{i,2}^-,\\[10pt]
\delta_{\sigma_{i,j}^+}^{\sigma_{i,j+1}^+}(\pi_{\beta_{i,j}^+})^2\boldsymbol{f_j}-\varpi_j (\pi_{\beta_{i,j+1}^+})^2\boldsymbol{f_{j+1}} +\delta_{\sigma_{i,j}^+}^{\sigma(j)} h_j \sum_{k=j+1}^n \boldsymbol{f_k}, & \mathrm{if}\ j \in \overline{\Sigma}_{i,2}^+,\\[10pt]
(\tilde{\lambda}_i^{\pm})^2\boldsymbol{f_j}-\varpi_j(\tilde{\lambda}_i^{\pm}-\pi_i)^2\boldsymbol{f_{j+1}}+h_j \sum_{k=j+1}^n \boldsymbol{f_k} ,& \mathrm{if}\ j=i,\\[3pt]
-\varpi_j \pi_j^2\boldsymbol{f_{j+1}}+h_j \sum_{k=j+1}^n \boldsymbol{f_k} ,& \mathrm{if}\ j=m_i^+,\\
\end{array}
\right.
\label{1factoCj}
\end{equation}
and if $j=n$, $\tilde{\mathrm{C}}_j$ is equal to
\begin{equation}
\tilde{\mathrm{C}}_n=h_n \sum_{k=1}^n \boldsymbol{f_k}.
\label{1factoCj2}
\end{equation}
Then, to every the column $\tilde{\mathrm{C}}_j$, with $j \in [\![1,n-1]\!]$ such that one of the following conditions is verified:
\begin{equation}
\left\{
\begin{array}{l}
j=i,\\[6pt]
j=m_i^+,\\[6pt]
j \in \overline{\Sigma}_{i,2}^-\ \mathrm{and}\ \sigma_{i,j}^-=\sigma(j),
\end{array}
\right.
\end{equation}
we perform the next operations:
\begin{equation}
\tilde{\mathrm{C}}_j \leftarrow \tilde{\mathrm{C}}_j - \frac{h_j}{h_n} \tilde{\mathrm{C}}_n.
\label{1opcjtild}
\end{equation}

\noindent
We define
\begin{equation}
\begin{array}{ll}
\tilde{g}(\tilde{\lambda}_i^{\pm},\boldsymbol{\pi},\boldsymbol{\varpi},\sigma)&:=\frac{g(\lambda_i^{\pm},\boldsymbol{u},\boldsymbol{\gamma})}{\epsilon^{\sigma_i}}\big|_{\epsilon=0}\\[4pt]
&=\mathrm{det}((\tilde{\mathrm{C}}_j)_{j \in [\![1,n]\!]}),
\end{array}
\end{equation}
where $\sigma_i:=\sum_{j=1}^{n} o(i,j,\sigma)$, $\boldsymbol{\pi}:=(\pi_i)_{i \in [\![1,n-1]\!]}$ and $\boldsymbol{\varpi}:=(\varpi_i)_{i \in [\![1,n-1]\!]}$. Then, according to \eqref{1factoCj}, \eqref{1factoCj2} and \eqref{1opcjtild}, the determinant $\tilde{g}(\tilde{\lambda}_i^{\pm},\boldsymbol{\pi},\boldsymbol{\varpi},\sigma)$ is under the following form:
\begin{equation}
\begin{array}{ll}
\tilde{g}(\tilde{\lambda}_i^{\pm},\boldsymbol{\pi},\boldsymbol{\varpi},\sigma)=\left|
\begin{array}{c|c|c}
 \begin{array}{c} \mathsf{\Delta^1}(\boldsymbol{\pi},\boldsymbol{\varpi},\sigma)  \end{array}
 & \begin{array}{c} \mathsf{\Omega^1}(\boldsymbol{\pi},\boldsymbol{\varpi},\sigma) \end{array}
 & \begin{array}{c} \mathsf{0} \end{array}\\[2pt]
 \hline
 \begin{array}{c} \mathsf{0} \end{array}
 & \begin{array}{c} \mathsf{\Lambda}(\tilde{\lambda}_i^{\pm},\pi_i,\boldsymbol{\varpi},\sigma) \end{array}
 & \begin{array}{c} \mathsf{0} \end{array}\\[2pt]
 \hline
 \begin{array}{c} \mathsf{0} \end{array}
 & \begin{array}{c} \mathsf{\Omega^2}(\boldsymbol{\pi},\boldsymbol{\varpi},\sigma) \end{array}
 & \begin{array}{c} \mathsf{\Delta^2}(\boldsymbol{\pi},\boldsymbol{\varpi},\sigma) \end{array}\\
\end{array}
\right|
\end{array},
\nonumber
\end{equation}
where $\mathsf{\Lambda}(\tilde{\lambda}_i^{\pm},\pi_i,\boldsymbol{\varpi},\sigma)$, $\mathsf{\Delta^1}(\boldsymbol{\pi},\boldsymbol{\varpi},\sigma)$ and $\mathsf{\Delta^2}(\boldsymbol{\pi},\boldsymbol{\varpi},\sigma)$ are square matrices with respective dimensions $m_i^+-m_i^- +1$, $m_i^- -1$ and $n-m_i^+$; $\mathsf{\Omega^1}(\boldsymbol{\pi},\boldsymbol{\varpi},\sigma)$ and $\mathsf{\Omega^2}(\boldsymbol{\pi},\boldsymbol{\varpi},\sigma)$ are rectangular matrices with respective dimensions $m_i^- -1 \times m_i^+-m_i^- +1 $ and $n-m_i^+ \times m_i^+-m_i^-$.

\noindent
Then, it is clear that
\begin{equation}
\tilde{g}(\tilde{\lambda}_i^{\pm},\boldsymbol{\pi},\boldsymbol{\varpi},\sigma)= \det \big( \mathsf{\Delta^1}(\boldsymbol{\pi},\boldsymbol{\varpi},\sigma) \big) \det \big( \mathsf{\Lambda}(\tilde{\lambda}_i^{\pm},\pi_i,\boldsymbol{\varpi},\sigma) \big) \det \big( \mathsf{\Delta^2}(\boldsymbol{\pi},\boldsymbol{\varpi},\sigma) \big).
\end{equation}

\noindent
The important point of this proof is there is just $\mathsf{\Lambda}(\tilde{\lambda}_i^{\pm},\pi_i,\boldsymbol{\varpi},\sigma)$ which depends of $\tilde{\lambda}_i^{\pm}$, therefore it is necessary to find the solution, $\tilde{\lambda}_i^{\pm}$, such that
\begin{equation}
\det \big( \mathsf{\Lambda}(\tilde{\lambda}_i^{\pm},\pi_i,\boldsymbol{\varpi},\sigma) \big)=0.
\end{equation}
where, according to the previous analysis, the columns of $\det\big(\mathsf{\Lambda}(\tilde{\lambda}_i^{\pm},\pi_i,\boldsymbol{\varpi},\sigma)\big)$ are such that for all $j \in [\![1,m_i^+-m_i^- +1]\!]$,
\begin{equation}
\mathrm{C}_j\big( \mathsf{\Lambda} \big)=
\left\{
\begin{array}{lr}
(\tilde{\lambda}_i^{\pm})^2(\boldsymbol{f_j}-\varpi_j \boldsymbol{f_{j+1}}), & \mathrm{if}\ m_i^- \le j_{i} \le  i-1, \\[6pt]
(\tilde{\lambda}_i^{\pm}-\pi_i)^2(\boldsymbol{f_j}-\varpi_j \boldsymbol{f_{j+1}}), & \mathrm{if}\ i+1 \le j_{i} \le m_i^+ -1,\\[6pt]
(\tilde{\lambda}_i^{\pm})^2\boldsymbol{f_i}-\varpi_i(\tilde{\lambda}_i^{\pm}-\pi_i)^2\boldsymbol{f_{i+1}}-h_i \sum_{k=m_i^-}^i \boldsymbol{f_k} ,& \mathrm{if}\ j_{i}=i,\\[6pt]
-h_j \sum_{j=1}^{m_i^+} \boldsymbol{f_{k}},& \mathrm{if}\ j_{i}=m_i^+,
\end{array}
\right.
\label{1columnsLambda}
\end{equation}
where $j_i:=j+m_i^- - 1$.

\begin{lemma}
Let $(\pi_i,\boldsymbol{h}) \in \mathbb{R} \times \mathbb{R}^{n}$ and an injective function $\sigma \in \mathbb{R}_+^{*\ [\![1,n-1]\!]}$. Then, $\tilde{\lambda} \in \mathbb{R}$ is solution of
\begin{equation}
\det \big( \mathsf{\Lambda}(\tilde{\lambda},\pi_i,\boldsymbol{\varpi},\sigma) \big)=0.
\label{1detequation}
\end{equation}
if and only if
\begin{equation}
\tilde{\lambda} \in \bigg\{0,\pi_i,\frac{\pi_i h_{\sigma,i}^-}{h_{\sigma,i}^- + h_{\sigma,i}^+} \pm \left[\frac{h_{\sigma,i}^- h_{\sigma,i}^+}{(h_{\sigma,i}^- + h_{\sigma,i}^+)^2}\left(h_{\sigma,i}^- + h_{\sigma,i}^+-\pi_i^2\right) \right]^{\frac{1}{2}}\bigg\},
\label{1lambda}
\end{equation}
where $h_{\sigma,i}^-:=h_i(1+ \sum_{k=m_i^-}^{i-1} \prod_{j=k}^{i-1} \varpi_j)$ and $h_{\sigma,i}^+:=h_{m_i^+}(1+ \sum_{k=i+1}^{m_i^+-1} \prod_{j=k}^{m_i^+ -1} \varpi_j)$. Moreover, the respective multiplicities are $2(i-m_i^-)$, $2(m_i^+-i-1)$ and $1$.
\end{lemma}
\begin{proof}
As $\mathrm{C}_n \big(\mathsf{\Lambda}(\tilde{\lambda},\pi_i,\boldsymbol{\varpi},\sigma) \big)$ does not depend of $\tilde{\lambda}$, $\det \big( \mathsf{\Lambda}(\tilde{\lambda},\pi_i,\boldsymbol{\varpi},\sigma) \big)$ is a polynomial in $\tilde{\lambda}$, with a degree equal to $2(m_i^+-m_i^-)$. According to \eqref{1columnsLambda}, $0$ and $\pi_i$ are two roots, with respective multiplicity $2(i-m_i^-)$ and $2(m_i^+-i-1)$. To determine the expression of the other roots, it is sufficient to perform the next operations on two columns of $\det\big(\mathsf{\Lambda}(\tilde{\lambda}_i^{\pm},\pi_i,\boldsymbol{\varpi},\sigma)\big)$, assuming that $\tilde{\lambda} \not \in \{0,\pi_i\}$:
\begin{equation}
\forall j \in [\![m_i^-,i-1]\!],
\left\{
\begin{array}{l}
\mathrm{C}_i \leftarrow \mathrm{C}_i + \frac{h_i \big(1+ \sum_{k=1}^{j-1} \prod_{q=k}^{j-1} \varpi_q\big)}{\tilde{\lambda}^2} \mathrm{C}_j,\\[6pt]
\mathrm{C}_{m_i^+} \leftarrow \mathrm{C}_{m_i^+} + \frac{h_n \big(1+ \sum_{k=1}^{j-1} \prod_{q=k}^{j-1} \varpi_q\big)}{\tilde{\lambda}^2} \mathrm{C}_j,
\end{array}
\right.
\end{equation}
and afterwards
\begin{equation}
\forall j \in [\![i+1,m_i^+-1]\!],
\mathrm{C}_{m_i^+} \leftarrow \mathrm{C}_{m_i^+} + \frac{h_i \big(1+ \sum_{k=1}^{j-1} \prod_{q=k}^{j-1} \varpi_q\big)}{(\tilde{\lambda}-\pi_i)^2} \mathrm{C}_j.
\end{equation}
Therefore, for all $j \in [\![1,m_i^+-m_i^-+1]\!]$, the new column of $\det\big(\mathsf{\Lambda}(\tilde{\lambda}_i^{\pm},\pi_i,\boldsymbol{\varpi},\sigma)\big)$, $\mathrm{C}_j$, is equal to
\begin{equation}
\left\{
\begin{array}{lr}
(\tilde{\lambda}_i^{\pm})^2(\boldsymbol{f_j}-\varpi_j \boldsymbol{f_{j+1}}), & \mathrm{if}\ m_i^- \le j_i \le  i-1, \\[6pt]
(\tilde{\lambda}_i^{\pm}-\pi_i)^2(\boldsymbol{f_j}-\varpi_j \boldsymbol{f_{j+1}}), & \mathrm{if}\ i+1 \le j_i \le m_i^+ -1,\\[6pt]
((\tilde{\lambda}_i^{\pm})^2 -h_{\sigma,i}^- +)\boldsymbol{f_i}-\varpi_i(\tilde{\lambda}_i^{\pm}-\pi_i)^2\boldsymbol{f_{i+1}},& \mathrm{if}\ j_i=i,\\[6pt]
-\frac{h_{m_i^+}}{h_i} h_{\sigma,i}^- \boldsymbol{f_{i}}  - h_{\sigma,i}^+ \boldsymbol{f_{m_i^+}},& \mathrm{if}\ j_i=m_i^+,
\end{array}
\right.
\end{equation}
where $j_i:=j+m_i^- - 1$. An expansion of the determinant about the last column, $\mathrm{C}_{m_i^+}$, provides
\begin{equation}
\det \big( \mathsf{\Lambda}(\tilde{\lambda},\pi_i,\boldsymbol{\varpi},\sigma) \big)= - \tilde{\lambda}^{2i-2} (\tilde{\lambda}-\pi_i)^{2(m_i^+-i-1)} \big[ h_{\sigma,i}^-(\tilde{\lambda}-\pi_i)^2 + h_{\sigma,i}^+ (\tilde{\lambda}^2 - h_{\sigma,i}^-) \big].
\end{equation}
\noindent
Finally, the only solutions of \eqref{1detequation}, different from  $0$ and $\pi_i$ are:
\begin{equation}
\tilde{\lambda}_i^{\pm} =\frac{\pi_i h_{\sigma,i}^-}{h_{\sigma,i}^- + h_{\sigma,i}^+} \pm \left[\frac{h_{\sigma,i}^- h_{\sigma,i}^+}{(h_{\sigma,i}^- + h_{\sigma,i}^+)^2}\left(h_{\sigma,i}^- + h_{\sigma,i}^+-\pi_i^2\right) \right]^{\frac{1}{2}},
\label{1lambdaipmtilde}
\end{equation}
with multiplicity equal to $1$.
\end{proof}

\noindent
Consequently, according to the lemma \ref{1lambda}
\begin{equation}
\tilde{g}(\tilde{\lambda},\boldsymbol{\pi},\boldsymbol{\varpi},\sigma)=0,
\end{equation}
if and only if
\begin{equation}
\tilde{\lambda} \in \bigg\{0,\pi_i,\frac{\pi_i h_{\sigma,i}^-}{h_{\sigma,i}^- + h_{\sigma,i}^+} \pm \left[\frac{h_{\sigma,i}^- h_{\sigma,i}^+}{(h_{\sigma,i}^- + h_{\sigma,i}^+)^2}\left(h_{\sigma,i}^- + h_{\sigma,i}^+-\pi_i^2\right) \right]^{\frac{1}{2}}\bigg\}.
\label{1lambda2}
\end{equation}

\noindent
According to the implicit functions theorem,
\begin{equation}
g(\lambda,\boldsymbol{u},\boldsymbol{\gamma})=0,
\end{equation}
if and only if $\lambda -  \mathcal{O}(\epsilon^{\frac{\sigma(i)+1}{2}})$ is in
\begin{equation}
\begin{array}{ll}
&\bigg\{u_i,u_i+\pi_i \epsilon^{\frac{\sigma(i)}{2}},u_i+\bigg(\frac{\pi_i h_{\sigma,i}^-}{h_{\sigma,i}^- + h_{\sigma,i}^+} \pm \left[\frac{h_{\sigma,i}^- h_{\sigma,i}^+}{(h_{\sigma,i}^- + h_{\sigma,i}^+)^2}\left(h_{\sigma,i}^- + h_{\sigma,i}^+-\pi_i^2\right) \right]^{\frac{1}{2}}\bigg)\epsilon^{\frac{\sigma(i)}{2}}\bigg\}\\[10pt]
&=\bigg\{u_i,u_{i+1},\frac{u_{i+1} h_{\sigma,i}^- + u_i h_{\sigma,i}^+}{h_{\sigma,i}^- + h_{\sigma,i}^+} \pm \left[\frac{h_{\sigma,i}^- h_{\sigma,i}^+}{h_{\sigma,i}^- + h_{\sigma,i}^+}\left(1-\gamma_i-(u_{i+1}-u_i)^2\right) \right]^{\frac{1}{2}}\bigg\}.
\end{array}
\label{1lambda3}
\end{equation}
$\lambda \in \{u_i + \mathcal{O}(\epsilon^{\frac{\sigma(i)+1}{2}}),u_{i+1}+ \mathcal{O}(\epsilon^{\frac{\sigma(i)+1}{2}})\}$ corresponds to the merger of layers and does not provide the correct roots because the multiplicities are not equal to $1$. That is why we chose
\begin{equation}
\lambda=\lambda_i^{\pm}:=\frac{u_{i+1} h_{\sigma,i}^- + u_i h_{\sigma,i}^+}{h_{\sigma,i}^- + h_{\sigma,i}^+} \pm \left[\frac{h_{\sigma,i}^- h_{\sigma,i}^+}{h_{\sigma,i}^- + h_{\sigma,i}^+}\left(1-\gamma_i-(u_{i+1}-u_i)^2\right) \right]^{\frac{1}{2}}+ \mathcal{O}(\epsilon^{\frac{\sigma(i)+1}{2}}),
\end{equation}
which provides two roots, with multiplicity equal to $1$, and the proposition \ref{1proplambdaipmapprox} is proved.
\end{proof}

\noindent
{\em Remark:} The asymptotic expansion of $\lambda_i^{\pm}({\bf u},\boldsymbol{\gamma})$ in proposition \ref{1proplambdanexpand}  corresponds to the asymptotic expansion of $\lambda_2^{\pm}$ in the set of hyperbolicity $|F_x| < F_{crit}^-$ in \cite{monjarret2014local}, in the two-layer case. Moreover, it is also in accordance with \cite{abgrall2009two}, \cite{barros2008hyperbolicity}, \cite{castro2011numerical}, \cite{kim2008two}, \cite{ovsyannikov1979two},  \cite{schijf1953theoretical} and \cite{stewart2012multilayer}. 2) In the oceanographic applications, the {\em French Naval Hydrographic and Oceanographic Service} uses the multi-layer shallow water model with $40$ layers. For instance, in the bay of Biscay, the assumption \eqref{1hypasympt} is verified, with
\begin{equation}
\epsilon \simeq 10^{-4}.
\end{equation}

\noindent
However, the matter is that
\begin{equation}
\forall i \in [\![1,n-1]\!],\ 1 \le \sigma(i) \le 2,
\end{equation}
\noindent
which implies that the baroclinic eigenvalues are not much separated. Moreover, the assumption on $\boldsymbol{\pi}$ is verified, but the one on $\boldsymbol{\varpi}$ can be contradicted. On the one hand, a part of the layers used to describe the deep sea are reduced with a thickness of the order of $\epsilon$ and then, there would exist $j_0 \in [\![1,n-1]\!]$ such that
\begin{equation}
1 \ll \varpi_{j_0}.
\label{1config40couches1}
\end{equation}

\noindent
On the other hand, it can be interesting to increase the number of layers in a certain area where well-known phenomena occur, in order to provide more accurate results. Then, there would exist $k_0 \in [\![1,n-1]\!]$ such that
\begin{equation}
\varpi_{k_0} \ll 1. 
\label{1config40couches2}
\end{equation}

\begin{figure}[ht]
\centering
\includegraphics[width=10cm,height=70mm]{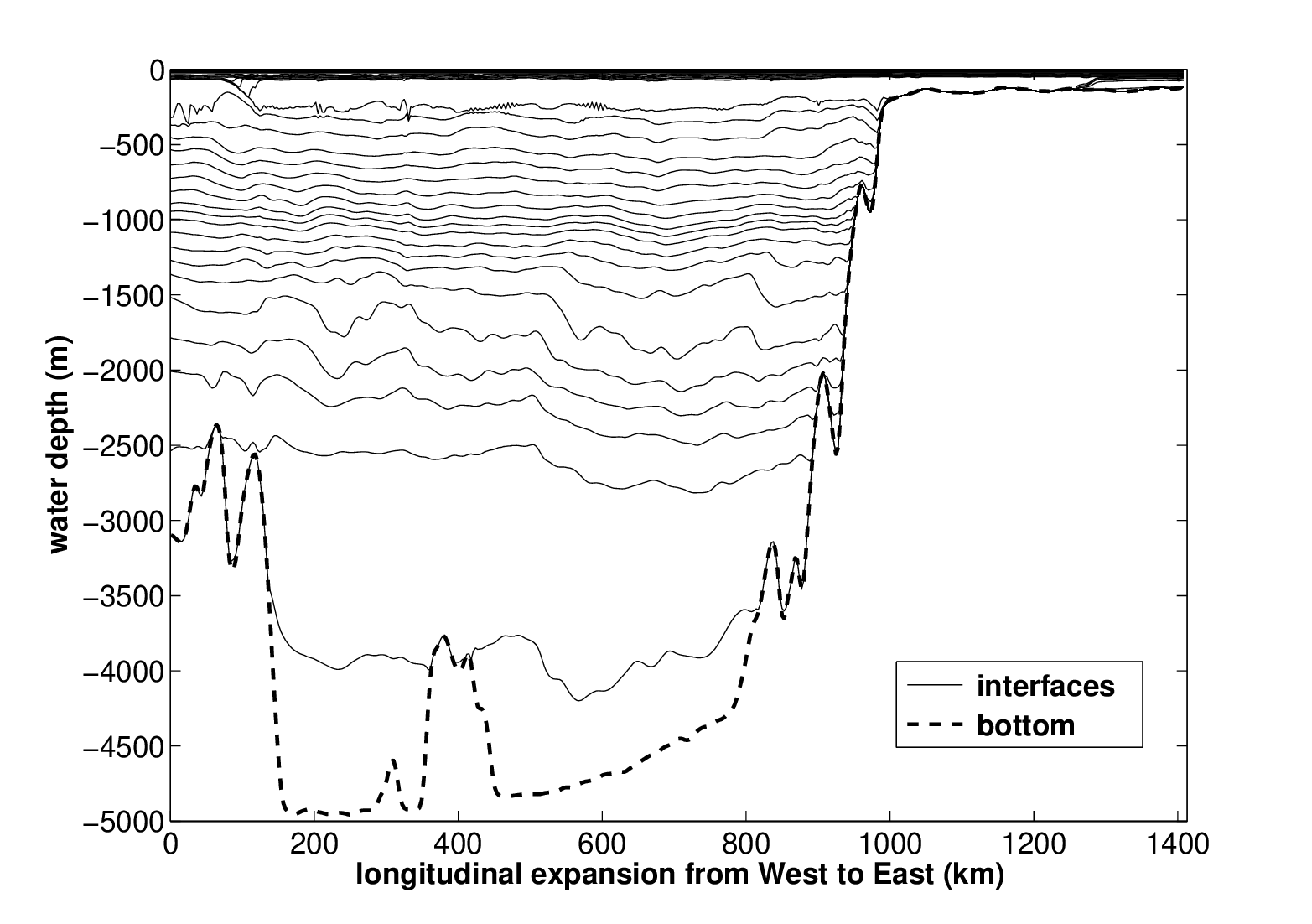}
\caption{Configuration of $40$ layers in the bay of Biscay}
\label{1figure40couches}
\end{figure}

\noindent
On the figure \ref{1figure40couches}, both of these cases occur. The first one concerns the layers near the bottom: these layers have high heights in the deep sea but, at the oceanic plateau, these heights tend to $0$ and then \eqref{1config40couches1} is verified for some $j_0 \in [\![1,n-1]\!]$. Moreover, the second case concerns the layers near the free surface: to describe well the mixing-zone, $20$ layers are necessary --- the black band just below the free surface --- and therefore the assumption \eqref{1config40couches2} is verified for some $k_0 \in [\![1,n-1]\!]$.

\noindent
{\em Remark:} According to the figure \ref{1figure40couches}, the assumption \eqref{1config40couches2} would be verified for all $X \in \mathbb{R}^2$. However, this is not true in the case of \eqref{1config40couches1}.

\noindent
As a consequence of the proposition \ref{1proplambdaipmapprox}, we can deduce the next theorem:
\begin{theorem}
Let ${\bf u^0} \in \mathcal{L}^2(\mathbb{R}^2)^{3n}$, $\boldsymbol{\gamma} \in ]0,1[^{n-1}$, $\epsilon >0$ and an injective function $\sigma: [\![1,n-1]\!] \rightarrow \mathbb{R}_+^*$ such that $\boldsymbol{\gamma}$ verifies \eqref{1hypasympt} and for all $X \in \mathbb{R}^2$, ${\bf u^0}(X)$ verifies \eqref{1hypasympt2}.

\noindent
There exists $\delta > 0$ such that if
\begin{equation}
\epsilon \le \delta,
\label{1hyplambdaipm2}
\end{equation}
and
\begin{equation}
\left\{
\begin{array}{lr}
\inf_{X \in \mathbb{R}^2} h_i^0(X) > 0, & \forall i \in [\![1,n]\!],\\[6pt]
\inf_{X \in \mathbb{R}^2} \phi_{\sigma,i}(\boldsymbol{h^0}(X))-|u^0_{i+1}-u^0_i|^2(X)-|v^0_{i+1}-v^0_i|^2(X) > 0,& \forall i \in [\![1,n-1]\!],
\end{array}\right.
\label{1condnechyp2}
\end{equation}
where $\phi_{\sigma,i}(\boldsymbol{h}):=(h_{\sigma,i}^- +h_{\sigma,i}^+)(1-\gamma_i)$. Then, the system {\em (\ref{1systemmultilayer})}, with initial data ${\bf u^0}$, is hyperbolic.
\label{1corolcaluderiv3}
\end{theorem}

\noindent
{\em Remark:} A direct consequence of this theorem is that if the model \eqref{1systemmultilayer} is hyperbolic, then the {\em Rayleigh-Taylor} stability is verified ({\em i.e.} $ \rho_n > \rho_{n-1} > \ldots > \rho_1 > 0$).

\begin{proof}
To verify the hyperbolicity of the system (\ref{1systemmultilayer}), all the eigenvalues of $\mathsf{A}({\bf u},\boldsymbol{\gamma},\theta)$ need to be real. Let $i \in [\![1,n-1]\!]$, according to the rotational invariance (\ref{1rotinv}) and the proposition \ref{1thmapproxlambdaipm}, if $({\bf u},\boldsymbol{\gamma})$ verify \eqref{1hypasympt}, \eqref{1hypasympt2} and \eqref{1hyplambdaipm2}, the asymptotic expansion of $\lambda_i^{\pm}(\mathsf{P}(\theta){\bf u},\boldsymbol{\gamma}) \in \sigma \left( \mathsf{A}({\bf u},\boldsymbol{\gamma},\theta \right)$ is
\begin{equation}
\begin{array}{ll}
\lambda_i^{\pm}(\mathsf{P}(\theta){\bf u},\boldsymbol{\gamma})=& \cos(\theta) \frac{u_{i+1} h_{\sigma,i}^-+u_{i}h_{\sigma,i}^-}{h_{\sigma,i}^-+h_{\sigma,i}^+}+ \sin(\theta) \frac{v_{i+1} h_{\sigma,i}^-+v_{i}h_{\sigma,i}^+}{h_{\sigma,i}^- + h_{\sigma,i}^+}\\[6pt]
& \pm \left[\frac{ h_{\sigma,i}^- h_{\sigma,i}^+}{h_{\sigma,i}^- +h_{\sigma,i}^+}\left(1-\gamma_i- \frac{(\cos(\theta) (u_{i+1}-u_i)+\sin(\theta) (v_{i+1}-v_i))^2}{h_{\sigma,i}^- +h_{\sigma,i}^+}\right)\right]^{\frac{1}{2}}\\[8pt]
&  +\mathcal{O}(\epsilon^{\frac{\sigma(i)+1}{2}}).
\end{array}
\label{1approxlambdaipmtheta2}
\end{equation}
\noindent
Then, as $h_{\sigma,i}^- >0$ and $h_{\sigma,i}^+ >0$, for all $i \in [\![1,n-1]\!]$, if $h_k >0$ for all $k \in [\![1,n]\!]$, a necessary condition to have $\lambda_i^{\pm}(\mathsf{P}(\theta){\bf u},\boldsymbol{\gamma}) \in \mathbb{R}$, for all $\theta \in [0,2\pi]$, is
\begin{equation}
\forall \theta \in [0,2\pi],\  1-\gamma_i- \frac{(\cos(\theta) (u_{i+1}-u_i)+\sin(\theta) (v_{i+1}-v_i))^2}{h_{\sigma,i}^- +h_{\sigma,i}^+} \ge 0.
\label{1condneclambdarealtheta2}
\end{equation}
Finally, using \eqref{maxPthetau}, the necessary condition of hyperbolicity \eqref{1condnechyp2} is obtained.
\end{proof}

\noindent
Then, the theorem \ref{1corolcaluderiv3} insures the set of hyperbolicity, $\mathcal{H}_{\boldsymbol{\gamma}}$, contains all the elements verifying the conditions \eqref{1hypasympt2} and \eqref{1condnechyp2}, when $\boldsymbol{\gamma}$ verifies \eqref{1hypasympt} and \eqref{1hyplambdaipm2}.

\noindent
{\em Remarks:} 1) The theorem \ref{1corolcaluderiv3} is a generalization of the theorem \ref{1corolcaluderiv1}, in the asymptotic regime \eqref{1hypasympt} and \eqref{1hypasympt2}. Moreover, the shape of the baroclinic eigenvalues, in the merged-layer case \eqref{1approxlambdaipm}, is the same in the considered asymptotic regime, with the assumptions \eqref{1hyplambdaipm2}. 2) In \cite{stewart2012multilayer}, the numerical set of hyperbolicity of the three-layer model, in one dimension (see figure $4$), seems that the difference of velocities $u_{i+1}-u_i$, for $i \in [\![1,2]\!]$, is allowed to be very large: this should prove the last criterion, we gave in theorem \ref{1corolcaluderiv3}, is just very different from the entire set of hyperbolicity. However, as it was proved in \cite{monjarret2014local} for the two-layer case, there is a gap between the one and the two dimensions sets of hyperbolicity. Indeed, the elements in the one dimension set have to be rotational invariant ({\em i.e.} remain in the one dimension set of hyperbolicity if a rotation is applied) to be in the two dimensions one. This is why it should not be far from the exact set of hyperbolicity, even if the criterion \eqref{1condnechyp2} is a necessary condition of hyperbolicity of the multi-layer shallow water model, in two dimensions, and not a sufficient one.

\noindent
In conclusion, we managed to obtain an asymptotic expansion of the baroclinic eigenvalues, $\left(\lambda_i^{\pm}({\bf u},\boldsymbol{\gamma})\right)_{i \in [\![1,n-1]\!]}$, considering the asymptotic regime \eqref{1hypasympt} and assuming the heights of each layer have all the same range and the difference of velocity between an interface has the same order as the square root of the relative difference of density, at this interface. The expansions of $\lambda_i^{\pm}({\bf u},\boldsymbol{\gamma})$, for $i \in [\![1,n-1]\!]$, has been proved with a precision in $ \mathcal{O}(\epsilon^{\frac{\sigma(i)+1}{2}})$.

\subsection{Comparison of the criteria}
In this paper, we expressed two explicit criteria of local well-posedness of the multi-layer shallow water model with free surface: an explicit criterion of symmetrizability --- see \eqref{1condwellposed} with \eqref{1estimdeltai2} --- and there is an explicit criterion of hyperbolicity --- see \eqref{1corolcaluderiv3}. The main difference between both of them is that the $1^{\mathrm{st}}$ one gives conditions on $|u_i - \bar{u}|^2$, for all $i \in [\![1,n]\!]$,while the $2^{\mathrm{nd}}$ one gives conditions on $|u_{i+1} - u_i|^2$, for all $i \in [\![1,n-1]\!]$. Then, to compare these two criteria, we need to know which one of the following assertions is true, in the asymptotic regime \eqref{1hypasympt} and \eqref{1hypasympt2}:
\begin{equation}
\forall i \in [\![1,n-1]\!],\ \phi_{\sigma,i}(\boldsymbol{h}) \le \left( \frac{\alpha_{n,i}+\alpha_{n,i+1}}{\alpha_{n,i}\alpha_{n,i+1}a({\bf h},\boldsymbol{\gamma})}\right)^2,
\label{1comparaisoncriteres1}
\end{equation}
or
\begin{equation}
\forall i \in [\![1,n-1]\!],\ \left( \frac{\alpha_{n,i}+\alpha_{n,i+1}}{\alpha_{n,i}\alpha_{n,i+1}a({\bf h},\boldsymbol{\gamma})}\right)^2 \le \phi_{\sigma,i}(\boldsymbol{h}).
\label{1comparaisoncriteres2}
\end{equation}
Indeed, if
\begin{equation}
\forall i \in [\![1,n]\!],\
\frac{1}{\left( \alpha_{n,i} a({\bf h},\boldsymbol{\gamma}) \right)^2} > | u_i - \bar{u} |^2,
\label{1symwellposedcond}
\end{equation}
then
\begin{equation}
\forall i \in [\![1,n-1]\!],\ \left( \frac{\alpha_{n,i}+\alpha_{n,i+1}}{\alpha_{n,i}\alpha_{n,i+1}a({\bf h},\boldsymbol{\gamma})}\right)^2 > | u_{i+1} - u_i |^2.
\end{equation}

\noindent
Consequently, if \eqref{1comparaisoncriteres1} is verified, for instance, it implies the conditions of hyperbolicity:
\begin{equation}
\forall i \in [\![1,n-1]\!],\
\phi_{\sigma,i}(\boldsymbol{h}) > | u_{i+1} - u_i |^2.
\end{equation}

\noindent
Let $\boldsymbol{\gamma} \in ]0,1[^{n-1}$, $\epsilon >0$ and an injective function $\sigma: [\![1,n-1]\!] \rightarrow \mathbb{R}_+^*$ such that $\boldsymbol{\gamma}$ verifies \eqref{1hypasympt}. We define the next subset of $\mathcal{L}^2(\mathbb{R}^2)^{3n}$
\begin{equation}
\mathcal{H}_{\sigma,\epsilon}:= \left\{ {\bf u^0} \in \mathcal{L}^2(\mathbb{R}^2)^{3n} /\ \forall X \in \mathbb{R}^2,\ {\bf u^0}(X)\ \mathrm{verifies}\ \mathrm{conditions}\ {\rm \eqref{1hypasympt2}}\ {\rm and}\ {\rm \eqref{1condnechyp2}} \right\}
\label{1defHgammasigmaepsilon}
\end{equation}
and the subset of $\mathcal{H}^s(\mathbb{R}^2)^{3n}$
\begin{equation}
\mathcal{S}_{\sigma,\epsilon}^s:= \left\{ {\bf u^0} \in \mathcal{H}^s(\mathbb{R}^2)^{3n} /\ \forall X \in \mathbb{R}^2,\ {\bf u^0}(X)\ \mathrm{verifies}\ \mathrm{conditions}\ {\rm \eqref{1hypasympt2}}\ {\rm and}\ {\rm \eqref{1symwellposedcond}} \right\}
\end{equation}

\noindent
According to the proposition \ref{1propestimdeltai2}, it is clear that
\begin{equation}
\mathcal{S}_{\sigma,\epsilon}^s \subset \mathcal{S}_{\gamma}^s,
\end{equation}

\noindent
and according to the theorem \ref{1corolcaluderiv3}, it is clear that if $\epsilon \le \delta$,
\begin{equation}
\mathcal{H}_{\sigma,\epsilon} \subset \mathcal{H}_{\gamma}.
\end{equation}

\noindent
Moreover, there is the next proposition:

\begin{proposition}
Let $s>2$, ${\bf u^0}: \mathbb{R}^2 \rightarrow \mathbb{R}^{3n}$, $\boldsymbol{\gamma} \in ]0,1[^{n-1}$, $\epsilon >0$ and an injective function $\sigma: [\![1,n-1]\!] \rightarrow \mathbb{R}_+^*$ such that $\boldsymbol{\gamma}$ verifies \eqref{1hypasympt} and for all $X \in \mathbb{R}^2$, ${\bf u}(X)$ verifies \eqref{1hypasympt2}.

\noindent
There exists $\delta > 0$ such that if
\begin{equation}
\epsilon \le \delta,
\label{1hyplambdaipm3}
\end{equation}
then,
\begin{equation}
\mathcal{S}^s_{\sigma,\epsilon} \subset \mathcal{H}_{\sigma,\epsilon} \cap \mathcal{H}^s(\mathbb{R}^2)^{3n}.
\label{1subsethypsym}
\end{equation}
\label{1prop1subsethypsym}
\end{proposition}

\begin{proof}
First, we remind we proved in proposition \ref{1propestimdeltai2}
\begin{equation}
\forall i \in [\![1,n]\!],\ \frac{1}{\left( \alpha_{n,i} a({\bf h},\boldsymbol{\gamma}) \right)^2} \le \delta_i(\boldsymbol{h},\boldsymbol{\gamma}),
\end{equation}
with
\begin{equation}
a({\bf h},\boldsymbol{\gamma}):=\max\left(\left(\frac{\alpha_{n,2}}{\alpha_{n,1}}+1 \right)p_1,2\max_{k \in [\![1,n-2]\!]}\left( p_k+p_{k+1} \right), \max_{i \in [\![1,n]\!]} \left(\alpha_{n,i}^{-1} h_i^{-1}\right) \right),
\label{1ahgamma2}
\end{equation}
and for all $k \in [\![1,n-1]\!]$, $p_k:=\frac{1}{\alpha_{n,k+1}-\alpha_{n,k}}=\frac{\rho_n}{\rho_{k+1}-\rho_k}$.

\noindent
According to the asymptotic regime considered, we have
\begin{equation}
\left(\frac{\alpha_{n,2}}{\alpha_{n,1}}+1 \right)p_1=\frac{\rho_n}{\rho_2}\big( 2\epsilon^{-\sigma(1)}+1-\epsilon^{\sigma(1)} \big),
\end{equation}
and
\begin{equation}
\forall k \in [\![1,n-2]\!],\ p_k + p_{k+1}=\rho_n \big( \frac{\epsilon^{-\sigma(k)}}{\rho_{k+1}}+\frac{\epsilon^{-\sigma(k+1)}}{\rho_{k+2}} \big),
\end{equation}

\noindent
then, defining $m_{\sigma}^+ \in [\![1,n-1]\!]$ such that
\begin{equation}
\sigma(m_{\sigma}^+)=\max_{i \in [\![1,n-1]\!]} \sigma(i),
\label{1msigmap}
\end{equation}
we have:
\begin{equation}
\max\left( \left(\frac{\alpha_{n,2}}{\alpha_{n,1}}+1 \right)p_1,2 \max_{k \in [\![1,n-2]\!]}(p_k + p_{k+1})\right) \sim \frac{2 \rho_n \epsilon^{-\sigma(m_{\sigma}^+)}}{\rho_{m_{\sigma}^+ +1}}.
\end{equation}
Afterwards, we differentiate two cases: let $i \in [\!1,n-1]\!]$,
\begin{equation}
\frac{1}{(\alpha_{n,i}a({\bf h},\boldsymbol{\gamma}))^2}=
\left\{
\begin{array}{ll}
\frac{(\min_{k \in [\![1,n]\!]} \rho_k h_k)^2}{\rho_i^2},& \mathrm{if}\ \exists k \in [\![1,n]\!], \rho_k h_k \le \frac{\rho_{m_{\sigma}^+ +1}}{2}\epsilon^{\sigma(m_{\sigma}^+)},\\[6pt]
\left(\frac{\rho_{m_{\sigma}^+ +1}}{2 \rho_i} \right)^2\epsilon^{2 \sigma(m_{\sigma}^+)},& \mathrm{otherwise}.
\end{array}
\right.
\label{1ordredeltai}
\end{equation}

\noindent
In the $1^{\mathrm{st}}$ case, as $1-\epsilon \le \gamma_i < 1$ with $\epsilon \le \delta$, we have, for all $i \in [\![1,n-1]\!]$:
\begin{equation}
\begin{array}{ll}
\left( \frac{\alpha_{n,i}+\alpha_{n,i+1}}{\alpha_{n,i}\alpha_{n,i+1}a({\bf h},\boldsymbol{\gamma})}\right)^2 &\le \left(\frac{2\min_{k \in [\![1,n]\!]} \rho_k h_k}{\rho_i}\right)^2,\\[8pt]
&\le \left(\frac{\rho_{m_{\sigma}^+ +1}}{\rho_i} \right)^2\epsilon^{2\sigma(m_{\sigma}^+)},
\end{array}
\end{equation}
and then, as $\frac{\rho_{m_{\sigma}^+ +1}}{\rho_i} \epsilon^{\sigma(m_{\sigma}^+)} \le h_{\sigma,i}^- +h_{\sigma,i}^+$ and $\frac{\rho_{m_{\sigma}^+ +1}}{\rho_i} \epsilon^{\sigma(m_{\sigma}^+)} \le \epsilon^{\sigma(i)}$, the next inequality remains true
\begin{equation}
\frac{\rho_{m_{\sigma}^+ +1}^2}{1-\epsilon} \epsilon^{2\sigma(m_{\sigma}^+)} \le (h_{\sigma,i}^- +h_{\sigma,i}^+)\epsilon^{\sigma(i)} = \phi_{\sigma,i}(\boldsymbol{h}),
\end{equation}

\noindent
In the $2^{\mathrm{nd}}$ case, as $1-\epsilon \le \gamma_i < 1$ with $\epsilon \ll 1$, we have, for all $i \in [\![1,n-1]\!]$::
\begin{equation}
\begin{array}{ll}
\left( \frac{\alpha_{n,i}+\alpha_{n,i+1}}{\alpha_{n,i}\alpha_{n,i+1}a({\bf h},\boldsymbol{\gamma})}\right)^2 &\le \left( \frac{\rho_{m_{\sigma}^+ +1}}{2 \rho_i}\epsilon^{\sigma(m_{\sigma}^+)}+\frac{\rho_{m_{\sigma}^+ +1}}{2 \rho_{i+1}} \epsilon^{\sigma(m_{\sigma}^+)} \right)^2\\[8pt]
&  < \frac{\rho_{m_{\sigma}^+ +1}^2}{1-\epsilon} \epsilon^{2\sigma(m_{\sigma}^+)},
\end{array}
\end{equation}
Moreover, there is also the next inequality:
\begin{equation}
\frac{\rho_{m_{\sigma}^+ +1}^2}{1-\epsilon} \epsilon^{2\sigma(m_{\sigma}^+)} \le (h_{\sigma,i}^- +h_{\sigma,i}^+)\epsilon^{\sigma(i)} = \phi_{\sigma,i}(\boldsymbol{h}),
\end{equation}
according to the definition of $m_{\sigma}^+$ and that all the layers $i \in [\![1,n]\!]$ are supposed to verify
\begin{equation}
\rho_i h_i \ge \frac{\rho_{m_{\sigma}^+}}{2}\epsilon^{\sigma(m_{\sigma}^+)}.
\end{equation}

\noindent
Finally, we have
\begin{equation}
\forall i \in [\![1,n-1]\!],\ \left( \frac{\alpha_{n,i}+\alpha_{n,i+1}}{\alpha_{n,i}\alpha_{n,i+1}a({\bf h},\boldsymbol{\gamma})}\right)^2 \le \phi_{\sigma,i}(\boldsymbol{h}),
\end{equation}
and the proposition \ref{1prop1subsethypsym} is proved. 
\end{proof}

\noindent
To conclude, we proved that the criterion, highlighted in this section, is a weaker one than the criterion of symmetrizability, we proved in the previous section, when the asymptotic regime \eqref{1hypasympt}, \eqref{1hypasympt2} and \eqref{1hyplambdaipm3} are verified. In the next section, we perform the asymptotic expansion of the eigenvectors, in order to specify the nature of the waves associated to each eigenvalues and to prove the diagonalizability of the matrix $\mathsf{A}({\bf u},\boldsymbol{\gamma},\theta)$.

\section{Asymptotic expansion of all the eigenvectors}
In the previous section, the asymptotic expansion of all the eigenvalues associated to $\mathsf{A}_x({\bf u},\boldsymbol{\gamma})$ were proved, in a particular regime, which enables to separate all the baroclinic eigenvalues. In this section, we will give the associated expressions of the asymptotic expansions of all the eigenvectors of $\mathsf{A}_x({\bf u},\boldsymbol{\gamma})$. Moreover, we will deduce the diagonalizability of the matrix $\mathsf{A}({\bf u},\boldsymbol{\gamma},\theta)$ and the local well-posedness of the model \eqref{1systemmultilayer} in $\mathcal{H}^s(\mathbb{R}^2)^{3n}$, with $s>2$. Finally, the nature of the waves associated to each eigenvalues --  shock, contact or rarefaction wave -- in the asymptotic regime considered in the previous section, is deduced.

\subsection{The barotropic eigenvectors}
In the asymptotic regime \eqref{1hypasympt} and \eqref{1hypasympt2}, we can deduce the asymptotic expansions of the right and left eigenvectors associated to $\lambda_n^{\pm}({\bf u},\boldsymbol{\gamma})$.

\begin{proposition}
Let $({\bf u},\boldsymbol{\gamma}) \in \mathbb{R}^{3n} \times ]0,1[^{n-1}$, $\epsilon >0$ and an injective function $\sigma \in \mathbb{R}_+^{*\ [\![1,n-1]\!]}$ such that $\boldsymbol{\gamma}$ verifies \eqref{1hypasympt}, ${\bf u}$ verifies \eqref{1hypasympt2} and
\begin{equation}
\epsilon \ll 1.
\label{1assumpteigenvectorlambdan}
\end{equation}
Then, the asymptotic expansion of the right eigenvector associated to $\lambda_n^{\pm}({\bf u},\boldsymbol{\gamma})$, with precision about $\mathcal{O}(1-\gamma_{m_{\sigma}^-})$, is such that
\begin{equation}
\begin{array}{ll}
\boldsymbol{r_x^{\lambda_n^{\pm}}}({\bf u},\boldsymbol{\gamma})=& \sum_{k=1}^n \boldsymbol{ e_{n+k}}\pm \frac{h_k}{\sqrt{H}}\boldsymbol{ e_k}\\[6pt]
& -\sum_{k=1}^{m_{\sigma}^-}  \frac{(u_{m_{\sigma}^-+1}-u_{m_{\sigma}^-})}{\sqrt{H}} \left( \frac{2 h_k}{\sqrt{H}} \boldsymbol{ e_{k}} \pm \boldsymbol{ e_{n+k}}\right)\\[6pt]
& + \sum_{k=m_{\sigma}^-+1}^{n}  \frac{(u_{m_{\sigma}^-+1}-u_{m_{\sigma}^-}) h_{\sigma,m_{\sigma}^-}^-}{\sqrt{H}h_{\sigma,m_{\sigma}^-}^+} \left( \frac{2 h_k}{\sqrt{H}} \boldsymbol{ e_{k}} \pm \boldsymbol{ e_{n+k}} \right)\\[6pt]
&+\mathcal{O}(1-\gamma_{m_{\sigma}^-}),\\[6pt]
\end{array}
\label{1eigvectlambdnpmr2}
\end{equation}

\noindent
and the asymptotic expansion of the left eigenvector associated to $\lambda_n^{\pm}({\bf u},\boldsymbol{\gamma})$, with precision about $\mathcal{O}(1-\gamma_{m_{\sigma}^-})$, is such that
\begin{equation}
\begin{array}{ll}
\boldsymbol{l_x^{\lambda_n^{\pm}}}({\bf u},\boldsymbol{\gamma})=& \sum_{k=1}^n {}^{\top} \boldsymbol{ e_{k}}\pm \frac{h_k}{\sqrt{H}}{}^{\top}\boldsymbol{ e_{n+k}}\\[6pt]
& -\sum_{k=1}^{m_{\sigma}^-}  \frac{(u_{m_{\sigma}^-+1}-u_{m_{\sigma}^-})}{\sqrt{H}} \left(\frac{2 h_k}{\sqrt{H}} {}^{\top}\boldsymbol{ e_{n+k}} \pm {}^{\top} \boldsymbol{ e_{k}}\right)\\[6pt]
& + \sum_{k=m_{\sigma}^-+1}^{n}  \frac{(u_{m_{\sigma}^-+1}-u_{m_{\sigma}^-}) h_{\sigma,m_{\sigma}^-}^-}{\sqrt{H}h_{\sigma,m_{\sigma}^-}^+} \left( \frac{2 h_k}{\sqrt{H}} {}^{\top} \boldsymbol{ e_{n+k}} \pm {}^{\top} \boldsymbol{ e_{k}} \right)\\[6pt]
&+\mathcal{O}(1-\gamma_{m_{\sigma}^-}).
\end{array}
\label{1expressionlefteigenvectn}
\end{equation}
\label{1proplambdanexpandeigenvector}
\end{proposition}

\begin{proof}
According to the proposition \ref{1proplambdanexpandeigenvector2},
\begin{equation}
\begin{array}{ll}
\lambda_n^{\pm}({\bf u},\boldsymbol{\gamma})&=u_{m_{\sigma}^-} \pm \sqrt{H}+\psi_{m_{\sigma}^-} \epsilon^{\frac{1}{2}} \pm \big( -h_{\sigma,m_{\sigma}^-}^- h_{\sigma,m_{\sigma}^-}^+ \big) \epsilon +\mathcal{O}(\epsilon^{\frac{3}{2}})\\[6pt]
&=u_{m_{\sigma}^-} \pm \sqrt{H}+\psi_{m_{\sigma}^-} \epsilon^{\frac{1}{2}} +\mathcal{O}(\epsilon),
\end{array}
\label{1reformullambdai22}
\end{equation}
where $\psi_{m_{\sigma}^-}:=\pi_{m_{\sigma}^-}\frac{h_{\sigma,m_{\sigma}^-}^+}{h_{\sigma,m_{\sigma}^-}^- + h_{\sigma,m_{\sigma}^-}^+}=\pi_{m_{\sigma}^-}\frac{h_{\sigma,m_{\sigma}^-}^+}{H}$. We expand the eigenvectors $\boldsymbol{r_x^{\lambda_n^{\pm}}}({\bf u},\boldsymbol{\gamma})$ and $\boldsymbol{l_x^{\lambda_n^{\pm}}}({\bf u},\boldsymbol{\gamma})$ such that
\begin{equation}
\left\{
\begin{array}{l}
\boldsymbol{r_x^{\lambda_n^{\pm}}}({\bf u},\boldsymbol{\gamma})= \boldsymbol{r_{n,0}^{\pm}}({\bf u},\boldsymbol{\gamma})+\epsilon^{\frac{1}{2}} \boldsymbol{r_{n,1}^{\pm}}({\bf u},\boldsymbol{\gamma})+\mathcal{O}(\epsilon),\\[6pt]
\boldsymbol{l_x^{\lambda_n^{\pm}}}({\bf u},\boldsymbol{\gamma})= \boldsymbol{l_{n,0}^{\pm}}({\bf u},\boldsymbol{\gamma})+\epsilon^{\frac{1}{2}} \boldsymbol{l_{n,1}^{\pm}}({\bf u},\boldsymbol{\gamma})+\mathcal{O}(\epsilon),
\end{array}
\right.
\label{1eigvectoexpand22}
\end{equation}
\noindent
where
\begin{equation}
\left\{
\begin{array}{l}
\boldsymbol{r_{n,0}^{\pm}}({\bf u},\boldsymbol{\gamma}):=\sum_{k=1}^n \boldsymbol{ e_{n+k}}\pm \frac{h_k}{\sqrt{H}}\boldsymbol{ e_k},\\[6pt]
\boldsymbol{l_{n,0}^{\pm}}({\bf u},\boldsymbol{\gamma}):=\sum_{k=1}^n {}^{\top} \boldsymbol{ e_{k}}\pm \frac{h_k}{\sqrt{H}}{}^{\top}\boldsymbol{ e_{n+k}}.
\end{array}
\right.
\end{equation}

\noindent
Moreover, we have
\begin{equation}
\begin{array}{ll}
\mathsf{A}_x({\bf u},\boldsymbol{\gamma})=& \mathsf{A}_x({\bf u}_{\boldsymbol{\sigma}}^{\boldsymbol{m_{\sigma}^-}},\boldsymbol{\gamma^{m_{\sigma}^-}_{\sigma}})+\sum_{k=1 }^{n}(u_{k}-u_{m_{\sigma}^-}) \mathsf{A}_x^{k-1,1},\\[6pt]
&+ \sum_{k=1}^{n-1}(1-\gamma_k)\mathsf{A}_x^{k,3}(\boldsymbol{\gamma}),
\end{array}
\label{1equalityA1A224}
\end{equation}

\noindent
where the $3n \times 3n$ matrices, $\mathsf{A}_x^{k-1,1}$ and $\mathsf{A}_x^{k,2}(\boldsymbol{\gamma})$, are defined respectively in (\ref{1defA1}--\ref{1defA2}), $\mathsf{A}_x^{k,3}(\boldsymbol{\gamma}):=\left[ \mathsf{A}_{l,m}^{k,3} \right]_{(l,m) \in [\![1,n]\!]^2}$ is defined by
\begin{equation}
\mathsf{A}_{l,m}^{k,3}:=
\left\{
\begin{array}{lr}
0, & \mathrm{if}\ m \not=k,\\[6pt]
0,& \mathrm{if}\ l \le n+k\ \mathrm{or}\ l \ge 2n+1,\\[6pt]
-\alpha_{l-n-1,k},& \mathrm{otherwise},
\end{array}
\right.
\label{1defA3}
\end{equation}
 and for all $i \in [\![1;n-1]\!]$, $({\bf u}_{\boldsymbol{\sigma}}^{\boldsymbol{i}},\boldsymbol{\gamma^i_{\sigma}})$ are defined by
\begin{equation}
\left\{
\begin{array}{lr}
u_{j}=u_i,& \forall j \in [\![m_i^-,m_i^+]\!],\\[6pt]
\gamma_j=1,& \forall j \in [\![m_i^-,m_i^+ -1]\!].
\end{array}
\right.
\label{1usigmagammaj}
\end{equation}

\noindent
Consequently, according to the assumptions \eqref{1hypasympt}, \eqref{1hypasympt2} and  \eqref{1assumpteigenvectorlambdan},
\begin{equation}
\begin{array}{ll}
\mathsf{A}_x({\bf u},\boldsymbol{\gamma})&=\mathsf{A}_x({\bf u}_{\boldsymbol{\sigma}}^{\boldsymbol{m_{\sigma}^-}},\boldsymbol{\gamma^{m_{\sigma}^-}_{\sigma}})+ \sum_{k=m_{\sigma}^-+1}^n \pi_{m_{\sigma}^-} \epsilon^{\frac{1}{2}} \mathsf{A}_x^{k-1,1}+\epsilon\ \mathsf{A}_x^{{m_{\sigma}^-},3}(\boldsymbol{\gamma}) + \mathcal{O}(\epsilon^{\frac{3}{2}}),\\[6pt]
&=\mathsf{A}_x({\bf u}_{\boldsymbol{\sigma}}^{\boldsymbol{m_{\sigma}^-}},\boldsymbol{\gamma^{m_{\sigma}^-}_{\sigma}})+ \sum_{k=m_{\sigma}^-+1}^n \pi_{m_{\sigma}^-} \epsilon^{\frac{1}{2}} \mathsf{A}_x^{k-1,1}+ \mathcal{O}(\epsilon).
\end{array}
\label{1equalityA1A24}
\end{equation}

\noindent
Therefore, $\boldsymbol{r_x^{\lambda_n^{\pm}}}({\bf u},\boldsymbol{\gamma})$ is the approximation of the right eigenvector associated to $\lambda_n^{\pm}({\bf u},\boldsymbol{\gamma})$, with a precision about $\mathcal{O}(\epsilon)$, if and only if $\boldsymbol{r_{n,1}^{\pm}}({\bf u},\boldsymbol{\gamma})$ verifies:
\begin{equation}
\begin{array}{l}
\big(\sum_{k=m_{\sigma}^- +1}^n \pi_{m_{\sigma}^-} \mathsf{A}_x^{k-1,1}-\psi_{m_{\sigma}^-} \mathsf{I}_{3n}\big)\boldsymbol{r_{n,0}^{\pm}}({\bf u},\boldsymbol{\gamma})\\[6pt]
=-\big(\mathsf{A}_x({\bf u}_{\sigma}^{\boldsymbol{m_{\sigma}^-}},\boldsymbol{\gamma_{\sigma}^{m_{\sigma}^-}})-(u_{m_{\sigma}^-} \pm \sqrt{H}) \mathsf{I}_{3n}\big)\boldsymbol{r_{n,1}^{\pm}}({\bf u},\boldsymbol{\gamma}),
\end{array}
\label{1eigsystem4}
\end{equation}

\noindent
and $\boldsymbol{l_x^{\lambda_n^{\pm}}}({\bf u},\boldsymbol{\gamma})$ is the approximation of the left eigenvector associated to $\lambda_n^{\pm}({\bf u},\boldsymbol{\gamma})$, with a precision about $\mathcal{O}(\epsilon)$, if and only if $\boldsymbol{l_{n,1}^{\pm}}({\bf u},\boldsymbol{\gamma})$ verifies:
\begin{equation}
\begin{array}{l}
\boldsymbol{l_{n,0}^{\pm}}({\bf u},\boldsymbol{\gamma}) \big(\sum_{k=m_{\sigma}^- +1}^n \pi_{m_{\sigma}^-} \mathsf{A}_x^{k-1,1}-\psi_{m_{\sigma}^-} \mathsf{I}_{3n}\big)\\[6pt]
=-\boldsymbol{l_{n,1}^{\pm}}({\bf u},\boldsymbol{\gamma})\big(\mathsf{A}_x({\bf u}_{\sigma}^{\boldsymbol{m_{\sigma}^-}},\boldsymbol{\gamma_{\sigma}^{m_{\sigma}^-}})-(u_{m_{\sigma}^-}\pm \sqrt{H}) \mathsf{I}_{3n}\big).
\end{array}
\label{1eigsystem5}
\end{equation}

\noindent
Therefore, $\boldsymbol{r_{n,1}^{\pm}}({\bf u},\boldsymbol{\gamma}):= {}^{\top} (r_{n,1}^{\pm,k})_{k \in [\![1,3n]\!]}$ is solution of \eqref{1eigsystem4} if and only if
\begin{equation}
r_{n,1}^{\pm,k}=
\left\{
\begin{array}{lr}
\frac{h_k}{H}( S_r - 2 \psi_{m_{\sigma}^-}),& \mathrm{if}\ k \in [\![1,m_{\sigma}^-]\!],\\[6pt]
\frac{h_k}{H}(S_r+2 \frac{h_{\sigma,m_{\sigma}^-}^-}{h_{\sigma,m_{\sigma}^-}^+} \psi_{m_{\sigma}^-}),& \mathrm{if}\ k \in [\![m_{\sigma}^-+1,n]\!],\\[6pt]
\pm \frac{1}{\sqrt{H}} (S_r-\psi_{m_{\sigma}^-}),& \mathrm{if}\ k \in [\![n+1,n+m_{\sigma}^-]\!],\\[6pt]
\pm \frac{1}{\sqrt{H}}(S_r+\frac{h_{\sigma,m_{\sigma}^-}^-}{h_{\sigma,m_{\sigma}^-}^+} \psi_{m_{\sigma}^-}),& \mathrm{if}\ k \in [\![n+m_{\sigma}^-+1,2n]\!],\\[6pt]
0,& \mathrm{if}\ k \in [\![2n+1,3n]\!],
\end{array}
\right.
\end{equation}

\noindent
and $\boldsymbol{l_{n,1}^{\pm}}({\bf u},\boldsymbol{\gamma}):= (l_{n,1}^{\pm,k})_{k \in [\![1,3n]\!]}$ is solution of \eqref{1eigsystem5} if and only if
\begin{equation}
l_{n,1}^{\pm,k}=
\left\{
\begin{array}{lr}
\pm \frac{1}{\sqrt{H}} (S_l-\psi_{m_{\sigma}^-}),& \mathrm{if}\ k \in [\![1,m_{\sigma}^-]\!],\\[6pt]
\pm \frac{1}{\sqrt{H}}(S_l+\frac{h_{\sigma,m_{\sigma}^-}^-}{h_{\sigma,m_{\sigma}^-}^+} \psi_{m_{\sigma}^-}),& \mathrm{if}\ k \in [\![m_{\sigma}^-+1,n]\!],\\[6pt]
\frac{h_k}{H}( S_l - 2 \psi_{m_{\sigma}^-}),& \mathrm{if}\ k \in [\![n+1,n+m_{\sigma}^-]\!],\\[6pt]
\frac{h_k}{H}(S_l+2 \frac{h_{\sigma,m_{\sigma}^-}^-}{h_{\sigma,m_{\sigma}^-}^+} \psi_{m_{\sigma}^-}),& \mathrm{if}\ k \in [\![n+m_{\sigma}^-+1,2n]\!],\\[6pt]
0,& \mathrm{if}\ k \in [\![2n+1,3n]\!],
\end{array}
\right.
\end{equation}

\noindent
where $S_r:=\sum_{k=1}^n r_{n,1}^{\pm,k}$ and $S_l:= \sum_{k=1}^n l_{n,1}^{\pm,n+k}$. Finally, solutions of (\ref{1eigsystem4}--\ref{1eigsystem5}) are
\begin{equation}
\left\{
\begin{array}{ll}
\boldsymbol{r_{n,1}^{\pm}}({\bf u},\boldsymbol{\gamma})=& -\sum_{k=1}^{m_{\sigma}^-}  \frac{\psi_{m_{\sigma}^-}}{\sqrt{H}} \left( \frac{2 h_k}{\sqrt{H}} \boldsymbol{ e_{k}} \pm \boldsymbol{ e_{n+k}}\right)\\[4pt]
& + \sum_{k=m_{\sigma}^-+1}^{n}  \frac{\psi_{m_{\sigma}^-} h_{\sigma,m_{\sigma}^-}^-}{ \sqrt{H} h_{\sigma,m_{\sigma}^-}^+} \left( \frac{2 h_k}{\sqrt{H}} \boldsymbol{ e_{k}} \pm \boldsymbol{ e_{n+k}} \right),\\[8pt]
\boldsymbol{l_{n,1}^{\pm}}({\bf u},\boldsymbol{\gamma})=& -\sum_{k=1}^{m_{\sigma}^-}  \frac{\psi_{m_{\sigma}^-}}{\sqrt{H}} \left( \frac{2 h_k}{\sqrt{H}} {}^{\top}\boldsymbol{ e_{n+k}} \pm {}^{\top} \boldsymbol{ e_{k}}\right)\\[4pt]
& + \sum_{k=m_{\sigma}^-+1}^{n}  \frac{\psi_{m_{\sigma}^-} h_{\sigma,m_{\sigma}^-}^-}{ \sqrt{H} h_{\sigma,m_{\sigma}^-}^+} \left( \frac{2 h_k}{\sqrt{H}} {}^{\top} \boldsymbol{ e_{n+k}} \pm {}^{\top} \boldsymbol{ e_{k}} \right),
\end{array}
\right.
\label{1eigsystemsolution3}
\end{equation}

\noindent
and the approximations of the eigenvectors given in proposition \ref{1proplambdanexpandeigenvector} are verified.
\end{proof}

\noindent
{\em Remark:} The right eigenvectors of $\mathsf{A}({\bf u},\boldsymbol{\gamma},\theta)$: $\boldsymbol{r^{\lambda_n^{\pm}}}({\bf u},\boldsymbol{\gamma},\theta)$, associated to $\lambda_n^{\pm}({\bf u},\boldsymbol{\gamma},\theta)$, are defined by
\begin{equation}
\boldsymbol{r^{\lambda_n^{\pm}}}({\bf u},\boldsymbol{\gamma},\theta)=\mathsf{P}(\theta)^{-1} \boldsymbol{r_x^{\lambda_n^{\pm}}}(\mathsf{P}(\theta){\bf u},\boldsymbol{\gamma}),
\label{1eigvectleftbarotrop1}
\end{equation}
and the left ones: $\boldsymbol{l^{\lambda_n^{\pm}}}({\bf u},\boldsymbol{\gamma},\theta)$ are defined by
\begin{equation}
\boldsymbol{l^{\lambda_n^{\pm}}}({\bf u},\boldsymbol{\gamma},\theta)=\boldsymbol{l_x^{\lambda_n^{\pm}}}(\mathsf{P}(\theta){\bf u},\boldsymbol{\gamma})\mathsf{P}(\theta).
\label{1eigvectleftbarotrop}
\end{equation}

\noindent
To sum this subsection up, considering the asymptotic regime \eqref{1hypasympt}, \eqref{1hypasympt2} and assuming \eqref{1assumpteigenvectorlambdan}, we proved the expressions of the perturbations of the right and left eigenvectors associated to the barotropic eigenvalues, with a precision about $\mathcal{O}(\epsilon)$.


\noindent
In the next subsection, we give the asymptotic expansions of the right and left eigenvectors associated to the baroclinic eigenvalues.

\subsection{The baroclinic eigenvectors}

\noindent
With the asymptotic expansions of the baroclinic eigenvalues in \eqref{1lambdaipmasympt}, we can deduce asymptotic expansions of the baroclinic eigenvectors.

\begin{proposition}
Let $({\bf u},\boldsymbol{\gamma}) \in \mathbb{R}^{3n} \times ]0,1[^{n-1}$, $\epsilon >0$ and an injective function $\sigma \in \mathbb{R}_+^{*\ [\![1,n-1]\!]}$ such that $\boldsymbol{\gamma}$ verifies \eqref{1hypasympt}, ${\bf u}$ verifies \eqref{1hypasympt2} and
\begin{equation}
\epsilon \ll 1.
\end{equation}

\noindent
Then, for all $i \in [\![1,n-1]\!]$, the asymptotic expansions of the right eigenvectors associated to $\lambda_i^{\pm}({\bf u},\boldsymbol{\gamma})$, with precision about $\mathcal{O}(\epsilon^{\frac{\sigma(i)+1}{2}})$, is such that
\begin{equation}
\begin{array}{ll}
\boldsymbol{r_x^{\lambda_{i}^{\pm}}}({\bf u},\boldsymbol{\gamma})= & \sum_{j=m_i^-}^{i}\frac{\boldsymbol{ e_j}}{i-m_i^- +1} - \sum_{j=i+1}^{m_i^+}\frac{\boldsymbol{ e_{j}}}{m_i^+ -i}\\[4pt]
&+\frac{u_{i+1}-u_i}{h_{\sigma,i}^- +h_{\sigma,i}^+}\left(\sum_{j=m_i^-}^{i} \frac{h_{\sigma,i}^- \boldsymbol{e_{n+j}}}{(i-m_i^- +1)h_j}+\sum_{j=i+1}^{m_i^+}\frac{h_{\sigma,i}^+ \boldsymbol{e_{n+j}}}{(m_i^+ -i)h_{j}}\right)\\[4pt]
& \pm \left[ \frac{h_{\sigma,i}^- h_{\sigma,i}^+}{h_{\sigma,i}^- + h_{\sigma,i}^+} \left( 1-\gamma_i - \frac{(u_{i+1} - u_i)^2}{h_{\sigma,i}^- + h_{\sigma,i}^+} \right) \right]^{\frac{1}{2}}\sum_{j=m_i^-}^{i} \frac{\boldsymbol{e_{n+j}}}{(i-m_i^- +1)h_j}\\[4pt]
&\mp \left[ \frac{h_{\sigma,i}^- h_{\sigma,i}^+}{h_{\sigma,i}^- + h_{\sigma,i}^+} \left( 1-\gamma_i - \frac{(u_{i+1} - u_i)^2}{h_{\sigma,i}^- + h_{\sigma,i}^+} \right) \right]^{\frac{1}{2}}\sum_{j=i+1}^{m_i^+}\frac{\boldsymbol{e_{n+j}}}{(m_i^+ -i)h_{j}}\\[4pt]
& +\mathcal{O}(\epsilon^{\frac{\sigma(i)+1}{2}}),
\end{array}
\label{1eigvectlambdipm21}
\end{equation}
and the asymptotic expansions of the left eigenvectors associated to $\lambda_i^{\pm}({\bf u},\boldsymbol{\gamma})$, with precision about $\mathcal{O}(\epsilon^{\frac{\sigma(i)+1}{2}})$, is such that
\begin{equation}
\begin{array}{ll}
\boldsymbol{l_x^{\lambda_{i}^{\pm}}}({\bf u},\boldsymbol{\gamma})= & \sum_{j=m_i^-}^{i} \frac{{}^{\top}\boldsymbol{ e_{n+j}}}{i-m_i^- +1} - \sum_{j=i+1}^{m_i^+} \frac{{}^{\top} \boldsymbol{ e_{n+j}}}{m_i^+ -i}\\[4pt]
& +\frac{u_{i+1}-u_i}{h_{\sigma,i}^- +h_{\sigma,i}^+}\left(\sum_{j=m_i^-}^{i}\frac{h_{\sigma,i}^- {}^{\top}{\boldsymbol{e_{j}}}}{(i-m_i^- +1)h_j}+\sum_{j=i+1}^{m_i^+}\frac{h_{\sigma,i}^+ {}^{\top}{\boldsymbol{e_{j}}}}{(m_i^+ -i)h_{j}}\right)\\[4pt]
& \pm \left[ \frac{h_{\sigma,i}^- h_{\sigma,i}^+}{h_{\sigma,i}^- + h_{\sigma,i}^+} \left( 1-\gamma_i - \frac{(u_{i+1} - u_i)^2}{h_{\sigma,i}^- + h_{\sigma,i}^+} \right) \right]^{\frac{1}{2}}\sum_{j=m_i^-}^{i}\frac{{}^{\top}\boldsymbol{e_{j}}}{(i-m_i^- +1)h_j}\\[4pt]
& \mp \left[ \frac{h_{\sigma,i}^- h_{\sigma,i}^+}{h_{\sigma,i}^- + h_{\sigma,i}^+} \left( 1-\gamma_i - \frac{(u_{i+1} - u_i)^2}{h_{\sigma,i}^- + h_{\sigma,i}^+} \right) \right]^{\frac{1}{2}} \sum_{j=i+1}^{m_i^+}\frac{{}^{\top}\boldsymbol{e_{j}}}{(m_i^+ -i)h_{j}}\\[4pt]
& +\mathcal{O}(\epsilon^{\frac{\sigma(i)+1}{2}}).\end{array}\\
\label{1eigvectlambdipm22}
\end{equation}
\label{1thmeigvectlambdipm2}
\end{proposition}

\begin{proof}
We consider $\lambda_i^{\pm}({\bf u},\boldsymbol{\gamma}) \in \mathbb{R}$, then, according to the asymptotic expansions of the proposition \ref{1proplambdaipmapprox},
\begin{equation}
\lambda_i^{\pm}({\bf u},\boldsymbol{\gamma})=u_i+\chi_{\sigma,i}^{\pm} \epsilon^{\frac{\sigma(i)}{2}} +\mathcal{O}(\epsilon^{\frac{\sigma(i)+1}{2}}),
\label{1reformullambdai2}
\end{equation}
where $\chi_{\sigma,i}^{\pm} := \pi_i \frac{h_{\sigma,i}^-}{h_{\sigma,i}^- +h_{\sigma,i}^+} \pm \left[ \frac{h_{\sigma,i}^- h_{\sigma,i}^+}{h_{\sigma,i}^- +h_{\sigma,i}^+}\left( 1-\frac{\pi_i^2}{h_{\sigma,i}^- +h_{\sigma,i}^+} \right)  \right]^{\frac{1}{2}}$. We expand the eigenvectors $\boldsymbol{r_x^{\lambda_i^{\pm}}}({\bf u},\boldsymbol{\gamma})$ and $\boldsymbol{l_x^{\lambda_i^{\pm}}}({\bf u},\boldsymbol{\gamma})$ such that
\begin{equation}
\forall  i \in [\![1,n-1]\!],\
\left\{
\begin{array}{l}
\boldsymbol{r_x^{\lambda_i^{\pm}}}({\bf u},\boldsymbol{\gamma})= \boldsymbol{r_{i,0}}({\bf u},\boldsymbol{\gamma})+\epsilon^{\frac{\sigma(i)}{2}} \boldsymbol{r_{i,1}^{\pm}}({\bf u},\boldsymbol{\gamma})+\mathcal{O}(\epsilon^{\frac{\sigma(i)+1}{2}}),\\[6pt]
\boldsymbol{l_x^{\lambda_i^{\pm}}}({\bf u},\boldsymbol{\gamma})= \boldsymbol{l_{i,0}}({\bf u},\boldsymbol{\gamma})+\epsilon^{\frac{\sigma(i)}{2}} \boldsymbol{l_{i,1}^{\pm}}({\bf u},\boldsymbol{\gamma})+\mathcal{O}(\epsilon^{\frac{\sigma(i)+1}{2}}),
\end{array}
\right.
\label{1eigvectoexpand2}
\end{equation}
\noindent
where $\boldsymbol{r_{i,0}}({\bf u},\boldsymbol{\gamma})$ and $\boldsymbol{l_{i,0}}({\bf u},\boldsymbol{\gamma})$ are right and left eigenvectors of the matrix $\mathsf{A}_x({\bf u}_{\boldsymbol{\sigma}}^{\boldsymbol{i}},\boldsymbol{\gamma^i_{\sigma}})$, associated to the eigenvalue $u_i$:
\begin{equation}
\forall  i \in [\![1,n-1]\!],\
\left\{
\begin{array}{l}
\boldsymbol{r_{i,0}}({\bf u},\boldsymbol{\gamma}):=\sum_{j=m_i^-}^{i}\frac{\boldsymbol{ e_j}}{i-m_i^- +1} - \sum_{j=i+1}^{m_i^+}\frac{\boldsymbol{ e_{j}}}{m_i^+ -i},\\
\boldsymbol{l_{i,0}}({\bf u},\boldsymbol{\gamma}):=\sum_{j=m_i^-}^{i} \frac{{}^{\top}\boldsymbol{ e_{n+j}}}{i-m_i^- +1} - \sum_{j=i+1}^{m_i^+} \frac{{}^{\top} \boldsymbol{ e_{n+j}}}{m_i^+ -i}.
\end{array}
\right.
\end{equation}

\noindent
Moreover, we have
\begin{equation}
\begin{array}{ll}
\mathsf{A}_x({\bf u},\boldsymbol{\gamma})=& \mathsf{A}_x({\bf u}_{\boldsymbol{\sigma}}^{\boldsymbol{i}},\boldsymbol{\gamma^i_{\sigma}})+\sum_{k=m_i^- }^{m_i^+}(u_{k}-u_i) \mathsf{A}_x^{k-1,1},\\[6pt]
&+(1-\gamma_{m_i^-})\mathsf{A}_x^{k,2}(\boldsymbol{\gamma}) + \sum_{k=m_i^- +1}^{m_i^+ -1}(1-\gamma_k)\mathsf{A}_x^{k,3}(\boldsymbol{\gamma}),
\end{array}
\label{1equalityA1A223}
\end{equation}
where the $3n \times 3n$ matrices, $\mathsf{A}_x^{k-1,1}$, $\mathsf{A}_x^{k,2}(\boldsymbol{\gamma})$ and $\mathsf{A}_x^{k,3}(\boldsymbol{\gamma})$ are defined respectively in (\ref{1defA1}--\ref{1defA2}) and \eqref{1defA3}, and $({\bf u}_{\boldsymbol{\sigma}}^{\boldsymbol{i}},\boldsymbol{\gamma^i_{\sigma}})$ is defined in \eqref{1usigmagammaj}. Consequently,
\begin{equation}
\mathsf{A}_x({\bf u},\boldsymbol{\gamma})= \mathsf{A}_x({\bf u}_{\boldsymbol{\sigma}}^{\boldsymbol{i}},\boldsymbol{\gamma^i_{\sigma}})+ \sum_{k=i+1}^{m_i^+} \pi_i \epsilon^{\frac{\sigma(i)}{2}} \mathsf{A}_x^{k-1,1}+\mathcal{O}(\epsilon^{\frac{\sigma(i)+1}{2}})
\label{1equalityA1A23}
\end{equation}

\noindent
Therefore, $\boldsymbol{r_x^{\lambda_i^{\pm}}}({\bf u},\boldsymbol{\gamma})$ and $\boldsymbol{l_x^{\lambda_i^{\pm}}}({\bf u},\boldsymbol{\gamma})$ are respectively the approximations of the right and left eigenvectors associated to $\lambda_i^{\pm}({\bf u},\boldsymbol{\gamma})$, with a precision about $\mathcal{O}(\epsilon^{\frac{\sigma(i)+1}{2}})$, if and only if $\boldsymbol{r_{i,1}^{\pm}}({\bf u},\boldsymbol{\gamma})$ and $\boldsymbol{l_{i,1}^{\pm}}({\bf u},\boldsymbol{\gamma})$ verify:

\begin{equation}
\left\{
\begin{array}{l}
\big(\sum_{k=i+1}^{m_i^+}\pi_i \mathsf{A}_x^{k-1,1}-\chi_{\sigma,i}^{\pm} \mathsf{I}_{3n}\big)\boldsymbol{r_{i,0}}({\bf u},\boldsymbol{\gamma})=-\big(\mathsf{A}_x({\bf u}_{\sigma}^{\boldsymbol{i}},\boldsymbol{\gamma_{\sigma}^i})-u_i \mathsf{I}_{3n}\big)\boldsymbol{r_{i,1}^{\pm}}({\bf u},\boldsymbol{\gamma}),\\[6pt]
\boldsymbol{l_{i,0}}({\bf u},\boldsymbol{\gamma}) \big(\sum_{k=i+1}^{m_i^+}\pi_i \mathsf{A}_x^{k-1,1}-\chi_{\sigma,i}^{\pm} \mathsf{I}_{3n}\big)=-\boldsymbol{l_{i,1}^{\pm}}({\bf u},\boldsymbol{\gamma})\big(\mathsf{A}_x({\bf u}_{\sigma}^{\boldsymbol{i}},\boldsymbol{\gamma_{\sigma}^i})-u_i \mathsf{I}_{3n}\big),
\end{array}
\right.
\label{1eigsystem2}
\end{equation}
\noindent
Finally, a solution of \eqref{1eigsystem2} is
\begin{equation}
\left\{
\begin{array}{l}
\boldsymbol{r_{i,1}^{\pm}}({\bf u},\boldsymbol{\gamma})=\sum_{j=m_i^-}^{i} \frac{\chi_{\sigma,i}^{\pm} }{(i-m_i^- +1)h_j}\boldsymbol{ e_{n+j}}-\sum_{j=i+1}^{m_i^+} \frac{\chi_{\sigma,i}^{\pm}-\pi_i}{(m_i^+ -i)h_{j}}\boldsymbol{ e_{n+j}} ,\\[6pt]
\boldsymbol{l_{i,1}^{\pm}}({\bf u},\boldsymbol{\gamma})=\sum_{j=m_i^-}^{i}\frac{\chi_{\sigma,i}^{\pm} }{(i-m_i^- +1)h_j}{}^{\top} \boldsymbol{ e_{j}}  - \sum_{j=i+1}^{m_i^+}\frac{\chi_{\sigma,i}^{\pm}-\pi_i}{(m_i^+ -i)h_{j}} {}^{\top} \boldsymbol{ e_{j}},
\end{array}
\right.
\label{1eigsystemsolution2}
\end{equation}
\noindent
and the approximations of the eigenvectors given in proposition \ref{1thmeigvectlambdipm2} are verified.
\end{proof}

\noindent
{\em Remark:} 1) For all $i \in [\![1,n-1]\!]$, the right eigenvectors of $\mathsf{A}({\bf u},\boldsymbol{\gamma},\theta)$, $\boldsymbol{r^{\lambda_i^{\pm}}}({\bf u},\boldsymbol{\gamma},\theta)$, associated to the baroclinic eigenvalues, are defined by
\begin{equation}
\boldsymbol{r^{\lambda_i^{\pm}}}({\bf u},\boldsymbol{\gamma},\theta)=\mathsf{P}^{-1}(\theta)\boldsymbol{r_x^{\lambda_i^{\pm}}}(\mathsf{P}(\theta){\bf u},\boldsymbol{\gamma}),
\label{1eigvectright}
\end{equation}
and the left eigenvectors of $\mathsf{A}({\bf u},\boldsymbol{\gamma},\theta)$, $\boldsymbol{l^{\lambda_i^{\pm}}}({\bf u},\boldsymbol{\gamma},\theta)$, associated to these eigenvalues, are defined by
\begin{equation}
\boldsymbol{l^{\lambda_i^{\pm}}}({\bf u},\boldsymbol{\gamma},\theta)=\boldsymbol{l_x^{\lambda_i^{\pm}}}(\mathsf{P}(\theta){\bf u},\boldsymbol{\gamma})\mathsf{P}(\theta).
\label{1eigvectleft}
\end{equation}

\noindent
2) Note that the asymptotic expansions \eqref{1expressionlefteigenvectn} for the barotropic eigenvectors and \eqref{1eigvectlambdipm22} for the baroclinic eigenvectors, are necessary to characterize the {\em Riemann} invariants, $r_{\lambda,x}({\bf u},\boldsymbol{\gamma})$, for all $\lambdaÊ\in \sigma(\mathsf{A}_x({\bf u},\boldsymbol{\gamma}))$, such that
\begin{equation}
\exists\ \alpha({\bf u},\boldsymbol{\gamma})>0,\ ({}^{\top}\boldsymbol{l^{\lambda}_x}({\bf u},\boldsymbol{\gamma})\cdot \frac{\partial {\bf u}}{\partial t})=\alpha({\bf u},\boldsymbol{\gamma}) \frac{\partial r_{\lambda,x}({\bf u},\boldsymbol{\gamma})}{\partial t}.
\label{1defriemanninv}
\end{equation}
However, it is possible that this last equation has no explicit solution, $r_{\lambda,x}({\bf u},\boldsymbol{\gamma})$, but the asymptotic expansion performed in this paper is still useful for a numerical resolution: we can approximately integrate the equation \eqref{1defriemanninv}.

\noindent
To conclude, we proved the expression of the asymptotic expansions of the baroclinic eigenvectors, considering the asymptotic regime \eqref{1hypasympt} and assuming the heights of each layer have all the same range and the difference of velocity between an interface has the same order as the square root of the relative difference of density at this interface. Moreover, the expansions of these eigenvectors, $\boldsymbol{r_x^{\lambda_i^{\pm}}}({\bf u},\boldsymbol{\gamma})$ and $\boldsymbol{l_x^{\lambda_i^{\pm}}}({\bf u},\boldsymbol{\gamma})$, for $i \in [\![1,n-1]\!]$, have been performed with a precision about $ \mathcal{O}(\epsilon^{\frac{\sigma(i)+1}{2}})$.

\noindent
In the next subsection, we deduce criterion of local well-posedness of the model \eqref{1systemmultilayer}, more general than \eqref{1condwellposed}.

\subsection{Local well-posedness of the model}
Using the previous asymptotic expansions, it is possible to prove the local well-posedness of the multi-layer shallow water model with free surface, in two space-dimensions.

\noindent
First, we can prove the next proposition
\begin{proposition}
Let $({\bf u},\boldsymbol{\gamma}) \in \mathbb{R}^{3n} \times ]0,1[^{n-1}$, $\theta \in [0,2\pi]$, $\epsilon >0$ and an injective function $\sigma \in \mathbb{R}_+^{*\ [\![1,n-1]\!]}$ such that $\boldsymbol{\gamma}$ verifies \eqref{1hypasympt}, ${\bf u}$ verifies \eqref{1hypasympt2}.

\noindent
Then, there exists $\delta >0$ such that if
\begin{equation}
\epsilon \le \delta,
\label{1hypepsilonpetit}
\end{equation}
the matrix $\mathsf{A}({\bf u},\boldsymbol{\gamma},\theta)$ is diagonalizable with real eigenvalues if
\begin{equation}
\left\{
\begin{array}{lr}
h_i > 0, & \forall i \in [\![1,n]\!],\\[6pt]
\phi_{\sigma,i}(\boldsymbol{h})-|u_{i+1}-u_i|^2-|v_{i+1}-v_i|^2 > 0,& \forall i \in [\![1,n-1]\!],
\end{array}\right.
\label{1condnechyp3}
\end{equation}
where $\phi_{\sigma,i}(\boldsymbol{h}):=(h_{\sigma,i}^- +h_{\sigma,i}^+)(1-\gamma_i)$.
\label{1diagA}
\end{proposition}

\begin{proof}
With the rotational invariance \eqref{1rotinv}, it is equivalent to prove the diagonalizability of $\mathsf{A}_x({\bf u},\boldsymbol{\gamma})$. Assuming \eqref{1hypasympt}, \eqref{1hypasympt2} and \eqref{1hypepsilonpetit}, according to \eqref{1vectpropVx} and the propositions \ref{1proplambdanexpandeigenvector} and \ref{1thmeigvectlambdipm2}, the right eigenvectors
\begin{equation}
\left(\boldsymbol{r_x^{\lambda_i^{\pm}}}({\bf u},\boldsymbol{\gamma})\right)_{i \in [\![1,n]\!]} \cup \left( \boldsymbol{r_x^{\lambda_i}}({\bf u},\boldsymbol{\gamma}) \right)_{i \in [\![2n+1,3n]\!]}
\end{equation}
constitute an eigenbasis of $\mathbb{R}^{3n}$ if \eqref{1condnechyp3} is verified. Indeed, the conditions \eqref{1condnechyp3} are necessary to insure the eigenvectors are in $\mathbb{R}^{3n}$ and it is a basis of this vector-space because: for $i \in [\![2n+1,3n]\!]$, giving a vector $\boldsymbol{r_x^{\lambda_i}}({\bf u},\boldsymbol{\gamma})$, it is obvious to find back $i$; for $i \in [\![1,n]\!]$, giving $\boldsymbol{r_x^{\lambda_i^{\pm}}}({\bf u},\boldsymbol{\gamma})$ it is also easy to detect $i$ --- where the sign of the $1^{\mathrm{st}}$ coordinates changes --- and, according to the strict inequalities \eqref{1condnechyp3},
\begin{equation}
\boldsymbol{r_x^{\lambda_i^{+}}}({\bf u},\boldsymbol{\gamma}) \not= \boldsymbol{r_x^{\lambda_i^{-}}}({\bf u},\boldsymbol{\gamma}).
\end{equation}

\noindent
Then, if the inequalities \eqref{1condnechyp3} are assumed, $\mathsf{A}_x({\bf u},\boldsymbol{\gamma})$ is diagonalizable with real eigenvalues and the proposition \ref{1diagA} is proved.
\end{proof}

\noindent
{\em Remark:} According to the propositions \ref{1proplambdanexpandeigenvector2} and \ref{1proplambdaipmapprox}, we would expect, in the asymptotic regime \eqref{1hypasympt}, \eqref{1hypasympt2} and $\epsilon \le \delta$, that
\begin{equation}
\lambda_n^- < \lambda_{m_{\sigma}^-}^- < \ldots < \lambda_{m_{\sigma}^+}^- < \lambda_{m_{\sigma}^+}^+ < \ldots < \lambda_{m_{\sigma}^-}^+ < \lambda_n^+.
\label{1ordreeigenvalues}
\end{equation}
In the particular case of two layers, these inequalities are true, in this asymptotic regime. Moreover, in the case of $n$ layers, with $n \ge 3$, if we assume for all $i \in [\![1,n-2]\!]$, $\pi_i \pi_{i+1} > 0$ and
\begin{equation}
\left[ \frac{h_{\sigma,i}^- h_{\sigma,i}^+}{h_{\sigma,i}^- + h_{\sigma,i}^+} \left( 1-\gamma_i - \frac{(u_{i+1} - u_i)^2}{h_{\sigma,i}^- + h_{\sigma,i}^+} \right) \right]^{\frac{1}{2}} \ll \frac{|\pi_i| h_{\sigma,i}^-}{h_{\sigma,i}^- + h_{\sigma,i}^+},
\end{equation}
then \eqref{1ordreeigenvalues} remains true, as we can deduce that
\begin{equation}
\forall i \in [\![1,n-1]\!],\ \min(u_i,u_{i+1}) < \lambda_i^{\pm} < \max(u_i,u_{i+1}),
\end{equation}
and the diagonalizability of the matrix $\mathsf{A}({\bf u},\boldsymbol{\gamma},\theta)$ is directly deduced. However, in the general case, \eqref{1ordreeigenvalues} is not verified: to prove the diagonalizability of $\mathsf{A}({\bf u},\boldsymbol{\gamma},\theta)$, we need the entirely eigenstructure of this matrix.

\noindent
Finally, as a consequence of the previous proposition, we deduce a criterion of local well-posedness in $\mathcal{H}^s(\mathbb{R}^2)^{3n}$, more general than criterion \eqref{1condwellposed}.

\begin{theorem}
Let $s >2$, $\boldsymbol{\gamma} \in ]0,1[^{n-1}$, $\epsilon >0$ and an injective function $\sigma: [\![1,n-1]\!] \rightarrow \mathbb{R}_+^*$ such that $\boldsymbol{\gamma}$ verifies \eqref{1hypasympt}.

\noindent
Then, there exists $\delta > 0$ such that if
\begin{equation}
\epsilon \le \delta,
\label{1assumpttheoremwellposed}
\end{equation}
the {\em Cauchy} problem, associated with {\em \eqref{1systemmultilayer}} and initial data ${\bf u^0} \in \mathcal{H}_{\sigma,\epsilon} \cap \mathcal{H}^s(\mathbb{R}^2)^{3n}$, is hyperbolic, locally well-posed in $\mathcal{H}^s(\mathbb{R}^2)^{3n}$ and the unique solution verifies conditions {\em \eqref{1strongsol}}.
\label{1wellprosedhyp2d}
\end{theorem}
\begin{proof}
Let $({\bf u^0},\boldsymbol{\gamma}) \in \mathcal{H}^s(\mathbb{R}^2)^{3n} \times ]0,1[^{n-1}$ such that conditions \eqref{1hypasympt}, \eqref{1hypasympt2} and \eqref{1assumpttheoremwellposed} are verified. As it was proved in the proposition \ref{1diagA}, for all $(X,\theta) \in \mathbb{R}^2 \times [0,2\pi]$, $\mathsf{A}({\bf u^0}(X),\boldsymbol{\gamma},\theta)$ is diagonalizable, with real eigenvalues, if ${\bf u} \in \mathcal{H}_{\sigma,\epsilon}$. Then, the {\em Cauchy} problem is hyperbolic. Moreover, according to the proposition \ref{1diagimpliqsym}, it is locally well-posed in $\mathcal{H}^s(\mathbb{R}^2)^{3n}$ and the unique solution verifies conditions \eqref{1strongsol}.
\end{proof}

\noindent
{\em Remark:} This criterion is less restrictive than \eqref{1condwellposed}, because, as it was proved in proposition \ref{1prop1subsethypsym}, if $({\bf u},\boldsymbol{\gamma})$ verifies these conditions and $\epsilon$ is sufficiently small, $\mathcal{S}^s_{\gamma} \subset \mathcal{H}_{\sigma,\epsilon} \cap \mathcal{H}^s(\mathbb{R}^2)^{3n}$.

\noindent
In conclusion, we proved the expression of the asymptotic expansions of the baroclinic eigenvectors, considering the asymptotic regime \eqref{1hypasympt} and assuming the heights of each layer have all the same range and the difference of velocity between an interface has the same order as the square root of the relative difference of density at this interface. Moreover, the expansions of these eigenvectors, $\boldsymbol{r_x^{\lambda_i^{\pm}}}({\bf u},\boldsymbol{\gamma})$ and $\boldsymbol{l_x^{\lambda_i^{\pm}}}({\bf u},\boldsymbol{\gamma})$, for $i \in [\![1,n-1]\!]$, have been performed with a precision about $ \mathcal{O}(\epsilon^{\frac{\sigma(i)+1}{2}})$, permitting to give a condition of local well-posedness of the multi-layer shallow water model with free surface, in two dimensions.

\noindent
In the next subsection, we deduce from the asymptotic expansions of the eigenstructure of $\mathsf{A}_x({\bf u},\boldsymbol{\gamma})$, the nature of the waves associated to each eigenvalues.

\subsection{Nature of the waves}
In order to know the type of the wave associated to each eigenvalue -- shock, contact or rarefaction wave -- there is the next proposition
\begin{proposition}
Let $\boldsymbol{\gamma} \in ]0,1[^{n-1}$, $\epsilon >0$ and an injective function $\sigma \in \mathbb{R}_+^{*\ [\![1,n-1]\!]}$ such that $\boldsymbol{\gamma}$ verifies \eqref{1hypasympt}.

\noindent
Then, there exists $\delta > 0$ such that if
\begin{equation}
\epsilon \le \delta,
\label{1assumpteigenvectorlambdan4}
\end{equation}
and ${\bf u} \in \mathcal{H}_{\sigma,\epsilon}$, we have
\begin{equation}
\left\{
\begin{array}{lr}
\mathrm{the}\ \lambda_i^{\pm}\mathrm{-characteristic\ field\ is\ genuinely\ non-linear},&\mathrm{if}\ i \in [\![1,n]\!],\\
\mathrm{the}\ \lambda_i\mathrm{-characteristic\ field\ is\ linearly\ degenerate},&\mathrm{if}\ i \in [\![2n+1,3n]\!].
\end{array}
\right.
\label{1muwavenature}
\end{equation}
\label{1genuinelynonlin}
\end{proposition}
\begin{proof}
If $({\bf u},\boldsymbol{\gamma})$ verify these assumptions, the asymptotic expansions \eqref{1approxlambdanpm3} and \eqref{1lambdaipmasympt} are valid. Moreover, we remark that for all $i \in [\![1,n]\!],\ \lambda_i^{\pm}$ depends analytically of the parameters of the problem and we deduce that the error of the asymptotic expansions still remains small after derivating. Then, with the right eigenvectors \eqref{1vectpropVx} and the asymptotic expansions of the right eigenvectors \eqref{1eigvectlambdnpmr2} and \eqref{1eigvectlambdipm21} of $\mathsf{A}_x({\bf u},\boldsymbol{\gamma})$, one can check that

\begin{equation}
\left\{
\begin{array}{lr}
\nabla \lambda_i^{\pm} \cdot \boldsymbol{r}_x^{\lambda_i^{\pm}}({\bf u}) \not=0,& \mathrm{if}\ i \in [\![1,n]\!],,\\
\nabla \lambda_i \cdot \boldsymbol{r}_x^{\lambda}({\bf u}) =0,& \mathrm{if}\ i \in [\![2n+1,3n]\!].
\label{1genuidef}
\end{array}
\right.
\end{equation}
Then, the proposition \ref{1genuinelynonlin} is proved.
\end{proof}

\noindent
{\em Remark:} When for all $i \in [\![1,n-1]\!]$,  $u_{i+1}-u_i$ and $1-\gamma_i$ are all equal to $0$, the $\lambda_n^{\pm}$-characteristic field remains genuinely non-linear but the $\lambda_i^{\pm}$-characteristic field becomes linearly degenerate.

\noindent
To conclude, under assumptions \eqref{1hypasympt}, \eqref{1hypasympt2} and \eqref{1assumpteigenvectorlambdan4}, for all $i \in [\![1,n]\!]$, the $\lambda_i^{\pm}$-wave is a shock wave or a rarefaction wave and for all $i \in [\![2n+1,3n]\!]$ the $\lambda_i$-wave is a contact wave.

\section{A {\em conservative} multi-layer shallow water model}
Even if the model (\ref{1massconservation}--\ref{1momentumconservation}) is conservative, in the one-dimensional case, with the unknowns $(h_i,u_i)$, $i \in [\![1,n]\!]$, it is not anymore true in the two-dimensional case. This section will treat this lack of conservativity by an augmented model, with a different approach from \cite{abgrall2009two}. We remind that no assumption has been made concerning the horizontal vorticity, in each layer
\begin{equation}
\forall i \in [\![1,n]\!],\ w_i:=curl(\boldsymbol{ u_i})=\frac{\partial v_i}{\partial x} - \frac{\partial u_i}{\partial y}.
\label{1wi}
\end{equation}

\subsection{Conservation laws}
Using a {\em Frobenius} problem, it was proved in \cite{barros2006conservation} that the one-dimensional two-layer shallow water model with free surface has a finite number of conservative quantities: the height and velocity in each layer, the total momentum and the total energy. However, in the two-dimensional case, it is still an open question.

\noindent
Concerning the multi-layer model, in one dimension, we can also reduce the study of conservative quantities to the study of a {\em Frobenius} problem. Indeed, defining the new unknowns
\begin{equation}
\forall i \in [\![1,n]\!],\ 
\left\{
\begin{array}{l}
\hat{h}_i:=\alpha_{n,i} h_i,\\[4pt]
\hat{u}_i:=u_i,
\end{array}
\right.
\end{equation}
and
\begin{equation}
{\bf \hat{u}}:= {}^{\top}(\hat{h}_1,\ldots,\hat{h}_n,\hat{u}_n,\ldots,\hat{u}_n),
\end{equation}
then, the model \eqref{1systemmultilayer} is equivalent to
\begin{equation}
\frac{\partial {\bf \hat{u}}}{\partial t} + \mathsf{\hat{A}}_x({\bf \hat{u}},\boldsymbol{\gamma}) \frac{\partial {\bf \hat{u}}}{\partial x} + {\bf \hat{b}}({\bf \hat{u}})=0,
\label{1systemmultilayerhat}
\end{equation}
with
\begin{equation}
\mathsf{\hat{A}}_x({\bf \hat{u}},\boldsymbol{\gamma}):= \left[\begin{array}{c|c}
 \begin{array}{c} \mathsf{\Delta}  \end{array}
 & \begin{array}{c} \mathsf{0} \end{array}\\
 \hline
  \begin{array}{c} \mathsf{0} \end{array}
 & \begin{array}{c} \mathsf{I_n} \end{array}\\
\end{array}\right] \left[\begin{array}{c|c}
 \begin{array}{c} \mathsf{V}_x  \end{array}
 & \begin{array}{c} \mathsf{H} \end{array}\\
 \hline
  \begin{array}{c} \mathsf{\Gamma} \end{array}
 & \begin{array}{c} \mathsf{V}_x \end{array}\\
\end{array}\right].
\end{equation}
Moreover, we have also
\begin{equation}
\mathsf{\hat{A}}_x({\bf \hat{u}},\boldsymbol{\gamma})= \mathsf{P} \nabla^2 \hat{e}_1({\bf \hat{u}},\boldsymbol{\gamma}),
\label{1formhatA}
\end{equation}
where $\hat{e}_1({\bf \hat{u}}):=\frac{1}{2}\sum_{i=1}^n \hat{h}_i \left(\hat{u}_i^2+\frac{\hat{h}_i}{\alpha_{n,i}}\right)+\sum_{i=1}^{n-1}\sum_{j=i+1}^n \hat{h}_i \frac{\hat{h}_j}{\alpha_{n,j}}$ and the $2n \times 2n$ block matrix $\mathsf{P}$ is defined by
\begin{equation}
\mathsf{P} :=  \left[\begin{array}{c|c}
 \begin{array}{c} \mathsf{0}  \end{array}
 & \begin{array}{c} \mathsf{I_n} \end{array}\\
 \hline
  \begin{array}{c} \mathsf{I_n} \end{array}
 & \begin{array}{c} \mathsf{0} \end{array}\\
\end{array}\right],
\end{equation}

\noindent
Therefore, $\eta({\bf \hat{u}})$ is a conservative quantity of the multi-layer model, in one dimension, if and only if the matrix $\nabla^2 \eta({\bf \hat{u}})\ \mathsf{\hat{A}}_x({\bf \hat{u}})$ is symmetric, which is equivalent to, according to \eqref{1formhatA},
\begin{equation}
(\mathsf{P} \nabla^2 \eta({\bf \hat{u}})) \mathsf{\hat{A}}_x({\bf \hat{u}})=\mathsf{\hat{A}}_x({\bf \hat{u}}) (\mathsf{P} \nabla^2 \eta({\bf \hat{u}})).
\end{equation}

\noindent
Consequently, if we denote by $\mathsf{X}:=\mathsf{P} \nabla^2 \eta({\bf \hat{u}})$, the conservative quantities of \eqref{1systemmultilayerhat} needs to verify the {\em Frobenius} problem:
\begin{equation}
\mathsf{X}\ \mathsf{\hat{A}}_x({\bf \hat{u}})=\mathsf{\hat{A}}_x({\bf \hat{u}})\ \mathsf{X}.
\label{1frobeniusproblem}
\end{equation}

\noindent
{\em Remark:} The condition \eqref{1frobeniusproblem} is just necessary: the solution $\mathsf{X}:=[\mathsf{X}_{i,j}]_{(i,j) \in [\![1,2n]\!]^2}$ needs to verify the compatibility conditions
\begin{equation}
\forall (i,j,k) \in [\![1,2n]\!]^3,\ \frac{\partial \mathsf{X}_{i,j}}{\partial \alpha_k}=\frac{\partial \mathsf{X}_{i,k}}{\partial \alpha_j},
\label{1conditioncompatibilityhessian}
\end{equation}
where for all $k \in [\![1,n]\!]$, $\alpha_k:=\hat{h}_k$ and $\alpha_{n+k}:=\hat{u}_k$, to insure that $\mathsf{X}$ is the hessian of a scalar field.

\noindent
We remind a useful property of the set of the solutions of \eqref{1frobeniusproblem}:

\begin{proposition}
Let $\boldsymbol{\gamma} \in ]0,1[^{n-1}$, $\epsilon >0$ and an injective function $\sigma \in \mathbb{R}_+^{*\ [\![1,n-1]\!]}$ such that $\boldsymbol{\gamma}$ verifies \eqref{1hypasympt}.

\noindent
Then, there exists $\delta > 0$ such that if
\begin{equation}
\epsilon \le \delta,
\label{1assumpteigenvectorlambdan6}
\end{equation}
and ${\bf u} \in \mathcal{H}_{\sigma,\epsilon}$, a matrix $\mathsf{X}$ is solution of the {\em Frobenius} problem \eqref{1frobeniusproblem} if and only if
\begin{equation}
\mathsf{X} \in \mathrm{Span}\big( \mathsf{\hat{A}}^k_x({\bf \hat{u}}),\ k \in [\![0,2n-1]\!] \big).
\end{equation}
\end{proposition}
\begin{proof}
Under these conditions, as it was proved in propositions \ref{1proplambdanexpandeigenvector2} and \ref{1proplambdaipmapprox}, the eigenvalues of $\mathsf{\hat{A}}_x({\bf \hat{u}})$ are all distinct. Then, the characteristic polynomial of $\mathsf{\hat{A}}_x({\bf \hat{u}})$ coincides with the minimal polynomial. Therefore, the set of solutions of \eqref{1frobeniusproblem} is equal to $\mathrm{Span}\big( \mathsf{\hat{A}}^k_x({\bf \hat{u}}),\ k \in [\![0,2n-1]\!] \big)$.
\end{proof}

\noindent
Then, according to the last proposition, there exists $\big( x_i \big)_{i \in [\![1,2n-1]\!]} \subset \mathbb{R}$ such that
\begin{equation}
\nabla^2 \eta({\bf \hat{u}})=\sum_{i=0}^{2n-1} x_i \mathsf{P} \mathsf{\hat{A}}^k_x({\bf \hat{u}}).
\end{equation}

\noindent
Using the compatibility conditions \eqref{1conditioncompatibilityhessian}, we should find conditions on $\big( x_i \big)_{i \in [\![1,2n-1]\!]}$, to insure $\eta({\bf \hat{u}})$ to be a conservative quantity. However, the question is still open as the complexity of \eqref{1conditioncompatibilityhessian} is very high. However we would expect to find
\begin{equation}
\left\{
\begin{array}{lr}
x_i\ \mathrm{is}\ \mathrm{a}\ \mathrm{constant},& \forall i \in [\![0,1]\!],\\
x_i=0,& \forall i \in [\![2,2n-1]\!],
\end{array}
\right.
\end{equation}
to deduce that there exist $(x_0,x_1) \in \mathbb{R}^2$ and $({\bf c},{\bf d}) \in \mathbb{R}^{2n}$ such that
\begin{equation}
\eta({\bf \hat{u}})= \frac{x_0}{2} {\bf \hat{u}} \cdot \mathsf{P}{\bf \hat{u}}+x_1 \hat{e}_1({\bf \hat{u}}) + {\bf c} \cdot {\bf \hat{u}} + {\bf d},
\end{equation}
as the only known conservative quantities, in one dimension, are the height, the velocity in each layer, the total momentum and the total energy of the system.

\noindent
Concerning the conservative quantities of the multi-layer model, in two dimensions, the question is quite more complex and is also still open. Moreover, the study performed below does not remain possible --- the structure \eqref{1formhatA} is not anymore verified.

\noindent
Nevertheless, introducing $w_i$, for $i \in [\![1,n]\!]$, in equations (\ref{1massconservation}--\ref{1momentumconservation}), the conservation of mass \eqref{1massconservation} is unchanged
\begin{equation}
\frac{\partial h_i}{\partial t} + {\bf \nabla} {\bf \cdot} (h_i \boldsymbol{u_i}) = 0,
\label{1massconservationrelax}
\end{equation}
but the equation of depth-averaged horizontal velocity \eqref{1momentumconservation} becomes conservative
\begin{equation}
\frac{\partial \boldsymbol{u_i}}{\partial t}+{\bf \nabla} \left(\frac{1}{2} (u_i^2+v_i^2) +P_i \right) - (f+w_i) \boldsymbol{u_i}^{\bot} = 0.
\label{1momentumconservationrelax}
\end{equation}
Moreover, the horizontal vorticity, in each layer, is also conservative:
\begin{equation}
\frac{\partial w_i}{\partial t} + \nabla \cdot \left((w_i+f)\boldsymbol{u_i} \right)=0.
\label{1rot}
\end{equation}

\noindent
Therefore, in the two-dimensional case, there are at least $3n+2$ conservative quantities: the height, the velocity and the horizontal vorticity in each layer, the total momentum and the energy $e_2$:
\begin{equation}
e_2({\bf v},\boldsymbol{\gamma}):=\frac{1}{2}\sum_{i=1}^n \alpha_{n,i} h_i \left(u_i^2+v_i^2+h_i\right)+\sum_{i=1}^{n-1}\sum_{j=i+1}^n \alpha_{n,i} h_i h_j.
\label{1energy2d}
\end{equation}

\subsection{A new augmented model}
From equations (\ref{1massconservation}--\ref{1momentumconservation}), it is possible to obtain a new model. We denote $({\bf u},{\bf v}) \in \mathcal{H}^s(\mathbb{R}^2)^{3n} \times \mathcal{H}^s(\mathbb{R}^2)^{4n}$, the vectors defined by
\begin{equation}
\left\{
\begin{array}{l}
{\bf u}:={}^{\top} (h_1,\ldots,h_n,u_1,\ldots,u_n,v_1,\ldots,v_n),\\
{\bf v}:={}^{\top} (h_1,\ldots,h_n,u_1,\ldots,u_n,v_1,\ldots,v_n,w_1,\ldots,w_n).
\end{array}
\right.
\label{1uvdefinition}
\end{equation}

\noindent
If ${\bf u}$ is a classical solution of \eqref{1systemmultilayer}, then ${\bf v}$ is solution of the augmented system
\begin{equation}
\frac{\partial {\bf v}}{\partial t} + \mathsf{A}^{\mathsf{a}}_x({\bf v},\boldsymbol{\gamma}) \frac{\partial {\bf v}}{\partial x} + \mathsf{A}^{\mathsf{a}}_y ({\bf v},\boldsymbol{\gamma}) \frac{\partial {\bf v}}{\partial y} + {\bf b^a}({\bf v})=0,
\label{1systemmultilayerrelax}
\end{equation}
where the $4n \times 4n$ block matrices $\mathsf{A}^{\mathsf{a}}_x({\bf v},\boldsymbol{\gamma})$ and $ \mathsf{A}^{\mathsf{a}}_y({\bf v},\boldsymbol{\gamma})$ are defined by
\begin{equation}
\mathsf{A}^{\mathsf{a}}_x({\bf v},\boldsymbol{\gamma}) :=  \left[\begin{array}{c|c|c|c}
 \begin{array}{c} \mathsf{V}_x  \end{array}
 & \begin{array}{c} \mathsf{H} \end{array}
 & \begin{array}{c} \mathsf{0} \end{array}
 & \begin{array}{c} \mathsf{0} \end{array}\\
 \hline
 \begin{array}{c} \mathsf{\Gamma} \end{array}
 & \begin{array}{c} \mathsf{V}_x \end{array}
 & \begin{array}{c} \mathsf{V}_y \end{array}
 & \begin{array}{c} \mathsf{0} \end{array}\\
 \hline
 \begin{array}{c} \mathsf{0} \end{array}
 & \begin{array}{c} \mathsf{0} \end{array}
 & \begin{array}{c} \mathsf{0} \end{array}
 & \begin{array}{c} \mathsf{0} \end{array}\\
 \hline
 \begin{array}{c} \mathsf{0} \end{array}
 & \begin{array}{c} \mathsf{W} \end{array}
 & \begin{array}{c} \mathsf{0} \end{array}
 & \begin{array}{c} \mathsf{V}_x \end{array}\\
\end{array}\right],
\label{1Arx}
\end{equation}

\begin{equation}
\mathsf{A}^{\mathsf{a}}_y({\bf v},\boldsymbol{\gamma}) :=  \left[\begin{array}{c|c|c|c}
 \begin{array}{c} \mathsf{V}_y  \end{array}
 & \begin{array}{c} \mathsf{0} \end{array}
 & \begin{array}{c} \mathsf{H} \end{array}
 & \begin{array}{c} \mathsf{0} \end{array}\\
 \hline
 \begin{array}{c} \mathsf{0} \end{array}
 & \begin{array}{c} \mathsf{0} \end{array}
 & \begin{array}{c} \mathsf{0} \end{array}
 & \begin{array}{c} \mathsf{0} \end{array}\\
 \hline
 \begin{array}{c} \mathsf{\Gamma} \end{array}
 & \begin{array}{c} \mathsf{V}_x \end{array}
 & \begin{array}{c} \mathsf{V}_y \end{array}
 & \begin{array}{c} \mathsf{0} \end{array}\\
\hline
 \begin{array}{c} \mathsf{0} \end{array}
 & \begin{array}{c} \mathsf{0} \end{array}
 & \begin{array}{c} \mathsf{W} \end{array}
 & \begin{array}{c} \mathsf{V}_y \end{array}\\\end{array}\right],
\label{1Ary}
\end{equation}
where $\mathsf{W}:=\mathrm{diag}[w_i+f]_{i \in [\![1,n]\!]}$ and ${\bf b^a}({\bf v})$ is defined by
\begin{equation}
{\bf b^{\mathsf{a}}}({\bf v}):=  \sum_{k=1}^{n} \big(-(w_k+f) v_k +  \frac{\partial b}{\partial x}\big)  \boldsymbol{e'_{n+k}}+ \big((w_k+f) u_k +  \frac{\partial b}{\partial y}\big) \boldsymbol{e'_{2n+k}},
\label{1br}
\end{equation}
where $(\boldsymbol{e'_i})_{i \in [\![1,4n]\!]}$ denotes the canonical basis of $\mathbb{R}^{4n}$.

\noindent
Even if the model (\ref{1massconservation}--\ref{1momentumconservation}) is not conservative, the model \eqref{1systemmultilayerrelax} is always so. Then, there is no need to chose a conservative path in the numerical resolution.

\noindent
{\em Remark:} 1) $e_2({\bf v},\boldsymbol{\gamma})$ is not the total energy of the augmented model \eqref{1systemmultilayerrelax}. Indeed, it is never a convex function with the variable ${\bf v}$ as it is independent of $(w_i)_{i \in [\![1,n]\!]}$. 2) Let ${\bf v} \in \mathbb{R}^{4n}$, the associated vector ${\bf u} \in \mathbb{R}^{3n}$ will be composed of the $3n$ first coordinates of the vector ${\bf v}$. All the quantities or functions with ${\bf u}$ as a variable will refer to the non-augmented model \eqref{1systemmultilayer} and all the ones with ${\bf v}$, as a variable, will refer to the new augmented model \eqref{1systemmultilayerrelax}.

\begin{proposition}
The augmented model {\em \eqref{1systemmultilayerrelax}} verifies the rotational invariance.
\label{1proprotinvrelax}
\end{proposition}

\begin{proof}
We denote by $\mathsf{A}^{\mathsf{a}}({\bf v},\boldsymbol{\gamma},\theta)$ the matrix defined by
\begin{equation}
\cos(\theta) \mathsf{A}^{\mathsf{a}}_x({\bf v},\boldsymbol{\gamma})+\sin(\theta) \mathsf{A}^{\mathsf{a}}_y({\bf v},\boldsymbol{\gamma}).
\end{equation}
One can check the next equality, for all $({\bf v},\boldsymbol{\gamma},\theta) \in \mathbb{R}^{4n} \times \mathbb{R}_+^{*\ n-1} \times [0,2\pi]$
\begin{equation}
\mathsf{A}^{\mathsf{a}}({\bf v},\boldsymbol{\gamma},\theta)=\mathsf{P^{a}}(\theta)^{-1} \mathsf{A}^{\mathsf{a}}_x(\mathsf{P^a}(\theta) {\bf v},\boldsymbol{\gamma}) \mathsf{P^a}(\theta),
\label{1rotinvrelax}
\end{equation}
where $\mathsf{P^a}(\theta)$ is the $4n \times 4n$ block matrix defined by
\begin{equation}
\mathsf{P^a}(\theta) :=  \left[\begin{array}{c|c|c|c}
 \begin{array}{c} \mathsf{I_n}  \end{array}
 & \begin{array}{c} \mathsf{0} \end{array}
 & \begin{array}{c} \mathsf{0} \end{array}
 & \begin{array}{c} \mathsf{0} \end{array}\\
 \hline
 \begin{array}{c} \mathsf{0} \end{array}
 & \begin{array}{c} \cos(\theta) \mathsf{I_n} \end{array}
 & \begin{array}{c} \sin(\theta)\mathsf{I_n} \end{array}
 & \begin{array}{c} \mathsf{0} \end{array}\\
 \hline
 \begin{array}{c} \mathsf{0} \end{array}
 & \begin{array}{c} -\sin(\theta)\mathsf{I_n} \end{array}
 & \begin{array}{c} \cos(\theta) \mathsf{I_n} \end{array}
 & \begin{array}{c} \mathsf{0} \end{array}\\
  \hline
 \begin{array}{c} \mathsf{0} \end{array}
 & \begin{array}{c} \mathsf{0} \end{array}
 & \begin{array}{c} \mathsf{0} \end{array}
 & \begin{array}{c} \mathsf{I_n} \end{array}\\
\end{array}\right],
\label{1Prtheta}
\end{equation}

\noindent
and, moreover, we notice $\mathsf{P^a}(\theta)^{-1}={}^{\top}\mathsf{P^a}(\theta)$.
\end{proof}

\subsection{A rough criterion of local well-posedness}
We give a $1^{\mathrm{st}}$ criterion of {\em Friedrichs}-symmetrizability to insure the local well-posedness in $\mathcal{H}^s(\mathbb{R}^2)^{4n}$ and $\mathcal{L}^2(\mathbb{R}^2)^{4n}$.

\begin{theorem}
Let $s>2$ and $({\bf v^0},\boldsymbol{\gamma}) \in  \mathcal{H}^s(\mathbb{R}^2)^{4n} \times ]0,1[^{n-1}$ and $\boldsymbol{u_0}:=(u_0,v_0) \in \mathbb{R}^2$ such that
\begin{equation}
\left\{
\begin{array}{lr}
\inf_{X \in \mathbb{R}^2} h_i^0(X) > 0,& \forall\ i \in [\![1,n]\!],\\[4pt]
\inf_{X \in \mathbb{R}^2} \delta^a({\bf v^0}(X),\boldsymbol{\gamma},\boldsymbol{u_0})>0,
\end{array}
\right.
\label{1condwellposeda}
\end{equation}
where for every ${\bf v} \in \mathbb{R}^{4n}$,
\begin{equation}
\begin{array}{ll}
\delta^{\mathsf{a}}({\bf v},\boldsymbol{\gamma},\boldsymbol{u_0}):=& \min \big(a(\mathbf{h},\boldsymbol{\gamma})^{-1},\min_{i \in [\![1,n]\!]}(h_i^2)\big)\\[4pt]
& -\max_{i \in [\![1,n]\!]} \alpha_{n,i}|u_i - u_0|- \max_{i \in [\![1,n]\!]} \alpha_{n,i}|v_i - v_0|\\[4pt]
 &+\min \bigg(0,\min_{i \in [\![1,n]\!]} \frac{w_i+f}{2}\big(w_i+f - \sqrt{(w_i+f)^2+4 h_i^2}  \big) \bigg).
\end{array}
\label{1defdeltaa}
\end{equation}
Then, the {\em Cauchy} problem, associated with the system {\em \eqref{1systemmultilayerrelax}} and the initial data ${\bf v^0}$, is hyperbolic, locally well-posed in $\mathcal{H}^s(\mathbb{R}^2)^{4n}$ and there exists $T>0$ such that ${\bf v}$, the unique solution of the {\em Cauchy} problem, verifies
\begin{equation}
\left\{
\begin{array}{l}
{\bf v} \in \mathcal{C}^1([0,T] \times \mathbb{R}^2)^{4n},\\
{\bf v} \in \mathcal{C}([0,T];\mathcal{H}^s(\mathbb{R}^2))^{4n} \cap \mathcal{C}^1([0,T];\mathcal{H}^{s-1}(\mathbb{R}^2))^{4n}.
\end{array}
\right.
\label{1strongsolrelax}
\end{equation}
\label{1condwellposedMLSWrelax}
\end{theorem}
\begin{proof}
We define the next $4n \times 4n$ symmetric matrix:
\begin{equation}
\mathsf{S}^{\mathsf{a}}(\mathbf{v},\boldsymbol{\gamma}) =  \left[\begin{array}{c|c|c|c}
 \begin{array}{c} \mathsf{\Delta} \mathsf{\Gamma}+\mathsf{W}^2  \end{array}
 & \begin{array}{c} \mathsf{\Delta}  \mathsf{V}_x \end{array}
 & \begin{array}{c} \mathsf{\Delta}  \mathsf{V}_y \end{array}
 & \begin{array}{c} -\mathsf{W}  \mathsf{H}  \end{array}\\
 \hline
 \begin{array}{c} \mathsf{\Delta} \mathsf{V}_x \end{array}
 & \begin{array}{c} \mathsf{\Delta} \mathsf{H} \end{array}
 & \begin{array}{c} \mathsf{0} \end{array}
 & \begin{array}{c} \mathsf{0} \end{array}\\
 \hline
 \begin{array}{c}  \mathsf{\Delta}  \mathsf{V}_y \end{array}
 & \begin{array}{c} \mathsf{0} \end{array}
 & \begin{array}{c} \mathsf{\Delta} \mathsf{H} \end{array}
 & \begin{array}{c} \mathsf{0} \end{array}\\
  \hline
 \begin{array}{c}  -\mathsf{W}  \mathsf{H} \end{array}
 & \begin{array}{c} \mathsf{0} \end{array}
 & \begin{array}{c} \mathsf{0}\end{array}
 & \begin{array}{c} \mathsf{H}^2 \end{array}\\
\end{array}\right],
\label{1Sarelax}
\end{equation}

\noindent
One can check that $\mathsf{S}^a(\mathbf{v},\boldsymbol{\gamma})$, $\mathsf{S}^a(\mathbf{v},\boldsymbol{\gamma}) \mathsf{A}_x^a(\mathbf{v},\boldsymbol{\gamma})$ and $\mathsf{S}^a(\mathbf{v},\boldsymbol{\gamma}) \mathsf{A}_y^a(\mathbf{v},\boldsymbol{\gamma})$ are unconditionally symmetric:
\begin{equation}
\mathsf{S}^{\mathsf{a}}(\mathbf{v},\boldsymbol{\gamma}) \mathsf{A}_x^{\mathsf{a}} =  \left[\begin{array}{c|c|c|c}
 \begin{array}{c} 2\mathsf{\Delta} \mathsf{\Gamma} \mathsf{V}_x+\mathsf{W}^2 \mathsf{V}_x  \end{array}
 & \begin{array}{c} \mathsf{\Delta} \mathsf{\Gamma} \mathsf{H} + \mathsf{\Delta}  \mathsf{V}_x^2 \end{array}
 & \begin{array}{c} \mathsf{\Delta} \mathsf{V}_x \mathsf{V}_y \end{array}
 & \begin{array}{c} -\mathsf{W}  \mathsf{H} \mathsf{V}_x  \end{array}\\
 \hline
 \begin{array}{c} \mathsf{\Delta} \mathsf{\Gamma} \mathsf{H} + \mathsf{\Delta}  \mathsf{V}_x^2 \end{array}
 & \begin{array}{c} 2 \mathsf{\Delta} \mathsf{H} \mathsf{V}_x \end{array}
 & \begin{array}{c} \mathsf{\Delta} \mathsf{H} \mathsf{V}_y \end{array}
 & \begin{array}{c} \mathsf{0} \end{array}\\
 \hline
 \begin{array}{c}  \mathsf{\Delta} \mathsf{V}_x \mathsf{V}_y \end{array}
 & \begin{array}{c} \mathsf{\Delta} \mathsf{H} \mathsf{V}_y \end{array}
 & \begin{array}{c} \mathsf{0} \end{array}
 & \begin{array}{c} \mathsf{0} \end{array}\\
  \hline
 \begin{array}{c}  -\mathsf{W}  \mathsf{H} \mathsf{V}_x \end{array}
 & \begin{array}{c} \mathsf{0} \end{array}
 & \begin{array}{c} \mathsf{0}\end{array}
 & \begin{array}{c} \mathsf{H}^2 \mathsf{V}_x \end{array}\\
\end{array}\right],
\label{1SaAxarelax}
\end{equation}

\begin{equation}
\mathsf{S}^{\mathsf{a}}(\mathbf{v},\boldsymbol{\gamma}) \mathsf{A}_y^{\mathsf{a}} =  \left[\begin{array}{c|c|c|c}
 \begin{array}{c} 2\mathsf{\Delta} \mathsf{\Gamma} \mathsf{V}_y+\mathsf{W}^2 \mathsf{V}_y  \end{array}
 & \begin{array}{c} \mathsf{\Delta} \mathsf{V}_x \mathsf{V}_y \end{array}
 & \begin{array}{c} \mathsf{\Delta} \mathsf{\Gamma} \mathsf{H} + \mathsf{\Delta}  \mathsf{V}_y^2   \end{array}
 & \begin{array}{c} -\mathsf{W}  \mathsf{H} \mathsf{V}_y  \end{array}\\
 \hline
 \begin{array}{c} \mathsf{\Delta} \mathsf{V}_x \mathsf{V}_y \end{array}
 & \begin{array}{c} \mathsf{0} \end{array}
 & \begin{array}{c} \mathsf{\Delta} \mathsf{H} \mathsf{V}_x  \end{array}
 & \begin{array}{c} \mathsf{0} \end{array}\\
 \hline
 \begin{array}{c}  \mathsf{\Delta} \mathsf{\Gamma} \mathsf{H} + \mathsf{\Delta}  \mathsf{V}_y^2  \end{array}
 & \begin{array}{c} \mathsf{\Delta} \mathsf{H} \mathsf{V}_x  \end{array}
 & \begin{array}{c} 2 \mathsf{\Delta} \mathsf{H} \mathsf{V}_y \end{array}
 & \begin{array}{c} \mathsf{0} \end{array}\\
  \hline
 \begin{array}{c}  -\mathsf{W}  \mathsf{H} \mathsf{V}_y \end{array}
 & \begin{array}{c} \mathsf{0} \end{array}
 & \begin{array}{c} \mathsf{0}\end{array}
 & \begin{array}{c} \mathsf{H}^2 \mathsf{V}_y \end{array}\\
\end{array}\right],
\label{1SaAyarelax}
\end{equation}

\noindent
Then, we need to verify $\mathsf{S}^{\mathsf{a}}(\mathbf{v},\boldsymbol{\gamma})>0$ ({\em i.e} $\lambda^{\min}\big( \mathsf{S}^{\mathsf{a}}(\mathbf{v},\boldsymbol{\gamma}) \big) > 0$), to insure that it is a {\em Friedrichs}-symmetrizer. We introduce the following decomposition of $\mathsf{S}^{\mathsf{a}}(\mathbf{v},\boldsymbol{\gamma})$:
\begin{equation}
\mathsf{S}^{\mathsf{a}}(\mathbf{v},\boldsymbol{\gamma})=\mathsf{S}^{\mathsf{a}}_0(\mathbf{h},\boldsymbol{\gamma})+\mathsf{S}^{\mathsf{a}}_1(\mathbf{v},\boldsymbol{\gamma})+\mathsf{S}^{\mathsf{a}}_2(\mathbf{v},\boldsymbol{\gamma})+\mathsf{S}^{\mathsf{a}}_3(\mathbf{v},\boldsymbol{\gamma}),
\end{equation}
where the $4n \times 4n$ symmetric matrices are defined by
\begin{equation}
\mathsf{S}^{\mathsf{a}}_0(\mathbf{h},\boldsymbol{\gamma}) =  \left[\begin{array}{c|c|c|c}
 \begin{array}{c} \mathsf{\Delta} \mathsf{\Gamma}  \end{array}
 & \begin{array}{c} \mathsf{0} \end{array}
 & \begin{array}{c} \mathsf{0} \end{array}
 & \begin{array}{c} \mathsf{0}  \end{array}\\
 \hline
 \begin{array}{c} \mathsf{0} \end{array}
 & \begin{array}{c} \mathsf{\Delta} \mathsf{H} \end{array}
 & \begin{array}{c} \mathsf{0} \end{array}
 & \begin{array}{c} \mathsf{0} \end{array}\\
 \hline
 \begin{array}{c}  \mathsf{0} \end{array}
 & \begin{array}{c} \mathsf{0} \end{array}
 & \begin{array}{c} \mathsf{\Delta} \mathsf{H} \end{array}
 & \begin{array}{c} \mathsf{0} \end{array}\\
  \hline
 \begin{array}{c}  \mathsf{0}   \end{array}
 & \begin{array}{c} \mathsf{0} \end{array}
 & \begin{array}{c} \mathsf{0}\end{array}
 & \begin{array}{c} \mathsf{H}^2 \end{array}\\
\end{array}\right],\label{1S0a}
\end{equation}

\begin{equation}
\mathsf{S}^{\mathsf{a}}_1(\mathbf{v},\boldsymbol{\gamma}) =  \left[\begin{array}{c|c|c|c}
 \begin{array}{c} \mathsf{0}  \end{array}
 & \begin{array}{c} \mathsf{\Delta}  \mathsf{V}_x \end{array}
 & \begin{array}{c} \mathsf{0} \end{array}
 & \begin{array}{c} \mathsf{0}  \end{array}\\
 \hline
 \begin{array}{c} \mathsf{\Delta} \mathsf{V}_x \end{array}
 & \begin{array}{c} \mathsf{0} \end{array}
 & \begin{array}{c} \mathsf{0} \end{array}
 & \begin{array}{c} \mathsf{0} \end{array}\\
 \hline
 \begin{array}{c}  \mathsf{0} \end{array}
 & \begin{array}{c} \mathsf{0} \end{array}
 & \begin{array}{c} \mathsf{0} \end{array}
 & \begin{array}{c} \mathsf{0} \end{array}\\
  \hline
 \begin{array}{c}  \mathsf{0}  \end{array}
 & \begin{array}{c} \mathsf{0} \end{array}
 & \begin{array}{c} \mathsf{0}\end{array}
 & \begin{array}{c} \mathsf{0} \end{array}\\
\end{array}\right],\label{1S1a}
\end{equation}

\begin{equation}
\mathsf{S}^{\mathsf{a}}_2(\mathbf{v},\boldsymbol{\gamma}) =  \left[\begin{array}{c|c|c|c}
 \begin{array}{c} \mathsf{0}  \end{array}
 & \begin{array}{c} \mathsf{0} \end{array}
 & \begin{array}{c} \mathsf{\Delta}  \mathsf{V}_y \end{array}
 & \begin{array}{c} \mathsf{0}  \end{array}\\
 \hline
 \begin{array}{c} \mathsf{0} \end{array}
 & \begin{array}{c} \mathsf{0} \end{array}
 & \begin{array}{c} \mathsf{0} \end{array}
 & \begin{array}{c} \mathsf{0} \end{array}\\
 \hline
 \begin{array}{c}  \mathsf{\Delta}  \mathsf{V}_y \end{array}
 & \begin{array}{c} \mathsf{0} \end{array}
 & \begin{array}{c} \mathsf{0} \end{array}
 & \begin{array}{c} \mathsf{0} \end{array}\\
  \hline
 \begin{array}{c}  \mathsf{0} \end{array}
 & \begin{array}{c} \mathsf{0} \end{array}
 & \begin{array}{c} \mathsf{0}\end{array}
 & \begin{array}{c} \mathsf{0} \end{array}\\
\end{array}\right],\label{1S2a}
\end{equation}

\begin{equation}
\mathsf{S}^{\mathsf{a}}_3(\mathbf{v},\boldsymbol{\gamma}) =  \left[\begin{array}{c|c|c|c}
 \begin{array}{c} \mathsf{W}^2  \end{array}
 & \begin{array}{c} \mathsf{0} \end{array}
 & \begin{array}{c} \mathsf{0} \end{array}
 & \begin{array}{c} -\mathsf{W}  \mathsf{H}  \end{array}\\
 \hline
 \begin{array}{c} \mathsf{0} \end{array}
 & \begin{array}{c} \mathsf{0} \end{array}
 & \begin{array}{c} \mathsf{0} \end{array}
 & \begin{array}{c} \mathsf{0} \end{array}\\
 \hline
 \begin{array}{c}  \mathsf{0} \end{array}
 & \begin{array}{c} \mathsf{0} \end{array}
 & \begin{array}{c} \mathsf{0} \end{array}
 & \begin{array}{c} \mathsf{0} \end{array}\\
  \hline
 \begin{array}{c}  -\mathsf{W}  \mathsf{H} \end{array}
 & \begin{array}{c} \mathsf{0} \end{array}
 & \begin{array}{c} \mathsf{0}\end{array}
 & \begin{array}{c} \mathsf{0} \end{array}\\
\end{array}\right].
\label{1S3a}
\end{equation}

\noindent
According to the inequality of convexity \eqref{1ConcConv},
\begin{equation}
\lambda^{\min}\big(\mathsf{S}^{\mathsf{a}}(\mathbf{v},\boldsymbol{\gamma}) \big) \ge \lambda^{\min}\big(\mathsf{S}^{\mathsf{a}}_0(\mathbf{h},\boldsymbol{\gamma}) \big) + \lambda^{\min}\big(\mathsf{S}^{\mathsf{a}}_1(\mathbf{v},\boldsymbol{\gamma}) \big) + \lambda^{\min}\big(\mathsf{S}^{\mathsf{a}}_2(\mathbf{v},\boldsymbol{\gamma}) \big)+\lambda^{\min}\big(\mathsf{S}^{\mathsf{a}}_2(\mathbf{v},\boldsymbol{\gamma}) \big).
\end{equation}

\noindent
An analysis of each spectrum leads to
\begin{equation}
\left\{
\begin{array}{l}
\lambda^{\min}\big(\mathsf{S}^{\mathsf{a}}_0(\mathbf{h},\boldsymbol{\gamma}) \big) \ge \min \big(a(\mathbf{h},\boldsymbol{\gamma})^{-1},\min_{i \in [\![1,n]\!]}(h_i^2)\big),\\[4pt]
\lambda^{\min}\big(\mathsf{S}^{\mathsf{a}}_1(\mathbf{v},\boldsymbol{\gamma}) \big) = - \max_{i \in [\![1,n]\!]} \alpha_{n,i}|u_i|,\\[4pt]
\lambda^{\min}\big(\mathsf{S}^{\mathsf{a}}_1(\mathbf{v},\boldsymbol{\gamma}) \big) = - \max_{i \in [\![1,n]\!]} \alpha_{n,i}|v_i|,\\[4pt]
\lambda^{\min}\big(\mathsf{S}^{\mathsf{a}}_1(\mathbf{v},\boldsymbol{\gamma}) \big) = \min \bigg(0,\min_{i \in [\![1,n]\!]} \frac{w_i+f}{2}\big(w_i+f - \sqrt{(w_i+f)^2+4 h_i^2}  \big)\bigg).
\end{array}
\right.
\end{equation}

\noindent
Finally, with the rescaling
\begin{equation}
\forall i \in [\![1,n-1]\!],\ 
\left\{
\begin{array}{l}
h_i \leftarrow h_i,\\[4pt]
u_i -u_0 \leftarrow u_i,\\[4pt]
v_i-v_0 \leftarrow v_i,
\end{array}
\right.
\label{3rescal}
\end{equation}
and under conditions \eqref{1condwellposeda}, the mapping
\begin{equation}
\mathsf{S^a} : (\mathbf{v},\boldsymbol{\gamma}) \mapsto \mathsf{S^a}(\mathbf{v},\boldsymbol{\gamma})
\end{equation}
is a {\em Friedrichs}-symmetrizer. Using the propositions \ref{1propsymm} and \ref{1symimpliqhyp}, the {\em Cauchy} problem is hyperbolic, locally well-posed and the unique solution verifies \eqref{1strongsolrelax}, if the initial data verifies \eqref{1condwellposeda}.
\end{proof}

\noindent
{\em Remark:} The non-augmented model \eqref{1systemmultilayer} has a {\em symbolic}-symmetrizer if
\begin{equation}
a(\boldsymbol{h},\boldsymbol{\gamma})^{-1}- \alpha_{n,i}\big(|u_i-\bar{u}|+|v_i-\bar{v}|\big)>0,
\end{equation}
which is stronger than the condition of symmetrizability \eqref{1condwellposeda}, for the augmented model \eqref{1systemmultilayerrelax}, if  for all $i \in [\![1,n]\!]$,
\begin{equation}
h_i^2 \ge a(\boldsymbol{h},\boldsymbol{\gamma})^{-1},
\end{equation}
and
\begin{equation}
\min \bigg(0,\min_{i \in [\![1,n]\!]} \frac{w_i+f}{2}\big(w_i+f - \sqrt{(w_i+f)^2+4 h_i^2}  \big) \bigg) > \max_{i \in [\![1,n]\!]}\alpha_{n,i}(|u_0|-|\bar{u}|+|v_0|-|\bar{v}|) ,
\end{equation}
and weaker otherwise.

\subsection{A weaker criterion of local well-posedness}
As it was reminded before, the description of the eigenstructure of $\mathsf{A}^{\mathsf{a}}({\bf v},\boldsymbol{\gamma},\theta)$ is a decisive point, as it permits to characterize exactly its diagonalizability, the nature of the waves and also the {\em Riemann} invariants. According to the rotational invariance \eqref{1rotinvrelax}, we restrict the analysis to the eigenstructure of $\mathsf{A}^{\mathsf{a}}_x({\bf v},\boldsymbol{\gamma})$. First of all, as the characteristic polynomial of $\mathsf{A}^{\mathsf{a}}_x({\bf v},\boldsymbol{\gamma})$ is equal to
\begin{equation}
\det(\mathsf{A}^{\mathsf{a}}_x({\bf v},\boldsymbol{\gamma})-\lambda \mathsf{I_{4n}})=\lambda^n \det(\mathsf{A}_x({\bf u},\boldsymbol{\gamma})-\lambda \mathsf{I_{3n}}),
\label{1polcarAr}
\end{equation}
we remark that the spectrum of $\mathsf{A}^{\mathsf{a}}_x({\bf v},\boldsymbol{\gamma})$ is such that
\begin{equation}
\sigma(\mathsf{A}^{\mathsf{a}}_x({\bf v},\boldsymbol{\gamma})):=\left(\lambda_i^{\pm}({\bf v},\boldsymbol{\gamma})\right)_{i \in [\![1,n]\!]} \cup \left( \lambda_{2n+i}({\bf v},\boldsymbol{\gamma})\right)_{i \in [\![1,2n]\!]},
\label{1spectrumAr}
\end{equation}
where $\left(\lambda_i^{\pm}({\bf v},\boldsymbol{\gamma})\right)_{i \in [\![1,n]\!]} \cup \left( \lambda_{2n+i}({\bf v},\boldsymbol{\gamma})\right)_{i \in [\![1,n]\!]}=:\sigma(\mathsf{A}_x({\bf u},\boldsymbol{\gamma}))$ and 
\begin{equation}
\forall i \in [\![1,n]\!],\
\lambda_{3n+i}({\bf v},\boldsymbol{\gamma})=0.\\
\label{1lambda3n2n}
\end{equation}

\noindent
Let $\boldsymbol{\gamma} \in ]0,1[^{n-1}$, $\epsilon >0$ and an injective function $\sigma: [\![1,n-1]\!] \rightarrow \mathbb{R}_+^*$ such that $\boldsymbol{\gamma}$ verifies \eqref{1hypasympt}. We define the next subset of $\mathcal{L}^2(\mathbb{R}^2)^{4n}$:
\begin{equation}
\mathcal{H}^{a}_{\sigma,\epsilon}:= \left\{ {\bf v^0} \in \mathcal{L}^2(\mathbb{R}^2)^{4n} /\ {\bf u^0} \in \mathcal{H}_{\sigma,\epsilon} \right\}
\label{1defHgammasigmaepsilonrelax}
\end{equation}

\begin{proposition}
Let $({\bf v},\boldsymbol{\gamma}) \in \mathbb{R}^{4n} \times ]0,1[^{n-1}$, $\theta \in [0,2\pi]$, $\epsilon >0$ and an injective function $\sigma \in \mathbb{R}_+^{*\ [\![1,n-1]\!]}$ such that $\boldsymbol{\gamma}$ verifies \eqref{1hypasympt} and the associated vector ${\bf u}$ verifies \eqref{1hypasympt2}.

\noindent
There exists $\delta > 0$ such that if
\begin{equation}
\epsilon \le \delta,
\label{1defdelta}
\end{equation}
then the matrix $\mathsf{A}^{\mathsf{a}}({\bf v},\boldsymbol{\gamma},\theta)$ is diagonalizable with real eigenvalues if the associated vector, ${\bf u}$, verifies
\begin{equation}
\left\{
\begin{array}{lr}
h_i > 0, & \forall i \in [\![1,n]\!],\\[6pt]
\phi_{\sigma,i}(\boldsymbol{h})-|u_{i+1}-u_i|^2-|v_{i+1}-v_i|^2 > 0,& \forall i \in [\![1,n-1]\!].
\end{array}\right.
\label{1condnechyp5}
\end{equation}
\label{1diagArelax}
\end{proposition}

\begin{proof}
With the rotational invariance \eqref{1rotinvrelax}, it is equivalent to prove the diagonalizability of $\mathsf{A}^{\mathsf{a}}_x({\bf v},\boldsymbol{\gamma})$. By denoting $(\boldsymbol{e'_i})_{i \in [\![1,4n]\!]}$ the canonical basis of $\mathbb{R}^{4n}$, one can prove the expressions of the right eigenvectors $\boldsymbol{r}_x^{\lambda}({\bf v},\boldsymbol{\gamma})$ of $\mathsf{A}^{\mathsf{a}}_x({\bf v},\boldsymbol{\gamma})$, associated to the eigenvalue $\lambda \in \sigma(\mathsf{A}^{\mathsf{a}}_x({\bf v},\boldsymbol{\gamma}))$, are, for all $i \in [\![1,n]\!]$, defined by
\begin{equation}
\boldsymbol{r}_x^{\lambda_i^{\pm}}({\bf v},\boldsymbol{\gamma})=\boldsymbol{r}_x^{\lambda_i^{\pm}}({\bf u},\boldsymbol{\gamma}) + \sum_{k=1}^n \frac{w_k +f}{\lambda_i^{\pm} - u_k}(\boldsymbol{r}_x^{\lambda_i^{\pm}}({\bf u},\boldsymbol{\gamma}) \cdot \boldsymbol{e_{n+k}}) \boldsymbol{e'_{3n+k}},
\label{1eigvectxrelax1}
\end{equation}

\begin{equation}
\boldsymbol{r}_x^{\lambda_{2n+i}}({\bf v},\boldsymbol{\gamma})=\boldsymbol{e'_{3n+i}},
\label{1eigvectxrelax3}
\end{equation}

\begin{equation}
\boldsymbol{r}_x^{\lambda_{3n+i}}({\bf v},\boldsymbol{\gamma})=
\left\{
\begin{array}{lr}
\begin{array}{l}
(\mathsf{\Gamma} \mathsf{V}_y \boldsymbol{e_i} \cdot \boldsymbol{e_i}) \boldsymbol{e'_i} - (\mathsf{V}_x \mathsf{H}^{-1} \mathsf{\Gamma}^{-1} \mathsf{V}_y \boldsymbol{e_i} \cdot \boldsymbol{e_i}) \boldsymbol{e'_{n+i}}\\
 -((\mathsf{I_n} -  \mathsf{V}_x^2 \mathsf{H}^{-1} \mathsf{\Gamma}^{-1}) \mathsf{V}_y \boldsymbol{e_i} \cdot \boldsymbol{e_i}) \boldsymbol{e'_{2n+i}}\\+ (\mathsf{W} \mathsf{H}^{-1} \mathsf{\Gamma}^{-1} \mathsf{V}_y \boldsymbol{e_i} \cdot \boldsymbol{e_i}) \boldsymbol{e'_{3n+i}}\end{array},& \mathrm{if}\ v_i \not=0,\\[20pt]
\boldsymbol{e'_{2n+i}},& \mathrm{if}\ v_i =0,
\end{array}
\right.
\label{1eigvectxrelax2}
\end{equation}


\noindent
where $\boldsymbol{r}_x^{\lambda}({\bf u},\boldsymbol{\gamma})$ are expressed in \eqref{1eigvectlambdnpmr2} and \eqref{1eigvectlambdipm21} and $( \cdot )$ is the inner product on $\mathbb{R}^{3n}$.

\noindent
Consequently, the right eigenvectors $\boldsymbol{r}^{\lambda}({\bf v},\boldsymbol{\gamma},\theta)$ of $\mathsf{A}^{\mathsf{a}}({\bf v},\boldsymbol{\gamma},\theta)$ are defined by
\begin{equation}
\forall \lambda \in \sigma\left(\mathsf{A}^{\mathsf{a}}({\bf v},\boldsymbol{\gamma},\theta) \right),\ \boldsymbol{r}^{\lambda}({\bf v},\boldsymbol{\gamma},\theta)=\mathsf{P^{\mathsf{a}}}(\theta)^{-1}\boldsymbol{r}_x^{\lambda}(\mathsf{P^{\mathsf{a}}}(\theta){\bf v},\boldsymbol{\gamma}).
\label{2eigvectrelax}
\end{equation}

\noindent
Moreover, as it was proved in the diagonalizability of $\mathsf{A}_x({\bf u},\boldsymbol{\gamma})$ in proposition \ref{1diagA}, if $\epsilon \le \delta$, the right eigenvectors induced an eigenbasis of $\mathbb{R}^{4n}$. A consequence is the right eigenvectors form an eigenbasis of $\mathbb{R}^{4n}$ and $\mathsf{A}^{\mathsf{a}}_x({\bf v},\boldsymbol{\gamma})$  is diagonalizable with real eigenvalues, if \eqref{1condnechyp5} is verified.
\end{proof}

\noindent
{\em Remark:} There is also the left eigenvectors $\boldsymbol{l}_x^{\lambda}({\bf v},\boldsymbol{\gamma})$ of $\mathsf{A}^{\mathsf{a}}_x({\bf v},\boldsymbol{\gamma})$, associated to the eigenvalue $\lambda \in \sigma(\mathsf{A}^{\mathsf{a}}_x({\bf v},\boldsymbol{\gamma}))$: for all $i \in [\![1,n]\!]$,
\begin{equation}
\boldsymbol{l}_x^{\lambda_i^{\pm}}({\bf v},\boldsymbol{\gamma})=\boldsymbol{l}_x^{\lambda_i^{\pm}}({\bf u},\boldsymbol{\gamma}) + \sum_{k=1}^n \frac{v_k}{\lambda_i^{\pm}}({}^{\top}\boldsymbol{l}_x^{\lambda_i^{\pm}}({\bf u},\boldsymbol{\gamma}) \cdot \boldsymbol{e_{n+k}}) {}^{\top} \boldsymbol{e'_{2n+k}},
\label{1leigvectxrelax1}
\end{equation}

\begin{equation}
\boldsymbol{l}_x^{\lambda_{2n+i}}({\bf v},\boldsymbol{\gamma})=-(w_i+f){}^{\top} \boldsymbol{e'_{i}} +h_i{}^{\top} \boldsymbol{e'_{3n+i}},
\label{1leigvectxrelax3}
\end{equation}

\begin{equation}
\boldsymbol{l}_x^{\lambda_{3n+i}}({\bf v},\boldsymbol{\gamma})={}^{\top} \boldsymbol{e'_{2n+i}},
\label{1leigvectxrelax2}
\end{equation}

\noindent
where $\boldsymbol{l}_x^{\lambda}({\bf u},\boldsymbol{\gamma})$ are expressed in \eqref{1expressionlefteigenvectn} and \eqref{1eigvectlambdipm22} and $( \cdot )$ is the inner product on $\mathbb{R}^{3n}$. Moreover, we made intentionally a mistake in \eqref{1eigvectxrelax1} and \eqref{1leigvectxrelax1}, as we did not provide the expression of $\boldsymbol{r}_x^{\lambda_i^{\pm}}({\bf u},\boldsymbol{\gamma})$ and $\boldsymbol{l}_x^{\lambda_i^{\pm}}({\bf u},\boldsymbol{\gamma})$, but it is the natural expression coming from \eqref{1eigvectlambdnpmr2}, \eqref{1expressionlefteigenvectn}, \eqref{1eigvectlambdipm21} and \eqref{1eigvectlambdipm22} and replacing $\boldsymbol{e_{i}}$ by $\boldsymbol{e'_{i}}$, for every $i \in [\![1,3n]\!]$.

\noindent
Then, the left eigenvectors $\boldsymbol{l}^{\lambda}({\bf v},\boldsymbol{\gamma},\theta)$ of $\mathsf{A}^{\mathsf{a}}({\bf v},\boldsymbol{\gamma},\theta)$ are also defined by
\begin{equation}
\forall \lambda \in \sigma\left(\mathsf{A}^{\mathsf{a}}({\bf v},\boldsymbol{\gamma},\theta) \right),\ \boldsymbol{l}^{\lambda}({\bf v},\boldsymbol{\gamma},\theta)=\boldsymbol{l}_x^{\lambda}(\mathsf{P^{\mathsf{a}}}(\theta){\bf v},\boldsymbol{\gamma})\mathsf{P^{\mathsf{a}}}(\theta).
\label{1eigvectleftrelax}
\end{equation}

\noindent
{\em Remark:} According to the asymptotic expansions \eqref{1expressionlefteigenvectn} and \eqref{1eigvectlambdipm22}, in the general case, $\lambda_i^{\pm} \not=0$ and $\lambda_i^{\pm}\not=u_i$; consequently, $\boldsymbol{r}_x^{\lambda_i^{\pm}}({\bf v},\boldsymbol{\gamma})$ in \eqref{1eigvectxrelax1} and $\boldsymbol{l}_x^{\lambda_i^{\pm}}({\bf v},\boldsymbol{\gamma})$ in \eqref{1leigvectxrelax1} are defined.

\noindent
Furthermore, the type of the wave associated to each eigenvalue is described in the next proposition.
\begin{proposition}
Let $\boldsymbol{\gamma} \in ]0,1[^{n-1}$, $\epsilon >0$ and an injective function $\sigma \in \mathbb{R}_+^{*\ [\![1,n-1]\!]}$ such that $\boldsymbol{\gamma}$ verifies \eqref{1hypasympt}.

\noindent
Then, there exists $\delta > 0$ such that if
\begin{equation}
\epsilon \le \delta,
\label{1assumpteigenvectorlambdan5}
\end{equation}
and ${\bf v} \in \mathcal{H}_{\sigma,\epsilon}^a$, we have
\begin{equation}
\left\{
\begin{array}{lr}
\mathrm{the}\ \lambda_i^{\pm}\mathrm{-characteristic\ field\ is\ genuinely\ non-linear},&\mathrm{if}\ i \in [\![1,n]\!],\\
\mathrm{the}\ \lambda_i\mathrm{-characteristic\ field\ is\ linearly\ degenerate},&\mathrm{if}\ i \in [\![2n+1,4n]\!].
\end{array}
\right.
\label{1muwavenaturerelax}
\end{equation}
\label{1genuinelynonlinrelax}
\end{proposition}
\begin{proof}
Using the same proof of proposition \ref{1genuinelynonlin}, and remarking that for all $i \in [\![2n+1,3n]\!],\ \lambda_i=0$, which implies
\begin{equation}
\nabla \lambda_i \cdot \boldsymbol{r}_x^{\lambda}({\bf v},\boldsymbol{\gamma}) =0,
\label{1genuidefrelax}
\end{equation}
and the proof of the proposition \ref{1genuinelynonlinrelax}.
\end{proof}

\noindent
To conclude, under conditions of the proposition \ref{1genuinelynonlinrelax}, for all $i \in [\![1,n]\!]$, the $\lambda_i^{\pm}$-wave is a shock wave or a rarefaction wave and the $\lambda_{2n+i}^{\pm}$-wave and $\lambda_{3n+i}^{\pm}$-wave are contact waves.

\noindent
Finally, the point is to know if this augmented system \eqref{1systemmultilayerrelax} is locally well-posed and if its solution provides the solution of the non-augmented system \eqref{1systemmultilayer}.
\begin{theorem}
Let $s >2$, $\boldsymbol{\gamma} \in ]0,1[^{n-1}$, $\epsilon >0$ and an injective function $\sigma: [\![1,n-1]\!] \rightarrow \mathbb{R}_+^*$ such that $\boldsymbol{\gamma}$ verifies \eqref{1hypasympt}.

\noindent
Then, there exists $\delta > 0$ such that if
\begin{equation}
\epsilon \le \delta,
\label{1assumpttheoremwellposedrelax}
\end{equation}
the {\em Cauchy} problem, associated with {\em \eqref{1systemmultilayerrelax}} and initial data ${\bf v^0} \in \mathcal{H}_{\sigma,\epsilon}^a \cap \mathcal{H}^s(\mathbb{R}^2)^{4n}$, is hyperbolic, locally well-posed in $\mathcal{H}^s(\mathbb{R}^2)^{4n}$ and the unique solution verifies conditions \eqref{1strongsolrelax}. Furthermore, ${\bf u}$, the associated vector field, verifies conditions {\em \eqref{1strongsol}} and is the unique classical solution of the {\em Cauchy} problem, associated with {\em \eqref{1systemmultilayer}} and initial data ${\bf u^0} \in \mathcal{H}_{\sigma,\epsilon} \cap \mathcal{H}^s(\mathbb{R}^2)^{3n}$, if and only if
\begin{equation}
\forall i \in [\![1,n]\!],\ w_i^0 = \frac{\partial v_i^0}{\partial x}-\frac{\partial u_i^0}{\partial y}.
\label{1initialcompatible}
\end{equation}
\label{1thmimplicsystrelax}
\end{theorem}
\begin{proof}
Using proposition \ref{1diagArelax}, $\sigma(\mathsf{A}^{\mathsf{a}}({\bf v},\boldsymbol{\gamma},\theta)) \subset \mathbb{R}$ and $\mathsf{A}^{\mathsf{a}}({\bf v}\boldsymbol{\gamma},\theta)$ is diagonalizable. Then, the proposition \ref{1diagimpliqsym} is verified: the hyperbolicity and the local well-posedness of the {\em Cauchy} problem, associated with system \eqref{1systemmultilayerrelax} and initial data ${\bf v^0}$, is insured and conditions \eqref{1strongsolrelax} are verified. Moreover, it is obvious to prove that, for all $i \in [\![1,n]\!]$, there exists $\Phi_i : \mathbb{R}^2 \rightarrow \mathbb{R}$ such that
\begin{equation}
\forall (t,x,y) \in \mathbb{R}_+ \times \mathbb{R}^2,\ w_i(t,x,y)=\frac{\partial v_i}{\partial x}(t,x,y)-\frac{\partial u_i}{\partial y}(t,x,y)+\Phi_i(x,y).
\label{2phi}
\end{equation}
As $\Phi_i$ does not depend on the time $t$, ${\bf u}$ -- the vector associated to ${\bf v}$ -- is solution of the non-augmented system \eqref{1systemmultilayer} if and only if $\Phi_i=0$, for all $i \in [\![1,n]\!]$, which is true if and only if it is verified at $t=0$.
\end{proof}

\noindent
We deduce directly the next corollary.
\begin{corollary}
Let $s >2$, $\boldsymbol{\gamma} \in ]0,1[^{n-1}$, $\epsilon >0$ and an injective function $\sigma: [\![1,n-1]\!] \rightarrow \mathbb{R}_+^*$ such that $\boldsymbol{\gamma}$ verifies \eqref{1hypasympt}.

\noindent
There exists $\delta > 0$ such that if we assume
\begin{equation}
\epsilon \le \delta,
\label{1assumpttheoremwellposedrelax2}
\end{equation}
and we consider ${\bf v}$, the unique solution of the {\em Cauchy} problem, associated with {\em \eqref{1systemmultilayerrelax}} and initial data ${\bf v^0}  \in \mathcal{H}^{a}_{\sigma,\epsilon} \cap \mathcal{H}^s(\mathbb{R}^2)^{4n}$, then the associated vector field, ${u} \in \mathcal{H}_{\sigma,\epsilon} \cap \mathcal{H}^s(\mathbb{R}^2)^{3n}$ is the unique solution of \eqref{1systemmultilayer} and verifies \eqref{1strongsol} if and only if ${\bf v^0}$ verifies
\begin{equation}
\forall i \in [\![1,n]\!],\ w_i^0 = \frac{\partial v_i^0}{\partial x}-\frac{\partial u_i^0}{\partial y}.
\label{1compatibilitycondwi}
\end{equation}
\label{1corollarywellpposedrelax}
\end{corollary}

\begin{proof}
If $(\sigma,\epsilon)$ verify these assumptions, then, according to the theorem \ref{1thmimplicsystrelax}, the unique solution of the {\em Cauchy} problem, associated with \eqref{1systemmultilayerrelax} and initial data ${\bf v^0} \in \mathcal{H}_{\sigma,\epsilon}^a \cap \mathcal{H}^s(\mathbb{R}^2)^{4n}$ is such that the associated vector field, ${\bf u}$, verifies conditions \eqref{1strongsol} and is the unique classical solution of the Cauchy problem, associated with \eqref{1systemmultilayer} and initial data ${\bf u^0} \in \mathcal{H}_{\sigma,\epsilon} \cap \mathcal{H}^s(\mathbb{R}^2)^{3n}$, if and only if \eqref{1compatibilitycondwi} is verified.
\end{proof}

\noindent
To cut a long story short, we introduced a new conservative multi-layer model, in two-dimensions, proved the {\em Friedrichs}-symmetrizability under conditions \eqref{1condwellposeda}, proved its local well-posedness in $\mathcal{H}^s(\mathbb{R}^2)^{3n}$, with $s>2$, under the same conditions expressed in the previous section. Moreover, we explained the link between the solutions of the augmented and the non-augmented models: they are the same if they verify the compatibility conditions \eqref{1thmimplicsystrelax}, when $t=0$.

\section{Discussions and perspectives}
In this paper, we proved,  with various techniques, the hyperbolicity and the local well-posedness, in $\mathcal{H}^s(\mathbb{R}^2)^{3n}$, of the two-dimensional multi-layer shallow water model, with free surface. All of them use the rotational invariance property \eqref{1rotinv}, reducing the problem from two dimensions to one dimension. We gave, at first, a criterion of local well-posedness, in $\mathcal{H}^s(\mathbb{R}^2)^{3n}$, using the symmetrizability of the system \eqref{1systemmultilayer}. Afterwards, we studied the hyperbolicity of different particular cases: the single-layer model, the merger of two layers and the asymptotic expansion of this last case. Then, we proved the asymptotic expansion of all the eigenvalues, in a particular asymptotic regime, and a new criterion of hyperbolicity of this system was explicitly characterized and compared with the set of symmetrizability. This criterion is clearly similar with the criterion well-known in the two-layer case. Moreover, we provided the asymptotic expansion of all the eigenvectors, in this regime, we characterized the nature of waves associated to each element of the spectrum of $\mathsf{A}_x({\bf u},\boldsymbol{\gamma})$ -- shock, rarefaction of contact wave -- and we proved the local well-posedness, in $\mathcal{H}^s(\mathbb{R}^2)^{3n}$, of the system \eqref{1systemmultilayer}, under conditions of hyperbolicity and weak density-stratification. Finally, after discussing about the conservative quantities of the system, we introduced a new augmented model \eqref{1systemmultilayerrelax}, adding the horizontal vorticity, in each layer, as a new unknown. We also characterized the eigenstructure, the nature of the waves, proved the local well-posedness in $\mathcal{H}^s(\mathbb{R}^2)$ and explained the link of a solution of the non-augmented model \eqref{1systemmultilayer} and a solution of the new model \eqref{1systemmultilayerrelax}. The {\em conservativity} of the new augmented model avoid choosing a conservative path, introduced in \cite{dal1995definition}, to solve the numerical problem.

\noindent
However, the characterization of all the conservative quantities is still an open question, in the general case of $n$ layers and in one and two dimensions. Moreover, we addressed the question of the hyperbolicity and the local well-posedness in a particular asymptotic regime. There are a lot of other possibilities which are not taken into account in this regime. Indeed, even if the assumption on the density-stratification seems to embrace most of the useful cases of the oceanography, the assumptions on the heights of each layer does clearly not. Then, other asymptotic expansions are needed to be performed, in order to characterize the other possibilities.

\noindent
Finally, the characterization of the eigenstructure is a decisive point of the numerical treatment of the open boundary problem, in a limited domain $\Omega$. Indeed, there are a lot of techniques to treat these kind of boundary conditions: the radiation methods as the {\em Sommerfeld} conditions from \cite{sommerfeld1949partial} or as the {\em Orlanski-type} conditions, for more complex hyperbolic flows, proposed in \cite{orlanski1976simple}; the absorbing conditions, explained in \cite{engquist1977absorbing}; relaxation methods studied in \cite{roed1987study}; or the {\em Flather} conditions proposed in \cite{flather1976tidal}. As it was underlined in \cite{blayo2005revisiting}, the characteristic-based methods, such as {\em Flather} conditions --- which is often seen as radiation conditions ---, are natural and efficient open boundary conditions: the outgoing waves does not need any conditions, while conditions are imposed on the characteristic variables, for the incoming waves. However, the integration of the characteristics variables is not an issue when $n=1$, the single-layer problem: the characteristics variables are exactly known, because the exact eigenvectors $\boldsymbol{l}^{\lambda_1^{\pm}}:={}^{\top}(1,\pm\sqrt{\frac{h}{g}},0)$, associated with the exact eigenvalue $\lambda_1^{\pm}:=u \pm \sqrt{g h}$, is such that:

\begin{equation}
{}^{\top} \boldsymbol{l}^{\lambda_1^{\pm}} \frac{\partial {\bf u}}{\partial t}= \sqrt{\frac{h}{g}} \frac{\partial}{\partial t}\big( u \pm 2 \sqrt{gh} \big),
\end{equation}

\noindent
where we remind that ${\bf u}:={}^{\top}(h,u,v)$, when $n=1$. Then, the characteristic variables of the single-layer model, at the surface $\partial\Omega$, are: $(\boldsymbol{u} \cdot \boldsymbol{n})$ and $(\boldsymbol{u} \cdot \boldsymbol{n}) \pm 2 \sqrt{gh}$, where $\boldsymbol{n}$ is the outward pointing unit vector of $\partial\Omega$. In the case $n=2$, \cite{bousquet2013boundary} and \cite{petcu2013interface} gave the $4$ characteristics variables associated to the two-layer model and we will give it for an eastern surface: $\boldsymbol{n}={}^{\top}(1,0)$. Two of them are associated to the total height of water ({\em i.e.} the barotropic waves):
\begin{equation}
\frac{h_1 u_1 + h_2 u_2}{h_1+h_2} \pm 2\sqrt{g(h_1+h_2)},
\label{1charactvariabletwolayer1}
\end{equation}

\noindent
and two of them to the interface ({\em i.e.} the baroclinic waves):
\begin{equation}
\arcsin\left(\frac{h_1 - h_2}{h_1+h_2}\right) \mp \arcsin\left(\frac{u_2 - u_1}{\sqrt{g(1-\gamma_1)(h_1+h_2)}} \right).
\label{1charactvariabletwolayer2}
\end{equation}

\noindent
Nevertheless, these expressions are formal approximations of the regime $\gamma_1 \approx 1$ in \eqref{1charactvariabletwolayer1}, and $\frac{\partial (h_1 + h_2)}{\partial t} \approx 0$ in \eqref{1charactvariabletwolayer2}. As far as we know, the question is still open about the precision of these characteristic variables, compared with a linearized treatment of the open boundary conditions. Finally, as we have proved in this paper, the eigenstructure of the multi-layer model, with $n \ge 3$ and in the asymptotic regime considered, looks like different two-layer models, with different layers considered. Then, in the asymptotic regime \eqref{1hypasympt} and \eqref{1hypasympt2}, another open question is the efficiency of the open boundary conditions with the following formal characteristic variables:

\begin{equation}
\bar{u} \pm 2\sqrt{gH},
\label{1charactvariablenlayer1}
\end{equation}
and
\begin{equation}
\forall i \in [\![1,n-1]\!],\ \arcsin\left(\frac{h_{\sigma,i}^- - h_{\sigma,i}^+}{h_{\sigma,i}^- +h_{\sigma,i}^+}\right) \mp \arcsin\left(\frac{u_{i+1} - u_i}{\sqrt{g(1-\gamma_i)(h_{\sigma,i}^- +h_{\sigma,i}^+)}} \right).
\label{1charactvariablenlayer2}
\end{equation}
compared with a linearized treatment of the open boundary conditions, as {\em Flather} conditions. It would be interesting to compute these two kind of open boundary conditions in the two-layer case and a more general one, in a simple limited domain such as a rectangular, to address these open questions.

\section*{Acknowledgments}
The author warmly thanks P. Noble, J.P. Vila, V. Duch\^ene, R. Baraille and F. Chazel, for their noteworthy contribution to this research.

\bibliographystyle{plain}

\bibliography{bib}

\begin{thebibliography}{10}

\bibitem{abgrall2009two}
R.~Abgrall and S.~Karni.
\newblock Two-layer shallow water system: a relaxation approach.
\newblock {\em SIAM Journal on Scientific Computing}, 31(3):1603--1627, 2009.

\bibitem{alvarez2007nash}
B.~Alvarez-Samaniego and D.~Lannes.
\newblock A nash-moser theorem for singular evolution equations. application to
  the serre and green-naghdi equations.
\newblock {\em arXiv preprint math/0701681}, 2007.

\bibitem{audusse2005multilayer}
E.~Audusse.
\newblock A multilayer saint-venant model: derivation and numerical validation.
\newblock {\em Discrete Contin. Dyn. Syst. Ser. B}, 5(2):189--214, 2005.

\bibitem{audusse2011approximation}
E.~Audusse, M.-O. Bristeau, M.~Pelanti, and J.~Sainte-Marie.
\newblock Approximation of the hydrostatic navier--stokes system for density
  stratified flows by a multilayer model: kinetic interpretation and numerical
  solution.
\newblock {\em Journal of Computational Physics}, 230(9):3453--3478, 2011.

\bibitem{audusse2011multilayer}
E.~Audusse, M.O. Bristeau, B.~Perthame, and J.~Sainte-Marie.
\newblock A multilayer saint-venant system with mass exchanges for shallow
  water flows. derivation and numerical validation.
\newblock {\em ESAIM: Mathematical Modelling and Numerical Analysis},
  45(01):169--200, 2011.

\bibitem{barros2006conservation}
R.~Barros.
\newblock Conservation laws for one-dimensional shallow water models for one
  and two-layer flows.
\newblock {\em Mathematical Models and Methods in Applied Sciences},
  16(01):119--137, 2006.

\bibitem{barros2008hyperbolicity}
R.~Barros and W.~Choi.
\newblock On the hyperbolicity of two-layer flows.
\newblock Proceedings of the 2008 Conference on FACM08 held at New Jersey
  Institute of Technology, 2008.

\bibitem{benzoni2007multi}
S.~Benzoni-Gavage and D.~Serre.
\newblock {\em Multi-dimensional hyperbolic partial differential equations}.
\newblock Clarendon Press Oxford, 2007.

\bibitem{blayo2005revisiting}
E.~Blayo and L.~Debreu.
\newblock Revisiting open boundary conditions from the point of view of
  characteristic variables.
\newblock {\em Ocean modelling}, 9(3):231--252, 2005.

\bibitem{bouchut2008entropy}
F.~Bouchut and T.~Morales~de Luna.
\newblock An entropy satisfying scheme for two-layer shallow water equations
  with uncoupled treatment.
\newblock {\em ESAIM: Mathematical Modelling and Numerical Analysis},
  42(04):683--698, 2008.

\bibitem{bousquet2013boundary}
A.~Bousquet, M.~Petcu, M.-C. Shiue, R.~Temam, and J.~Tribbia.
\newblock Boundary conditions for limited area models based on the shallow
  water equations.
\newblock {\em Communications in Computational Physics}, 14:664--702, 2013.

\bibitem{browning1982initialization}
G.~Browning and H.-O. Kreiss.
\newblock Initialization of the shallow water equations with open boundaries by
  the bounded derivative method.
\newblock {\em Tellus}, 34(4):334--351, 1982.

\bibitem{Castrowellposed2014}
A.~Castro and D.~Lannes.
\newblock Well-posedness and shallow-water stability for a new hamiltonian
  formulation of the water waves equations with vorticity.
\newblock {\em ArXiv e-prints 1402.0464}, 2014.

\bibitem{castro2010hyperbolicity}
M.~Castro, J.~Frings, S.~Noelle, C.~Par{\'e}s, and G.~Puppo.
\newblock {\em On the hyperbolicity of two-and three-layer shallow water
  equations}.
\newblock Inst. f{\"u}r Geometrie und Praktische Mathematik, 2010.

\bibitem{castro2011numerical}
M.J. Castro-D{\'\i}az, E.D. Fern{\'a}ndez-Nieto, J.M. Gonz{\'a}lez-Vida, and
  C.~Par{\'e}s-Madro{\~n}al.
\newblock Numerical treatment of the loss of hyperbolicity of the two-layer
  shallow-water system.
\newblock {\em Journal of Scientific Computing}, 48(1-3):16--40, 2011.

\bibitem{chumakova2009stability}
L.~Chumakova, F.~Menzaque, P.~Milewski, R.~Rosales, E.~Tabak, and C.~Turner.
\newblock Stability properties and nonlinear mappings of two and three-layer
  stratified flows.
\newblock {\em Studies in Applied Mathematics}, 122(2):123--137, 2009.

\bibitem{dal1995definition}
G.~Dal~Maso, P.G. LeFloch, and F.~Murat.
\newblock Definition and weak stability of nonconservative products.
\newblock {\em Journal de math{\'e}matiques pures et appliqu{\'e}es},
  74(6):483--548, 1995.

\bibitem{desaint1871theorie}
A.B. de~Saint-Venant.
\newblock Th{\'e}orie du mouvement non permanent des eaux, avec application aux
  crues des rivi{\`e}res et a l'introduction de mar{\'e}es dans leurs lits.
\newblock {\em Comptes rendus des s{\'e}ances de l'Acad{\'e}mie des Sciences},
  36:174--154, 1871.

\bibitem{duchene2010asymptotic}
V.~Duch{\^e}ne.
\newblock Asymptotic shallow water models for internal waves in a two-fluid
  system with a free surface.
\newblock {\em SIAM Journal on Mathematical Analysis}, 42(5):2229--2260, 2010.

\bibitem{duchene2013note}
V.~Duch{\^e}ne.
\newblock A note on the well-posedness of the one-dimensional multilayer
  shallow water model.
\newblock {\em ArXiv e-prints}, 2013.

\bibitem{duchene2013rigid}
V.~Duch{\^e}ne.
\newblock On the rigid-lid approximation for two shallow layers of immiscible
  fluids with small density contrast.
\newblock {\em arXiv preprint arXiv:1309.3115}, 2013.

\bibitem{engquist1977absorbing}
B.~Engquist and A.~Majda.
\newblock Absorbing boundary conditions for numerical simulation of waves.
\newblock {\em Proceedings of the National Academy of Sciences},
  74(5):1765--1766, 1977.

\bibitem{bouchut2010robust}
V.~Zeitlin F.~Bouchut et~al.
\newblock A robust well-balanced scheme for multi-layer shallow water
  equations.
\newblock {\em Discrete and Continuous Dynamical Systems-Series B},
  13(4):739--758, 2010.

\bibitem{flather1976tidal}
R.A. Flather.
\newblock A tidal model of the northwest european continental shelf.
\newblock {\em Mem. Soc. R. Sci. Liege}, 10(6):141--164, 1976.

\bibitem{frings2012adaptive}
J.T. Frings.
\newblock {\em An adaptive multilayer model for density-layered shallow water
  flows}.
\newblock Universit{\"a}tsbibliothek, 2012.

\bibitem{gill1982atmosphere}
A.E. Gill.
\newblock Atmosphere-ocean dynamics. intenational geophysics series 30.
\newblock {\em Donn, Academic, Orlando, Fla}, 1982.

\bibitem{kim2008two}
J.~Kim and R.J. LeVeque.
\newblock Two-layer shallow water system and its applications.
\newblock In {\em Proceedings of the Twelth International Conference on
  Hyperbolic Problems, Maryland}, 2008.

\bibitem{liska1997analysis}
R.~Liska and B.~Wendroff.
\newblock Analysis and computation with stratified fluid models.
\newblock {\em Journal of Computational Physics}, 137(1):212--244, 1997.

\bibitem{long1956long}
R.R. Long.
\newblock Long waves in a two-fluid system.
\newblock {\em Journal of Meteorology}, 13(1):70--74, 1956.

\bibitem{monjarret2014local}
R.~Monjarret.
\newblock Local well-posedness of the two-layer shallow water model with free
  surface.
\newblock {\em arXiv preprint arXiv:1402.3194}, 2014.

\bibitem{orlanski1976simple}
I\_ Orlanski.
\newblock A simple boundary condition for unbounded hyperbolic flows.
\newblock {\em Journal of computational physics}, 21(3):251--269, 1976.

\bibitem{ovsyannikov1979two}
L.V. Ovsyannikov.
\newblock Two-layer shallow water model.
\newblock {\em Journal of Applied Mechanics and Technical Physics},
  20(2):127--135, 1979.

\bibitem{pedlosky1982geophysical}
J.~Pedlosky.
\newblock Geophysical fluid dynamics.
\newblock {\em New York and Berlin, Springer-Verlag, 1982. p. 636}, 1, 1982.

\bibitem{petcu2013interface}
M.~Petcu and R.~Temam.
\newblock An interface problem: the two-layer shallow water equations.
\newblock {\em DCDS-A}, 6(2):401--422, 2013.

\bibitem{roed1987study}
L.P. R{\o}ed and C.K. Cooper.
\newblock A study of various open boundary conditions for wind-forced
  barotropic numerical ocean models.
\newblock {\em Elsevier oceanography series}, 45:305--335, 1987.

\bibitem{schijf1953theoretical}
J.B. Schijf and J.C. Schonfled.
\newblock Theoretical considerations on the motion of salt and fresh water.
\newblock IAHR, 1953.

\bibitem{Serre1996systemes}
D.~Serre.
\newblock {\em Syst{\`e}mes de lois de conservation}.
\newblock Diderot Paris, 1996.

\bibitem{smoller1983shock}
J.~Smoller.
\newblock Shock waves and reaction-diffusion equations, vol. 258 of fundamental
  principles of mathematical science, 1983.

\bibitem{sommerfeld1949partial}
A.~Sommerfeld.
\newblock Partial differential equation in physics.
\newblock {\em Lectures on Theoretical Physics-Pure and Applied Mathematics,
  New York: Academic Press, 1949}, 1, 1949.

\bibitem{stewart2012multilayer}
A.L. Stewart and P.J. Dellar.
\newblock Multilayer shallow water equations with complete coriolis force. part
  3. hyperbolicity and stability under shear.
\newblock {\em Journal of Fluid Mechanics}, 723:289--317, 2013.

\bibitem{taylor1996partial}
M.E. Taylor.
\newblock {\em Partial differential equations III: Nonlinear equations}.
\newblock Springer, 1996.

\bibitem{toro2009riemann}
E.F. Toro.
\newblock {\em Riemann solvers and numerical methods for fluid dynamics: a
  practical introduction}.
\newblock Springer, 2009.

\bibitem{vreugdenhil1979two}
C.B. Vreugdenhil.
\newblock Two-layer shallow-water flow in two dimensions, a numerical study.
\newblock {\em Journal of Computational Physics}, 33(2):169--184, 1979.

\end{thebibliography}

\end{document}